\documentclass[a4paper,10pt]{amsart}
\usepackage{amsmath,amsthm,amssymb,enumerate}
\usepackage{mathptmx}
\usepackage{epsfig}
\usepackage{amssymb}
\usepackage{amsmath}
\usepackage{amssymb}
\usepackage{amsmath,amsthm}
\usepackage[latin1]{inputenc}
\usepackage{bm}
\usepackage{bbm}
\usepackage{esint}
\usepackage{amsfonts}
\usepackage{amsxtra}
\usepackage{euscript,mathrsfs}
\usepackage{color}
\usepackage[left=3.8cm,right=3.8cm,top=4cm,bottom=3cm]{geometry}
\usepackage[citecolor=blue,colorlinks=true]{hyperref}
\allowdisplaybreaks
\normalsize

\usepackage{tikz-cd}
\usetikzlibrary{decorations.pathreplacing}

\usepackage{enumitem}
\setenumerate{label={\rm (\alph{*})}}

\usepackage{amsfonts}
\usepackage{amsxtra}

\allowdisplaybreaks

\numberwithin{equation}{section}

\newtheorem{theorem}{Theorem}[section]
\newtheorem{lemma}[theorem]{Lemma}
\newtheorem{proposition}[theorem]{Proposition}

\newtheorem{corollary}[theorem]{Corollary}

\theoremstyle{definition}

\newtheorem{definition}[theorem]{Definition}

\theoremstyle{remark}

\newtheorem{remark}[theorem]{Remark}

\newcommand{\R}{\mathbb R}

\newcommand{\E}{\operatorname{E}\!}

\newcommand{\dif}{\mathrm{d}}

\DeclareMathOperator{\dist}{dist}

\renewcommand{\dif}{\operatorname{d}\!}
\newcommand{\lebe}{\operatorname{L}}
\newcommand{\sobo}{\operatorname{W}}

\newcommand{\locc}{\operatorname{loc}}

\newcommand{\hold}{\operatorname{C}}

\newcommand{\sg}{\varepsilon}

\newcommand{\bv}{\operatorname{BV}}

\newcommand{\ball}{\operatorname{B}}
\newcommand{\di}{\operatorname{div}}
\newcommand{\bd}{\operatorname{BD}}
\newcommand{\A}{\mathbb{A}}
\newcommand{\rsym}{\mathbb{R}_{\operatorname{sym}}^{n\times n}}
\newcommand{\ld}{\operatorname{LD}}
\newcommand{\besov}{\operatorname{B}}
\newcommand{\mres}{\mathbin{\vrule height 1.6ex depth 0pt width
0.13ex\vrule height 0.13ex depth 0pt width 1.3ex}}

\newcommand{\D}{\operatorname{D}}

\newcommand{\dashint}{\fint}
\newcommand{\wstar}{\stackrel{*}{\rightharpoonup}}

\newcommand{\trace}{\operatorname{Tr}}
\newcommand{\poinc}{\operatorname{Poinc}}

\newcommand{\sym}{\operatorname{sym}}

\allowdisplaybreaks

\begin{document}

%\begin{frontmatter}

\title[Partial Regularity for Symmetric Quasiconvex Functionals on $\bd$]{Partial Regularity for \\ Symmetric Quasiconvex Functionals on BD}

\author[F.~Gmeineder]{Franz Gmeineder}

\address[]{Mathematical Institute, University of Bonn, Endenicher Allee 60, 53115 Bonn, Germany}
\email{fgmeined@math.uni-bonn.de}

\begin{abstract}
We establish the first partial regularity results for (strongly) symmetric quasiconvex functionals of linear growth on $\bd$, the space of functions of bounded deformation. 
By \textsc{Rindler}'s foundational work \cite{Ri1}, symmetric quasiconvexity is the pivotal notion as regards sequential weak*-lower semicontinuity and hence for the existence of minima of the relaxed functionals on $\bd$. The overarching main difficulty here is the lack of \textsc{Korn}'s Inequality in the $\lebe^{1}$-setting, hereby implying that the $\bd$-case is genuinely different from the study of variational integrals on $\bv$. Unlike for superlinear growth, symmetric quasiconvex functionals, where we establish partial regularity by direct reduction to the full gradient case by \textsc{Korn}-type inequalities, such a reduction does not work in the linear growth case and identifies the latter as the only situation requiring a treatment on its own. \\
\vspace{0.05cm}

\noindent
\textsc{R\'{e}sum\'{e}.} Nous \'{e}tablissons les premiers r\'{e}sultats de r\'{e}gularit\'{e}
partielles pour des fonctionnelles quasi-convexes (fortement) sym\'{e}triques ayant une croissance lin\'{e}aire sur BD, l'espace des fonctions dont la
d\'{e}formation est born\'{e}e. De part les travaux pr\'{e}curseurs de \textsc{Rindler} \cite{Ri1}, la quasi-convexit\'{e} sym\'{e}trique est la notion centrale relativement {\`{a}} la semi-continuit\'{e}
inf\'{e}rieure faible-* et donc pour l'existence des minima de fonctionnelles relax\'{e}es sur BD. La difficult\'{e} g\'{e}n\'{e}rale ici est due {\`{a}} l'absence de l'in\'{e}galit\'{e} de Korn dans le cadre fonctionel $\lebe^{1}$, qui de ce fait r\'{e}v{\`{e}}le
une diff\'{e}rence fondamentale entre le cas BD et celui relatif {\`{a}} l'\'{e}tude des
int\'{e}grales variationnelles sur BV. Contrairement aux fonctionnelles
quasi-convexes sym\'{e}triques {\`{a}} croissance superlin\'{e}aire pour lesquelles nous
\'{e}tablissons la r\'{e}gularit\'{e} partielle via r\'{e}duction directe au cas du
gradient complet gr\^{a}ce aux in\'{e}galit\'{e}s de type Korn, une telle r\'{e}duction ne
peut \^{e}tre impl\'{e}ment\'{e}e dans le cas de la croissance lin\'{e}aire et identifie
ce dernier comme \'{e}tant la seule situation qui n\'{e}c\'{e}ssite un traitement
particulier.
\end{abstract}

\date{\today}

\maketitle

\setcounter{tocdepth}{1}
%\end{frontmatter}
\tableofcontents

\section{Introduction}
\subsection{Aims and scope}
Let $n\geq 2$ and $\Omega$ be an open and bounded subset of $\R^{n}$ with Lipschitz boundary. A vast class of variational problems connected to plasticity is set up by virtue of linear growth functionals depending on the symmetric gradient, cf. \cite{AG,FuchsSeregin,Suquet,CMS}. Possibly allowing for non-convex energies, a unifying perspective on the topic as considered in variants in \cite{BFT,Ri1} is given by the canonical variational principle
\begin{align}\label{eq:varprin}
\text{to minimise}\;\;\;F[v]:=\int_{\Omega}f(\sg(v))\dif x\;\;\;\text{over a Dirichlet class}\;\mathcal{D}_{u_{0}}, 
\end{align}
where $u_{0}\colon\Omega\to\R^{n}$ is a suitable Dirichlet datum and $\sg(v):=\frac{1}{2}(\D v + \D v^{\top})$ denotes the \emph{symmetric gradient} of a map $v\colon\R^{n}\to\R^{n}$. Most crucially, $f\colon\rsym\to\R$ is assumed to be a continuous integrand of \emph{linear growth}. By this we understand that there exists a constant $L>0$ such that 
\begin{align}\label{eq:lingrowth}\tag{LG}
|f(z)|\leq L(1+|z|)\qquad\text{for all}\;z\in\rsym. 
\end{align}
Following the foundational work of \textsc{Rindler} \cite{Ri1}, a necessary and sufficient condition for the associated relaxed functionals to be suitably lower semicontinuous is that of \emph{symmetric quasiconvexity}, cf. Section~\ref{sec:convexity} below. In view of the direct method of the calculus of variations, symmetric quasiconvexity thus plays the central r\^{o}le for functionals of the form \eqref{eq:varprin}. Yet, for such symmetric quasiconvex functionals the properties of minima are far from being understood -- in particular, a regularity theory is still missing. This equally applies to the situation where $f\colon\rsym\to\R$ is of $p$-growth ($|f(z)|\leq L(1+|z|^{p})$ for all $z\in\rsym$) with $p>1$, thereby connecting to models from nonlinear elasticity or fluid mechanics \cite{CFI,FuchsSeregin}. Hence the objective of this paper is to make a first step in this direction and close this gap. As we shall see (cf. Section~\ref{sec:mainresult}), it is only the linear growth case which requires a separate theory; the $p$-growth case with $p>1$ can be fully reduced to the regularity theory for \emph{full gradient functionals}. 

To elaborate more on these matters, we start by noting that the growth bound \eqref{eq:lingrowth} suggests to consider \eqref{eq:varprin} on Dirichlet classes $\sobo_{u_{0}}^{1,1}(\Omega;\R^{n})=:u_{0}+\sobo_{0}^{1,1}(\Omega;\R^{n})$ for $u_{0}\in\sobo^{1,1}(\Omega;\R^{n})$. However, by \textsc{Ornstein}'s Non-Inequality \cite{Ornstein}, it is not possible to uniformly bound the $\lebe^{1}$-norm of $\D\!u$ against that of $\sg(u)$. In fact, for every $n\geq 2$ there exists a sequence $(\varphi_{j})\subset\hold_{c}^{\infty}(\ball(0,1);\R^{n})$ for which $(\sg(\varphi_{j}))$ remains bounded in $\lebe^{1}(\Omega;\rsym)$ whereas $\|D\varphi_{j}\|_{\lebe^{1}(\Omega;\R^{n\times n})}\to\infty$ as $j\to\infty$, cf. \textsc{Conti} et al. \cite{CFM0} and \textsc{Kirchheim \& Kristensen} \cite{KirchheimKristensen} for recent approaches. 
This is in stark contrast with the situation when $\lebe^{1}$ is replaced by $\lebe^{p}$, $1<p<\infty$. Indeed, in the latter case the corresponding result can be reduced to \textsc{Korn}-type inequalities, cf.~\cite{Friedrichs,Gobert,Mosolov,Necas} for classical material and \textsc{Ciarlet} et al. \cite{Ciarlet1,Ciarlet2,Ciarlet3} and the references therein for a selection of  nonlinear variants. In consequence, $F$ given by \eqref{eq:varprin} is not coercive on $\sobo^{1,1}(\Omega;\R^{n})$ but on 
\begin{align*}
\ld(\Omega):=\big\{u\in\lebe^{1}(\Omega;\R^{n})\colon\;\sg(u)\in\lebe^{1}(\Omega;\rsym) \big\}
\end{align*}
endowed with the $\ld$-norm $\|u\|_{\ld}:=\|u\|_{\lebe^{1}}+\|\sg(u)\|_{\lebe^{1}}$. It is thus natural to let the Dirichlet datum $u_{0}$ belong to $\ld(\Omega)$ and consider the variational principle \eqref{eq:varprin} over the affine class $\mathscr{D}_{u_{0}}:=\ld_{u_{0}}(\Omega):=u_{0}+\ld_{0}(\Omega)$, where $\ld_{0}(\Omega)$ is the closure of $\hold_{c}^{1}(\Omega;\R^{n})$ for the $\ld$-norm. Still, $\ld$ is an $\lebe^{1}$-based space and hence fails to be reflexive; as a consequence, it lacks an appropriate version of the Banach-Alaoglu theorem concerning weak convergence. Thus it is required to relax $F$ to the space $\bd(\Omega)$ given by 
\begin{align*}
\bd(\Omega):= \Big\{u\in\lebe^{1}(\Omega;\R^{n})\colon\;\E u\in\mathcal{M}(\Omega;\rsym) \Big\}. 
\end{align*}
Here, $\mathcal{M}(\Omega;\rsym)$ are the finite, $\rsym$-valued Radon measures on $\Omega$, and we use the notation $\E u$ instead of $\sg(u)$ to indicate that $\E u$ is a measure. The relaxation here is taken with respect to weak*-convergence in $\bd(\Omega)$, and we refer the reader to Sections~\ref{sec:convexity} and \ref{sec:bd} for the requisite background terminology. This space -- which contains $\bv(\Omega;\R^{n})$ as a proper subspace -- takes a prominent role in plasticity, and has been studied from various perspectives by a notable plenty of authors, see \cite{Suquet,CMS,Kohn1,Kohn2,ACD,AG0,ST,Baba} among others. Before embarking on the regularity issue raised above in detail, we briefly pause and discuss the relevant relaxed functionals $\overline{F}$ that are required for defining the notion of (local) minimality in the sequel.

\subsection{Symmetric quasiconvexity and relaxation}\label{sec:convexity}
From a calculus of variations and hereafter lower semicontinuity perspective, the central notion for functionals $F$ is a variant of \textsc{Morrey}'s \emph{quasiconvexity} \cite{Morrey}, namely the aforementioned \emph{symmetric quasiconvexity}. Already appearing in variants in  \cite{Dacorogna0,FonsecaMueller2}, we start by recalling the following definition as given, e.g., in \cite{BFT,Ri1}:
\begin{definition}[Symmetric quasiconvexity]
A continuous integrand $f\colon\rsym\to\R$ is said to be \emph{symmetric quasiconvex} provided there holds
\begin{align*}
f(z)\leq \int_{Q}f(z+\sg(\varphi))\dif x\qquad\text{for all}\;\varphi\in\hold_{0}^{1}(Q;\R^{n})\;\text{and}\;z\in\rsym, 
\end{align*}
where $Q=(0,1)^{n}$ is the open unit cube in $\R^{n}$.
\end{definition}
Letting $u\in\bd_{\locc}(\Omega)$, we denote $\E u = \E^{a}u+\E^{s}u=\mathscr{E}u\mathscr{L}^{n} + \frac{\dif\E^{s}u}{\dif |\E^{s}u|}|\E^{s}u|$ the Lebesgue-Radon-Nikod\'{y}m decomposition of $\E u$; cf. Section~\ref{sec:bd} for this and the subsequent terminology. Returning to the functional $F$ defined in terms of $f\colon\rsym\to\R$ by \eqref{eq:varprin} subject to \eqref{eq:lingrowth}, let $v\in\bd_{\locc}(\Omega)$ be given. For a subset $\omega\subseteq\Omega$ with Lipschitz boundary $\partial\omega$, we define the \emph{relaxed functional} by 
\begin{align}\label{eq:integralrepresentation}
\begin{split}
\overline{F}_{v}[u;\omega]  = \int_{\omega}f(\mathscr{E}u)\dif x & +\int_{\omega}f^{\infty}\Big(\frac{\dif \E u}{\dif |\!\E u|}\Big)\dif |\E^{s}u| \\ & + \int_{\partial\omega}f^{\infty}(\trace_{\partial\omega}(v-u)\odot\nu_{\partial\omega})\dif\mathcal{H}^{n-1},\qquad u\in \bd(\omega). 
\end{split}
\end{align}
Here, $f^{\infty}(z):=\limsup_{t\searrow 0}tf(\tfrac{z}{t})$ denotes the recession function of $f$ at $z\in\rsym$, capturing the integrand's behaviour at infinity. Also, $\trace_{\partial\omega}$ displays the boundary trace operator on $\bd(\omega)$ and $\nu_{\partial\omega}$ the outer unit normal to $\partial\omega$.

With this notation, we say that $u\in\bd_{\locc}(\Omega)$ is a \emph{local $\bd$-minimiser} (or \emph{local generalised minimiser}) for $F$ provided 
\begin{align}\label{eq:minimality}
\overline{F}_{u}[u;\omega]\leq \overline{F}_{u}[v;\omega]\qquad\text{for all}\;v\in\bd(\omega)
\end{align}
holds for all subsets $\omega\Subset\Omega$ with Lipschitz boundary $\partial\omega$. If $u_{0}\in\ld(\Omega)$ is a given Dirichlet datum, we put $\overline{F}_{u_{0}}[u]:=\overline{F}_{u_{0}}[u;\Omega]$. In this situation, we call $u\in\bd(\Omega)$ a \emph{$\bd$-minimiser} (or \emph{generalised minimiser}) for $F$ subject to $u_{0}$ provided 
\begin{align}\label{eq:minimality1}
\overline{F}_{u_{0}}[u]\leq \overline{F}_{u_{0}}[v]\qquad\text{for all}\;v\in\bd(\Omega).
\end{align}
Then, any $\bd$-minimiser subject to $u_{0}$ is a local $\bd$-minimiser. Essentially solely subject to the additional linear growth assumption \eqref{eq:lingrowth}, \textsc{Rindler} \cite{Ri1} identified the symmetric quasiconvexity of $f$ as a necessary and sufficient condition for $\overline{F}_{u_{0}}$ to be sequentially weak*-lower semicontinuous on $\bd(\Omega)$. Most notably, not only the extending the classical work of \textsc{Ambrosio \& Dal Maso} \cite{AmbrosioDalMaso} as well as partly that of \textsc{Fonseca \& M\"{u}ller} \cite{FonsecaMueller1} from the $\bv$- to the $\bd$-situation, the relevant lower semicontinuity was established in \cite{Ri1} without relying on the $\bd$-variant of \textsc{Alberti}'s rank-one theorem \cite{Alberti}. By now, the latter has been proved by \textsc{De Philippis \& Rindler} in the seminal work \cite{DPR} in a much more general context, allowing for a simplified proof of \eqref{eq:integralrepresentation} (cf. \cite{DPR1,ARDPR,BDG}) but had not been available at the time of \cite{Ri1}.

In consequence, augmenting the linear growth assumption \eqref{eq:lingrowth} with a suitable coerciveness condition on the symmetric quasiconvex integrand $f$, existence of $\bd$-minima for $F$ follows at ease. Equally, we have the \emph{no-gap-result} $\inf_{u_{0}+\ld_{0}(\Omega)}F=\min_{\bd(\Omega)}\overline{F}_{u_{0}}$, see Section~\ref{sec:existence}. As will be discussed below in Section~\ref{sec:mainresult}, such a coerciveness criterion goes hand in hand with the partial regularity of $\bd$-minima. It is thus natural to contextualise the partial regularity for $\bd$-minima with available results in the literature and thereby outline the main obstructions first. 
\subsection{Contextualisation and overview}\label{sec:context}
In the common language of regularity theory, the minimisation of functionals \eqref{eq:varprin} displays a \emph{purely vectorial problem}. Even in the case where the symmetric gradient in \eqref{eq:varprin} is replaced by the full gradient, it is a well-known feature of such multiple integrals to produce minima which are not everywhere $\hold^{1,\alpha}$-H\"{o}lder continuous but only on a large set; cf. \cite{Beck,Giusti,Mingione1} for overviews. This is referred to as \emph{partial regularity}. 

In the superlinear growth regime with full gradients, the study of partial regularity for minima has a long and rich history, starting with the seminal work of \textsc{Evans} \cite{Ev} and \textsc{Acerbi \& Fusco} \cite{AcerbiFusco1}; also see \textsc{Mingione} et al.  \cite{CFM,KuusiMingione1,Mingione1} and  \cite{Beck,DLSV,Du1,Du2,Du3,Giusti} for a non-exhaustive list of other contributions. However, until recently, for full gradient linear growth functionals the only contribution had been the local-in-phase-space result due to \textsc{Anzellotti \& Giaquinta} \cite{AG} and its adaptation to the model integrands $z\mapsto (1+|z|^{p})^{1/p}$, $p\neq 2$, by \textsc{Schmidt} \cite{Schmidt1}. This approach, which crucially relies on comparison with mollifications and thus works well for convex integrands by Jensen's inequality, has been extended to the $\bd$-setting by the author \cite{G1}. Yet, due to the very method of proof, it seems to be restricted to convex integrands and a generalisation of the strategy to the quasiconvex case seems difficult; also see the discussion in \cite{AG,Schmidt1} and \cite{G1}. 

At present, in the full gradient, quasiconvex linear growth case on $\bv$, the only partial regularity result up to date has been given recently by \textsc{Kristensen} and the author \cite{GK2}. In this work, a direct comparison with suitably $\A$-harmonic maps is implemented that overcomes any indirect argument as is found e.g. in the blow-up method or, quite implicitely, in the proof of the $\A$-harmonic approximation lemma due to \textsc{Duzaar} et al. \cite{Du1,Du2,Du3}. Let us note that similarly to \cite{AG,Schmidt1,G1}, the sole use of direct arguments is somehow dictated here by the comparatively weak compactness properties of $\bv$ and $\bd$. In fact, examplarily pursuing the blow-up method for linear growth functionals, it is necessary to establish that the weak*-limit of a blow-up sequence satisfies a strongly elliptic Legendre-Hadamard system. However, by possible concentration
effects, this conclusion seems unreachable since there are no general compactness improvements for the relevant blow-up sequences: Such compactness boosts would require some uniform local integrability enhancements, usually provided by \textsc{Gehring}'s lemma in reliance on Caccioppoli-type inequalites, or higher (fractional) differentiability estimates a l\'{a} \textsc{Mingione} \cite{Mingione00,Mingione0}. Whereas the former cannot be implemented in the linear growth situation -- essentially due to the non-availability of a sublinear Sobolev-Poincar\'{e}-type inequality, cf. \textsc{Buckley \& Koskela} \cite{BuckleyKoskela} and the discussion in Section~\ref{sec:extensions} --, higher fractional differentiability results on minima such as in \cite{Mingione00,Mingione0} are confined to the convex situation. Similar issues already arise in the full gradient situation, equally for other methods such as the classical $\A$-harmonic approximation, and we refer the reader to \cite{GK2} for a further discussion thereof.
\subsection{Main Results}\label{sec:mainresult}
After these preparations, we now pass to the description of the main results of the present paper. To begin with, symmetric quasiconvexity and the linear growth hypothesis \eqref{eq:lingrowth} together are easily seen not to be sufficient for $F$ given by  \eqref{eq:varprin} to produce bounded minimising sequences in $\ld(\Omega)$. To ensure the latter, we require a strong version of symmetric quasiconvexity. The same issue arises in the superlinear growth regime as well, and so we begin with the treatment of such functionals, in turn being linked to elasticity type problems. Our first result is that -- completely different from the linear growth situation -- \emph{in the superlinear growth case, partial regularity directly can be fully reduced to the corresponding full gradient theory}.

Given $1\leq p <\infty$, we put $V_{p}(z):=(1+|z|^{2})^{\frac{p}{2}}-1$ for $z\in\rsym$ and put $V(z)=V_{1}(z)=\sqrt{1+|z|^{2}}-1$. We say that $g\in\hold(\rsym)$ is of $p$-\emph{growth} provided there exists $L_{p}>0$ such that 
\begin{align}\label{eq:pgrowthintro}
|g(z)|\leq L_{p}(1+|z|^{p})\qquad\text{for all}\;z\in\rsym, 
\end{align}
and accordingly call $g$ \emph{$p$-strongly symmetric quasiconvex} provided there exists $\ell_{p}>0$ such that 
\begin{align}\label{eq:pSSQC}
\rsym\ni z\mapsto g(z)-\ell_{p} V_{p}(z)\in \R\qquad\text{is symmetric quasiconvex}. 
\end{align}
\begin{theorem}\label{thm:ppartialregularity}
Let $\Omega\subset\R^{n}$ be open and bounded. Let $1<p<\infty$ and suppose that $g\in\hold^{2}(\rsym)$ is an integrand which 
\begin{enumerate}
\item is of $p$-growth, i.e., satisfies \eqref{eq:pgrowthintro} for all $z\in\rsym$ and 
\item is $p$-strongly symmetric quasiconvex in the sense of \eqref{eq:pSSQC}. 
\end{enumerate}
Then for any local minimiser $u\in\sobo_{\locc}^{1,p}(\Omega;\R^{n})$ of the corresponding integral functional 
\begin{align}\label{eq:pintfunctional}
v\mapsto \int g(\sg(v))\dif x
\end{align}
there exists an open subset $\Omega_{u}$ of $\Omega$ such that $\mathscr{L}^{n}(\Omega\setminus\Omega_{u})=0$ and $u$ is of class $\hold^{1,\alpha}$ for each $0<\alpha<1$ in a neighbourhood of any of the elements of $\Omega_{u}$. 
\end{theorem}
Theorem~\ref{thm:ppartialregularity} displays a sample theorem; similar results can be inferred for variational integrals with integrands $f(x,u,\sg(u))$. In particular, the partial $\hold^{1,\alpha}$-regularity of minima of \emph{convex} elasticity-type functionals (for some elliptic $\mathbb{C}\in \mathscr{L}(\rsym,\rsym)$, $\mu>0$ and $g\in\lebe^{p}(\Omega;\R^{n})$) 
\begin{align}
v\mapsto \int_{\Omega}\frac{1}{p}\big(\big(\langle\mathbb{C}\sg(v),\sg(v)\rangle + \mu \big)^{\frac{p}{2}}-\mu^{\frac{p}{2}}\big)\dif x + \int_{\Omega}|v-g|^{p}\dif x, 
\end{align}
as recently considered in \cite{CFI} can similarly be covered and generalised by the methods underlying Theorem~\ref{thm:ppartialregularity}, even for any elliptic operator $A[D]$ in the sense of Section~\ref{sec:bd}, and we shall pursue this elsewhere. The reason for the reducibility of Theorem~\ref{thm:ppartialregularity} to the full gradient case is that the $p$-strong symmetric quasiconvexity \eqref{eq:pSSQC} expresses a coerciveness property of the associated variational integrals on $\sobo^{1,p}$. For $1<p<\infty$ and subject to \eqref{eq:pSSQC}, minimising sequences can be proven to remain bounded in $\sobo^{1,p}(\Omega;\R^{n})$ by \textsc{Korn}-type inequalities. The $p$-strong symmetric quasiconvexity, being an integral rather than a pointwise condition, gives us direct access to the requisite \textsc{Korn}-type inequalities; see Section~\ref{sec:SSC} for more detail. 

If $p=1$, \textsc{Ornstein}'s Non-Inequality \cite{Ornstein,CFM0,KirchheimKristensen} does not allow to employ a similar reduction scheme. More systematically, if $p=1$ and \eqref{eq:lingrowth} are in action, the ($1$-)strong symmetric quasiconvexity still expresses a coerciveness property for the associated variational integrals on $\ld$ (see Lemma~\ref{lem:minseq}). To streamline terminology, we simply say that $f\in\hold(\rsym)$ is \emph{strongly symmetric quasiconvex} provided there exists $\ell>0$ such that the function
\begin{align}\label{eq:SSQC}
\rsym \ni z \mapsto f(z)-\ell V(z)\qquad\text{is symmetric quasiconvex},
\end{align}
where $V:=V_{1}$ shall be referred to as the \emph{reduced area integrand}; the ubiquity of such functions follows along the lines of \cite[Prop. 2.14]{GK2} (also see \textsc{Zhang} \cite{Zhang2} for a related construction of quasiconvex, linear growth integrands). In conjunction with \eqref{eq:lingrowth}, this condition yields the existence of $\bd$-minima for the associated variational integrals subject to given Dirichlet data (cf.~Section~\ref{sec:app1}). Moreover, appealing to Lemma~\ref{lem:minseq}, it is closely related to \emph{all} minimising sequences being bounded in $\ld(\Omega)$. Since the $1$-strong quasiconvexity for full gradient functionals, in turn, yields boundedness of the respective minimising sequences in $\sobo^{1,1}(\Omega;\R^{n})$ (cf.~\cite{GK2}), a reduction to the full gradient is systematically excluded. Within the realm of $p$-growth, symmetric quasiconvex functionals, this hereafter identifies the limiting case $p=1$ as the only one requiring a treatment on its own. 

In this respect, the main result of this paper reads as follows: 
\begin{theorem}[Partial regularity of $\bd$-minimisers]\label{thm:main1}
Let $\Omega\subset\R^{n}$ be open and bounded.  Moreover, suppose that $f\colon\R_{\sym}^{n\times n}\to\R$ is a variational integrand that is
\begin{enumerate}
\item\label{item:reg1} of class $\hold_{\locc}^{2,1}(\rsym)$, 
\item\label{item:reg2} of linear growth in the sense of \eqref{eq:lingrowth} and
\item\label{item:reg3} strongly symmetric quasiconvex in the sense of \eqref{eq:SSQC}.
\end{enumerate}
Then for each $M>0$ there exist $\varepsilon_{M}=\varepsilon_{M}(n,\ell,L,M)>0$ and a radius $R_{0}=R_{0}(n,\ell,L,M)>0$ such that for every local $\bd$-minimiser $u\in\bd_{\locc}(\Omega)$ of $F$ (in the sense of \eqref{eq:minimality}) the following holds: If $x_{0}\in\Omega$ and $0<R<R_{0}$ with $\ball(x_{0},R)\Subset\Omega$ are such that   
\begin{align}\label{eq:thm:smallness}
\left\vert\frac{\E u(\ball(x_{0},R))}{\mathscr{L}^{n}(\ball(x_{0},R))}\right\vert < M\;\;\text{and}\;\;
\dashint_{\ball(x_{0},R)}\left\vert\mathscr{E}u-\frac{\E u(\ball(x_{0},R)}{\mathscr{L}^{n}(\ball(x_{0},R))}\right\vert\dif x + \frac{|\E^{s}u|(\ball(x_{0},R))}{\mathscr{L}^{n}(\ball(x_{0},R))} < \varepsilon_{M}, 
\end{align}
then $u$ is of class $\hold^{1,\alpha}$ on $\ball(x_{0},\tfrac{R}{2})$ for any $0<\alpha<1$. As a consequence, there exists an open subset $\Omega_{u}$ of $\Omega$ with $\mathscr{L}^{n}(\Omega\setminus\Omega_{u})=0$ such that for every $x_{0}\in\Omega_{u}$, $u$ is of class $\hold^{1,\alpha}$ in a neighbourhood of $x_{0}$ for any $0<\alpha<1$. In particular, denoting $\Sigma_{u}=\Omega\setminus\Omega_{u}$, we have 
\begin{align}
\begin{split}
\Sigma_{u} & = \Sigma_{u}^{1} \cup \Sigma_{u}^{2} \\ & :=\Big\{x_{0}\in\Omega\colon\; \liminf_{R\searrow 0}\Big(\dashint_{\ball(x_{0},R)}|\mathscr{E}u-(\mathscr{E}u)_{\ball(x_{0},R)}|\dif x +\frac{|\E^{s}u|(\ball(x_{0},R))}{\mathscr{L}^{n}(\ball(x_{0},R))}\Big)>0 \Big\}\\
& \;\;\;\;\;\;\;\;\;\;\;\;\;\;\;\;\;\;\;\;\;\;\;\;\;\;\;\;\;\;\;\;\;\;\;\;\;\;\;\;\cup \Big\{x_{0}\in\Omega\colon\;\limsup_{R\searrow 0}\left\vert\frac{\E u(\ball(x_{0},R))}{\mathscr{L}^{n}(\ball(x_{0},R))}\right\vert = \infty\Big\}.
\end{split}
\end{align}
\end{theorem}
Not being allowed to utilise full gradient techniques, Theorem~\ref{thm:main1} cannot be established in an analogous manner as its full gradient counterpart  from \cite{GK2}. Partially based on recently available results \cite{Baba,BDG,DPR,VS}, the proof of the previous theorem is given in Section~\ref{sec:main1}. Here we rely in an essential way on an improved estimate of the local $\bd$-minimisers' distances from suitable $\A$-harmonic approximations in terms of a \emph{superlinear} power of the excess. To the best of our knowledge, an estimate of this form has only been derived recently in the $\bv$-full gradient setup in \cite{GK2}, strongly relying on the full distributional gradients of $\bv$-minima being Radon measures of finite total variation. The aforementioned superlinear excess power, in turn, stems from the higher regularity properties of the $\A$-harmonic approximants on \emph{good} balls. To define the latter notion appropriately, we remark that the $\A$-harmonic approximants on \emph{generic balls} receive their higher Sobolev regularity up to the boundary from the higher regularity of their prescribed Dirichlet data; the precise correspondence is displayed in Proposition~\ref{thm:mazsha}. For arbitrary balls $\ball\Subset\Omega$ and $u\in\bd(\Omega)$, we can only assert that $\trace_{\partial\!\ball}(u)\in\lebe^{1}(\partial\!\ball;\R^{n})$. This motivates the Fubini-type Theorem~\ref{thm:fubini1}, implying that on sufficiently many spheres, $\bd$-maps have interior traces with some additional differentiability and integrability beyond $\lebe^{1}$. We wish to stress that by Ornstein's Non-Inequality this step does \emph{not} follow as for $\bv$, where the tangential traces $\partial_{\tau}u$ can be shown to belong to $\lebe^{1}(\partial\!\ball)$ on sufficiently many balls $\ball$ (see Remark~\ref{rem:VersusBV}). Here we crucially use the embedding $\bd\hookrightarrow\sobo^{s,\frac{n}{n-1+s}}$ for $n\geq 2$, $0<s<1$ together with novel Poincar\'{e}-type inequalities to be proved in Section~\ref{sec:prelims}. Up from here, it is then the overall aim of the proof to show that Ornstein's Non-Inequality essentially becomes \emph{invisible} throughout the comparison estimates, simultaneously keeping track of the enlarged nullspace of the symmetric gradient in comparison with that of the full gradient.  This comes along with both further conceptual and technical difficulties, and Section~\ref{sec:main1} is devoted to their precise discussion and resolution. Finally, let us mention that the approach as developed here should also give a streamlining and unifying treatment for the $\bv$-case in the dimensions $n=2$ and $n\geq 3$; cf.~Remark~\ref{rem:unifying}.

Lastly, let us comment on the hypotheses and extensions of Theorem~\ref{thm:main1}. Condition~\ref{item:reg1} is rather of technical than instrumental nature and can be relaxed (cf. \cite{GK2} for a related discussion); as our focus is on the symmetric quasiconvexity condition rather than regularity of the integrands, we stick to this assumption for simplicity. Let us note, however, that subject to \ref{item:reg1}--\ref{item:reg3} from above, it is moreover not too difficult to show that $\bd$-minima are actually \emph{$\hold^{2,\alpha}$-partially regular once the $\hold^{1,\alpha}$-regularity of Theorem~\ref{thm:main1} is established}. The methods underlying the proof of Theorem~\ref{thm:main1} also apply to suitable $x$-dependent integrands, whereas the case of fully non-autonomous integrands would require an additional argument. On the other hand, Theorem~\ref{thm:main1} exclusively establishes the partial regularity, but does not provide Hausdorff dimension bounds for the singular set $\Sigma_{u}$. Note that, by the strong symmetric quasiconvexity, such estimates require a refined argument; see Section~\ref{sec:extensions} for a discussion of these matters. 

\subsection{Structure of the Paper}
In Section~\ref{sec:prelims}, we fix notation, prove and collect miscellaneous background results. In Section~\ref{sec:SSC}, we deal with the superlinear growth situation and establish Theorem~\ref{thm:main1}. Section~\ref{sec:fubini} then serves to prove a Fubini-type theorem for $\bd$ that is instrumental in the proof of Theorem~\ref{thm:main1}, and Section~\ref{sec:main1} is devoted to the proof of the latter. We conclude with an appendix in Section~\ref{sec:appendix}.

{\small 
\subsection*{Acknowledgments}
I wish to thank \textsc{Jan Kristensen, Gregory Seregin}, \textsc{Gianni Dal Maso} and \textsc{Bernd Schmidt} for making valuable suggestions on the theme of the paper, motivating, among others, the regularity results in Section~\ref{sec:SSC}. I am moreover indebted to the anonymous referees, whose careful reading and valuable suggestions led to several improvements. This work has received funding from the EPSRC Research Council and the Hausdorff Center for Mathematics Bonn, for which I am equally grateful.}
\section{Preliminaries}\label{sec:prelims}
\subsection{General Notation}
We now briefly gather notation. Unless otherwise stated, $\Omega$ always denotes an open and bounded Lipschitz subset of $\R^{n}$. We denote $\ball(x_{0},r):=\{x\in\R^{n}\colon\;|x-x_{0}|<r\}$ the open ball of radius $r>0$ centered at $x_{0}$ and use the symbol $\mathbb{B}_{\sym}^{n\times n}$ to denote the \emph{closed unit ball} in $\R_{\sym}^{n\times n}$ with respect to the Frobenius norm $|A|:=(\sum_{i,j=1}^{n}|a_{ij}|^{2})^{1/2}$, $A=(a_{ij})_{i,j=1}^{n}\in\R^{n\times n}$. Whenever $X$ is a finite dimensional real vector space, the symbol $\langle\cdot,\cdot\rangle$ is used to denote the usual inner product on $X$ and $\mathbb{S}(X)$ is the space of symmetric bilinear forms on $X$. To avoid ambiguities, note that duality pairings are exclusively used with subscripts, so e.g. $\langle\cdot,\cdot\rangle_{\mathscr{D}'\times\mathscr{D}}$ for the pairing between distributions and test functions. Also, for two given vectors $a,b\in\R^{n}$, $a\odot b := \frac{1}{2}(ab^{\mathsf{T}}+ba^{\mathsf{T}})$ denotes their symmetric tensor product, and we record that 
\begin{align}\label{eq:symmetrictensorproduct}
\frac{1}{\sqrt{2}}|a|\,|b|\leq |a\odot b| \leq |a|\,|b|\qquad\text{for all}\;a,b\in\R^{n}.
\end{align}
The symbol $\mathscr{L}(V;W)$ denotes the bounded linear operators between two normed linear spaces $V$ and $W$. As usual, $\mathscr{L}^{n}$ and $\mathscr{H}^{n-1}$ denote the $n$-dimensional Lebesgue and the $(n-1)$-dimensional Hausdorff measure, respectively, and we put $\omega_{n}:=\mathscr{L}^{n}(\ball(0,1))$. For notational brevity, we shall also sometimes write $\dif\mathscr{H}^{n-1}(x)=\dif\sigma_{x}$, but this will be clear from the context. Moreover, we denote $\mathscr{M}(\Omega;\R^{m})$ the $\R^{m}$-valued finite Radon measures on $\Omega$. Given $\mu\in\mathscr{M}(\Omega;\R^{m})$ and $A\in\mathscr{B}(\Omega)$ (the Borel $\sigma$-algebra on $\Omega$), we recall that $\mu\mres A:=\mu(-\cap A)$ is the restriction of $\mu$ to $A$. When $u\in\lebe_{\locc}^{1}(\R^{n};\R^{m})$ or $\mu\in\mathscr{M}(\R^{n};\R^{m})$, we denote for a bounded set $A\in\mathscr{B}(\R^{n})$ with $\mathscr{L}^{n}(A)>0$ 
\begin{align*}
(u)_{A}:=\dashint_{A}u\dif\mathscr{L}^{n}:=\frac{1}{\mathscr{L}^{n}(A)}\int_{A}u\dif x\;\;\text{and}\;\;(\mu)_{A}:=\dashint_{A}\dif\mu := \frac{\mu(A)}{\mathscr{L}^{n}(A)}. 
\end{align*}
If $A=\ball(x,r)$ is a ball, we write $(u)_{x,r}:=(u)_{\ball(x,r)}$ or $(\mu)_{x,r}:=(\mu)_{\ball(x,r)}$. If, however, $A\in\mathscr{B}(\R^{n})$ is such that $\mathscr{H}^{n-1}(A)\in (0,\infty)$ and $u\colon A\to\R^{m}$ is integrable with respect to $\mathscr{H}^{n-1}$, then we employ the notation 
\begin{align*}
\dashint_{A}u\dif\mathscr{H}^{n-1}:=\frac{1}{\mathscr{H}^{n-1}(A)}\int_{A}u\dif\mathscr{H}^{n-1}. 
\end{align*}
Lastly, we denote by $c,C>0$ generic constants that might change from line to line and shall only be specified provided their precise dependence on foregoing parameters is required. Similarly, we write $a\simeq b$ if there exist two constants $c,C>0$ such that $ca \leq b \leq Ca$; in particular, $c,C>0$ do not depend on $a$ or $b$. 
\subsection{The space $\bd$}\label{sec:bd}
In the following we recall the definition and record the properties of $\bd$-maps as shall be required in the upcoming sections; for more detail, the reader is referred to \cite{ST,Baba,AG0,ACD} and the references therein. We say that a measurable map $u\colon\Omega\to\R^{n}$ belongs to $\bd(\Omega)$ (and is then said to be of \emph{bounded deformation}) if and only if $u\in\lebe^{1}(\Omega;\R^{n})$ and 
\begin{align}
|\E u|(\Omega):=\sup\left\{\int_{\Omega}\langle u,\di(\varphi)\rangle\dif x\colon\;\varphi\in\hold_{c}^{1}(\Omega;\mathbb{B}_{\sym}^{n\times n}) \right\}<\infty.
\end{align}
The space $\bd_{\locc}(\Omega)$ then is defined in the obvious manner. Given $u\in\bd(\Omega)$, the Lebesgue-Radon-Nikod\'{y}m decomposition of $\E u$ yields 
\begin{align}\label{eq:BDdecomp}
\E u = \E^{a}u + \E^{s}u & = \mathscr{E}u\mathscr{L}^{n}\mres\Omega + \frac{\dif \E^{s}u}{\dif |\E^{s}u|}|\E^{s}u|,
\end{align}
where $\E^{a}u\ll\mathscr{L}^{n}$ and $\E^{s}u\bot\mathscr{L}^{n}$ are the absolutely continuous or singular parts of $\E u$ with respect to $\mathscr{L}^{n}$, respectively; in particular, we have $u\in\ld(\Omega)$ if and only if $u\in\bd(\Omega)$ and $\E u\ll\mathscr{L}^{n}$. Moreover, $\mathscr{E}u$ is the density of $\E^{a}u$ with respect to $\mathscr{L}^{n}$ and coincides with the symmetric part of the approximate gradient of $u$, cf. \cite{ACD}. Throughout, we will work with the following modes of convergence: Let $u,u_{1},u_{2},...\in\bd(\Omega)$. We say that $(u_{k})$ converges to $u$ in the \emph{norm topology} provided $\|u_{k}-u\|_{\bd(\Omega)}\to 0$, where $\|v\|_{\bd(\Omega)}:=\|v\|_{\lebe^{1}(\Omega;\R^{n})}+|\E v|(\Omega)$ for $v\in\bd(\Omega)$. On the other hand, we say that $(u_{k})$ converges to $u$ in the\footnote{Strictly speaking, being usually reserved for the $\bv$-case, these notions should be termed \emph{symmetric} weak*-, strict and area-strict convergence. As we shall work with $\bd$ exclusively, however, no confusions will arise from this.} \emph{weak*-sense} if $u_{k}\to u$ strongly in $\lebe^{1}(\Omega;\R^{n})$ and $\E u_{k}\wstar \E u$ in the sense of weak*-convergence of $\rsym$-valued Radon measures on $\Omega$, and in the \emph{strict sense} (or \emph{strictly}) if $u_{k}\to u$ strongly in $\lebe^{1}(\Omega;\R^{n})$ and $|\E u_{k}|(\Omega)\to |\E u|(\Omega)$ as $k\to\infty$. Lastly, if $u_{k}\to u$ strongly in $\lebe^{1}(\Omega;\R^{n})$ and 
\begin{align}
\int_{\Omega}\sqrt{1+|\mathscr{E}u_{k}|^{2}}\dif x + |\E^{s}u_{k}|(\Omega) \to \int_{\Omega}\sqrt{1+|\mathscr{E}u|^{2}}\dif x + |\E^{s}u|(\Omega),\qquad k\to\infty, 
\end{align}
then we shall say that $(u_{k})$ converges to $u$ in the \emph{area-strict sense}. Note that, if we put $\langle \cdot \rangle :=\sqrt{1+|\cdot|^{2}}$, then area-strict convergence amounts to $\langle \E u_{k}\rangle(\Omega)\to\langle \E u\rangle(\Omega)$ in the sense of functions of measures to be recalled in Section~\ref{sec:functionsofmeasures} below. It is then routine to show that norm implies area-strict, area-strict implies strict and strict implies weak*-convergence. When working with $u\in\ld(\Omega)$, we employ the norm $\|u\|_{\ld(\Omega)}:=\|u\|_{\lebe^{1}(\Omega;\R^{n})}+\|\sg(u)\|_{\lebe^{1}(\Omega;\R_{\sym}^{n\times n})}$ (recall that we write $\sg(u)$ for $\E u$ provided $\E u\ll\mathscr{L}^{n}$). Moreover, if $u\in\bd(\Omega)$, then there exists a sequence $(u_{j})\subset \hold^{\infty}(\Omega;\R^{n})\cap\ld(\Omega)$ such that $u_{j}\to u$ strictly as $j\to\infty$; clearly, if $\Omega=\R^{n}$, we may even choose $(u_{j})\subset\hold_{c}^{\infty}(\R^{n};\R^{n})$.

As is by now well-known (cf.~\cite{ST,Baba,BDG}), Lipschitz regularity of $\partial\!\Omega$ implies the existence of a linear, bounded, surjective boundary trace operator $\trace_{\partial\Omega}\colon\bd(\Omega)\to\lebe^{1}(\partial\!\Omega;\R^{n})$, where boundedness is understood with respect to the respective norm topologies. Crucially, this operator is already surjective when acting on $\ld(\Omega)$. Moreover, it is also continuous for strict convergence in $\bd(\Omega)$ (and hence area-strict convergence, too) \emph{but not} for weak*-convergence as specified above. Let us moreover note that, if $u\in\bd(\Omega)$, then the trivial extension 
\begin{align*}
\overline{u}:=\begin{cases} u&\;\text{in}\;\Omega,\\
0&\;\text{in}\;\R^{n}\setminus\Omega\end{cases}
\end{align*}
belongs to $\bd(\R^{n})$ as well, and the operator $\mathfrak{E}\colon u\mapsto \overline{u}$ is linear and bounded from $\bd(\Omega)$ to $\bd(\R^{n})$. We can now collect a refined result on smooth approximation, cf. \cite[Sec.~4]{GK1}:
\begin{lemma}[(Area-)strict smooth approximation]\label{lem:smooth}
Let $\Omega\subset\R^{n}$ be an open and bounded with Lipschitz boundary and let $u_{0}\in\ld(\Omega)$. Then for each $u\in\bd(\Omega)$ there exists a sequence $(u_{j})\subset u_{0}+\hold_{c}^{\infty}(\Omega;\R^{n})$ such that $\|u_{j}-u\|_{\lebe^{1}(\Omega;\R^{n})}\to 0\,$ and 
\begin{align*}
\int_{\Omega}\sqrt{1+|\sg(u_{j})|^{2}}\dif x \to \int_{\Omega}\sqrt{1+|\mathscr{E}u|^{2}}\dif x + |\!\E^{s}u|(\Omega) + \int_{\partial\Omega}|\trace_{\partial\Omega}(u_{0}-u)\odot\nu_{\partial\Omega}|\dif\mathscr{H}^{n-1}. 
\end{align*}
\end{lemma}
If $\Sigma\subset\Omega$ is a $\hold^{1}$-manifold oriented by $\nu\colon\Sigma\to\mathbb{S}^{n-1}$ and $u\in\bd(\R^{n})$, then $\E u\mres\Sigma$ is given by \textsc{Kohn}'s formula (cf.~\cite{Kohn1})
\begin{align}\label{eq:Kohnformula}
\E u\mres\Sigma = (u^{+}-u^{-})\odot\nu \mathscr{H}^{n-1}\mres\Sigma, 
\end{align}
where $u^{+}$ and $u^{-}$ are the right and left interior traces of $u$ along $\Sigma$. These, in turn, are well-defined upon the orientation of $\nu$, and can be computed for $\mathscr{H}^{n-1}$-a.e. $x\in\Sigma$ by virtue of 
\begin{align}\label{eq:tracesdash}
\lim_{r\searrow 0}\dashint_{\Sigma^{\pm}(x,r)}|u(y)-u^{\pm}(x)|\dif y = 0
\end{align}
for $\mathscr{H}^{n-1}$-a.e. $x\in\Sigma$, where $\Sigma^{\pm}(x,r):=\ball(x,r)\cap\{y\in\R^{n}\colon\;\langle y-x,\nu(x)\rangle \gtrless 0\}$ for $r>0$.

We will also need a fractional embedding theorem for $\bd$ as one of the main ingredients in the partial regularity proof below. Let $0<\theta<1$ and $1\leq p <\infty$. Given $U,\Sigma\subset\R^{n}$ with $\mathscr{L}^{n}(U)>0$ and $\mathscr{H}^{n-1}(\Sigma)\in (0,\infty)$, we define
\begin{align*}
&[u]_{\sobo^{\theta,p}(U;\R^{m})}:=\Big(\iint_{U\times U}\frac{|u(x)-u(y)|^{p}}{|x-y|^{n+\theta p}}\dif x\dif y \Big)^{\frac{1}{p}}\;\;\;\;\;\;\;\;\;\;\;\;\;\;\;\;\;\;\;\;\;\;\;\;\;\;\;\;\;\;\;\;\;\text{for}\;u\in\lebe^{p}(U;\R^{m}),\\
&[v]_{\sobo^{\theta,p}(\Sigma;\R^{m})}:=\Big(\iint_{\Sigma\times\Sigma}\frac{|u(x)-u(y)|^{p}}{|x-y|^{n-1+\theta p}}\dif\mathscr{H}^{n-1}(x)\dif\mathscr{H}^{n-1}(y) \Big)^{\frac{1}{p}}\;\;\;\;\text{for}\;v\in\lebe^{p}(\Sigma;\R^{m};\mathscr{H}^{n-1}), 
\end{align*}
where $\lebe^{p}(\Sigma;\R^{m};\mathscr{H}^{n-1})$ is the space of maps $v\colon\Sigma\to\R^{m}$ which are $p$-integrable for $\mathscr{H}^{n-1}$. 
The full norm on $\sobo^{\theta,p}(U;\R^{m})$ or $\sobo^{\theta,p}(\Sigma;\R^{m})$ then is given by $\|\cdot\|_{\sobo^{\theta,p}(U;\R^{m})}:=\|u\|_{\lebe^{p}(U;\R^{m})}+[u]_{\sobo^{\theta,p}(U;\R^{m})}$ (analogously for $\sobo^{\theta,p}(\Sigma;\R^{m})$). Following \textsc{Kolyada} \cite{Kolyada} (also see \textsc{Bourgain} et al. \cite{BBM}), it is well-known that $\bv(\Omega)\hookrightarrow \sobo^{\theta,n/(n-1+\theta)}(\Omega)$ for $n\geq 2$ and $\theta\in (0,1)$. For $\bd(\Omega)$, we require the recent theory of \textsc{Van Schaftingen} \cite{VS}: By \cite[Prop.~6.3]{VS} and since $n\geq 2$ in our setting, the symmetric gradient $\sg$ is \emph{elliptic} and \emph{cancelling}. Writing $\sg=\sum_{k=1}^{n}A_{k}\partial_{k}$ with $A_{k}\in\mathscr{L}(\R^{n};\rsym)$, ellipticity here means that the symbol map $\sg[\xi]:=\sum_{k=1}^{n}A_{k}\xi_{k}\colon\R^{n}\to\rsym$ is injective for all $\xi=(\xi_{1},...,\xi_{n})\in\R^{n}\setminus\{0\}$. In turn, cancellation means that the Fourier symbol $\sg[\xi]$ is sufficiently spread in the sense that
\begin{align}\label{eq:cancelling}
\bigcap_{\xi\in\R^{n}\setminus\{0\}}\sg[\xi](\R^{n})=\{0\}. 
\end{align}
Hence by \cite[Thms.~1.3,~8.1]{VS}, for each $\theta\in (0,1)$ there exists $c=c(n,\theta)>0$ such that\footnote{Note that the embedding $\bd\hookrightarrow\lebe^{\frac{n}{n-1}}$ is originally due to \textsc{Strauss} \cite{Strauss}.} 
\begin{align}\label{eq:prelimscancelling}
\|\varphi\|_{\lebe^{\frac{n}{n-1}}(\R^{n};\R^{n})}+[\varphi]_{\sobo^{\theta,n/(n-1+\theta)}(\R^{n};\R^{n})}\leq c \int_{\R^{n}}|\sg(\varphi)|\dif x\qquad\text{for all}\;\varphi\in\hold_{c}^{\infty}(\R^{n};\R^{n}). 
\end{align}
To state the next proposition, we remind the reader that on connected, open subsets $\Omega$ of $\R^{n}$, the nullspace of the symmetric gradient operator in $\mathscr{D}'(\Omega;\R^{n})$ is given by the \emph{rigid deformations}
\begin{align}\label{eq:rigiddeformations}
\mathscr{R}(\Omega):=\big\{x\mapsto Ax+b\colon\;A=-A^{\top},\;b\in\R^{n} \big\}.
\end{align}
If $\partial\Omega$ moreover is Lipschitz, for each $u\in\bd(\Omega)$ there exists $a\in\mathscr{R}(\Omega)$ such that 
\begin{align}\label{eq:PoincareonBD}
\|u-a\|_{\lebe^{1}(\Omega;\R^{n})}=\inf_{b\in\mathscr{R}(\Omega)}\|u-b\|_{\lebe^{1}(\Omega;\R^{n})}\leq c|\E u|(\Omega),
\end{align}
where $c=c(\Omega,n)>0$. We refer to \eqref{eq:PoincareonBD} as the \emph{Poincar\'{e} inequality on $\bd(\Omega)$}. Now we have 
\begin{proposition}\label{prop:fractionalpoincareBD}
Let $n\geq 2$ and $0<\theta<1$. Moreover, let $\Omega\subset\R^{n}$ be an open and bounded domain with Lipschitz boundary. Then there holds 
\begin{align}\label{eq:mainembeddingBD}
\bd(\Omega)\hookrightarrow \sobo^{\theta,\frac{n}{n-1+\theta}}(\Omega;\R^{n}), 
\end{align}
continuity of the embedding being understood with respect to the norm topologies.
\begin{enumerate}
\item\label{item:poincare1} If $\Omega$ moreover is connected, then for each $u\in\bd(\Omega)$ there exists $a\in\mathscr{R}(\Omega)$ such that 
\begin{align*}
\|u-a\|_{\sobo^{\theta,\frac{n}{n-1+\theta}}(\Omega;\R^{n})}\leq c|\E u|(\Omega), 
\end{align*}
where $c>0$ is a constant that only depends on $\Omega,n$ and $\theta$. 
\item\label{item:poincare2} There exists a constant $c=c(n,\theta)>0$ such that for every $x_{0}\in\R^{n}$, $R>0$ and every $u\in\bd(\R^{n})$ there exists $a\in\mathscr{R}(\R^{n})$ such that 
\begin{align}\label{eq:scalingSQC1}
\left(\dashint_{\ball(x_{0},R)}\int_{\ball(x_{0},R)}\frac{|u_{a}(x)-u_{a}(y)|^{\frac{n}{n-1+\theta}}}{|x-y|^{n+\theta n/(n-1+\theta)}}\dif x \dif y\right)^{\frac{n-1+\theta}{n}}\leq CR^{1-\theta}\dashint_{\ball(x_{0},R)}|\E u|, 
\end{align}
where $u_{a}:=u-a$. 
\end{enumerate}
\end{proposition}
\begin{proof}
Let $u\in\bd(\R^{n})$ first and choose a sequence $(u_{j})\subset\hold_{c}^{\infty}(\R^{n};\R^{n})$ such that $u_{j}\to u$ strictly in $\bd(\R^{n})$. Passing to a suitable subsequence, we may assume without loss of generality that $u_{j}\to u$ $\mathscr{L}^{n}$-a.e. in $\R^{n}$. Therefore, by Fatou's lemma and \eqref{eq:prelimscancelling},
\begin{align}\label{eq:extendextend}
\begin{split}
\|u\|_{\lebe^{\frac{n}{n-1}}(\R^{n};\R^{n})}+[u]_{\sobo^{\theta,n/(n-1+\theta)}(\R^{n};\R^{n})} & \leq \liminf_{j\to\infty} \|u_{j}\|_{\lebe^{\frac{n}{n-1}}(\R^{n};\R^{n})}+[u_{j}]_{\sobo^{\theta,n/(n-1+\theta)}(\R^{n};\R^{n})} \\ & \leq c\liminf_{j\to\infty} \int_{\R^{n}}|\sg(u_{j})|\dif x = c|\E u|(\R^{n}), 
\end{split}
\end{align}
Now let $u\in\bd(\Omega)$, where $\Omega\subset\R^{n}$ is open and bounded with Lipschitz boundary $\partial\Omega$. By the above, there exists a bounded linear extension operator $\mathfrak{E}\colon \bd(\Omega)\to\bd(\R^{n})$. Therefore, 
\begin{align*}
\|u\|_{\sobo^{\theta,n/(n-1+\theta)}(\Omega;\R^{n})} & \leq \max\{1,\mathscr{L}^{n}(\Omega)^{\theta/n}\}(\|u\|_{\lebe^{\frac{n}{n-1}}(\Omega;\R^{n})}+[u]_{\sobo^{\theta,n/(n-1+\theta)}(\Omega;\R^{n})})\\
& \leq \max\{1,\mathscr{L}^{n}(\Omega)^{\theta/n}\}(\|\mathfrak{E}u\|_{\lebe^{\frac{n}{n-1}}(\R^{n};\R^{n})}+[\mathfrak{E}u]_{\sobo^{\theta,n/(n-1+\theta)}(\R^{n};\R^{n})})\\
& \!\!\!\!\stackrel{\eqref{eq:extendextend}}{\leq} C(n,\theta,\mathscr{L}^{n}(\Omega))|\E\mathfrak{E}u|(\R^{n}) \leq C(n,\theta,\Omega)\|u\|_{\bd(\Omega)}, 
\end{align*}
and \eqref{eq:mainembeddingBD} follows. If $\Omega$ moreover is connected, pick $a\in\mathscr{R}(\Omega)$ such that \eqref{eq:PoincareonBD} holds; applying the preceding inequality to $u-a$ and \eqref{eq:PoincareonBD} consequently imply \ref{item:poincare1}. Ad~\ref{item:poincare2}. We may assume that $x_{0}=0$, and shall write $\ball_{r}:=\ball(0,r)$ for $r>0$. Letting $u\in\bd(\R^{n})$, we first determine an element $b\in\mathcal{R}(\R^{n})$ such that $(*)$ in the following inequality holds on the unit ball, due to part~\ref{item:poincare1} with $\Omega=\ball(0,1)$ and $p=\frac{n}{n-1+\theta}$:
\begin{align*}
\left(\int_{\ball_{R}}\int_{\ball_{R}}\frac{|u_{b}(x)-u_{b}(y)|^{p}}{|x-y|^{n+\theta p}}\dif x \dif y\right)^{\frac{1}{p}} & = \frac{R^{\frac{2n}{p}}}{R^{\frac{n+\theta p}{p}}}\left(\int_{\ball_{1}}\int_{\ball_{1}}\frac{|u_{b}(Rx)-u_{b}(Ry)|^{p}}{|x-y|^{n+\theta p}}\dif x \dif y\right)^{\frac{1}{p}} \\
& \stackrel{(*)}{\leq} c\frac{R^{\frac{2n}{p}}}{R^{\frac{n+\theta p}{p}}}R\int_{\ball_{1}}|(\sg(u_{b}))(Rx)|\dif x = c R^{\frac{n}{p}}R^{1-\theta}\dashint_{\ball_{R}}|\E u|.
\end{align*}
This in turn determines $a\in\mathcal{R}(\R^{n})$ for \eqref{eq:scalingSQC1}, and the proof is hereby complete. 
\end{proof}
%Two technical remarks are in order, first a comment on mollifications. 
%\begin{remark}
%If $u\in\bd(\R^{n})$ and $\rho\in\hold_{c}^{\infty}(\ball(0,1);[0,1])$ is a radially symmetric standard mollifier, we have $\rho_{\varepsilon}*(u-a)\to u-a$ in $\lebe_{\locc}^{1}(\R^{n};\R^{n})$ as $\varepsilon>0$. Therefore, 
%\begin{align*}
%\lim_{\varepsilon\searrow 0}\|\rho_{\varepsilon}*(u-a)\|\leq 
%\end{align*}
%\end{remark}
The dimensional hypothesis $n\geq 2$ in Proposition~\ref{prop:fractionalpoincareBD} in fact cannot be omitted: 
\begin{remark}[$n=1$]\label{rem:BVFubinibad}
The previous proposition does not remain valid for $n=1$. This can be seen by the fact that $\sobo^{1,1}((a,b))\not\hookrightarrow\sobo^{\theta,1/\theta}((a,b))$ for any $0<\theta<1$ and all $-\infty<a<b<\infty$. For example, pick $\theta=\frac{1}{2}$. Then it is well-known that $\sobo^{1,1}((a,b))$ embeds into the Besov-Nikolski\u{\i}-space $\besov_{2,\infty}^{1/2}$ but not into $\besov_{2,2}^{1/2}((a,b))=\sobo^{1/2,2}((a,b))$. In fact, continuity of the embedding would yield that, as $n=1$, $\bd((a,b))=\bv((a,b))$ embeds into $\sobo^{1/2,2}((a,b))$ by smooth approximation, but the sign function belongs to $\bv((-1,1))$ but not to $\sobo^{1/2,2}((-1,1))$. 
\end{remark}
\begin{remark}[Projection stability]\label{rem:stability}
Since $\mathscr{R}(\Omega)$ is a finite dimensional space of polynomials, the map $a$ as in \eqref{eq:PoincareonBD} can be taken to be the $\lebe^{2}$-projection of $u$ onto $\mathscr{R}(\Omega)$ (which here is well-defined for $\lebe^{1}$-maps, too); cf.~\cite[Sec.~3]{BDG}. In particular, it satisfies the stability estimate 
\begin{align*}
\dashint_{\ball(x_{0},r)}|a|\dif x \leq c \dashint_{\ball(x_{0},r)}|u|\dif x,\qquad u\in\bd(\ball(x_{0},r)). 
\end{align*}
\end{remark}
\subsection{Functions of measures}\label{sec:functionsofmeasures}
Here we briefly record the most important features of functions being applied to measures. First, let $f\colon\rsym\to\R_{\geq 0}$ be convex and satisfy the growth bound $c_{1}|z|-c_{2}\leq f(z) \leq c_{3}(1+|z|)$ for some $c_{1},c_{2},c_{3}>0$ and all $z\in\rsym$. We recall that the \emph{recession function} $f^{\infty}\colon\rsym\to\R$ is given by 
\begin{align*}
f^{\infty}(z):=\limsup_{t\searrow 0}tf\Big(\frac{z}{t}\Big),\qquad z\in\rsym, 
\end{align*}
and by convexity and the linear growth hypothesis, the $\limsup$ is a actually a limit. Given $\mu\in\mathcal{M}(\Omega;\rsym)$, we denote its Lebesgue-Radon-Nikod\'{y}m decomposition $\mu=\mu^{a}+\mu^{s}$ and then define the measure $f(\mu)$ for $A\in\mathscr{B}(\Omega)$ by 
\begin{align}\label{eq:fcmeas}
f(\mu)(A):= \int_{A}f(\mu) := \int_{A}f\left(\frac{\dif\mu^{a}}{\dif\mathscr{L}^{n}} \right)\dif\mathscr{L}^{n}+\int_{A}f^{\infty}\left(\frac{\dif\mu^{s}}{\dif |\mu^{s}|}\right)\dif |\mu^{s}|. 
\end{align}
If $\xi\in\rsym$, we put $f(\mu-\xi):=f(\mu-\xi\mathscr{L}^{n})$. Now suppose that $f\in\hold(\rsym)$ is merely assumed symmetric-rank-one convex (so is convex with respect to directions in the symmetric rank-one cone $\R^{n}\odot\R^{n}:=\{a\odot b\colon\;a,b\in\R^{n}\}$) and of linear growth in the sense of \eqref{eq:lingrowth}. Even though not giving rise to a positive measure, \eqref{eq:fcmeas} still is a valid definition provided the density $\frac{\dif\mu^{s}}{\dif |\mu^{s}|}$ takes values in the symmetric-rank-one cone $|\mu^{s}|$-a.e.. In fact, in this situation, $f$ is convex along directions contained in $\R^{n}\odot\R^{n}$ and so, by the linear growth assumption, $f^{\infty}(z)$ exists provided $z\in\R^{n}\odot\R^{n}$. When applying such integrands $f$ to $\E u$ for $u\in\bd(\Omega)$, then the recent work of \textsc{De Philippis \& Rindler} \cite{DPR} yields $\frac{\dif\E u}{\dif |\E^{s}u|}\in\R^{n}\odot\R^{n}$ $|\E^{s}u|$-a.e.. Hence 
\begin{align*}
f(\E u)(A):=\int_{A}f(\E u) := \int_{A}f(\mathscr{E}u)\dif\mathscr{L}^{n}+\int_{A}f^{\infty}\left(\frac{\dif\E^{s}u}{\dif |\E^{s}u|}\right)\dif |\E^{s}u|\qquad\text{for all}\;A\in\mathscr{B}(\Omega)
\end{align*}
for $u\in\bd(\Omega)$ is in fact a well-posed definition. Working from the previous ideas and upon the method of proof for signed variants given in \cite{KristensenRindler}, our fundamental background fact result is
\begin{theorem}[{\textsc{Rindler} \cite{Ri1}}]\label{thm:rindler}
Let $\Omega\subset\R^{n}$ be an open and bounded Lipschitz domain and let $f\in\hold(\R_{\sym}^{n\times n})$ be a symmetric quasiconvex integrand which, in addition, satisfies \eqref{eq:lingrowth}. Also, let $u_{0}\in\bd(\Omega)$. Then, with the notation of \eqref{eq:BDdecomp}, the functional 
\begin{align*}
\overline{F}_{u_{0}}[u;\Omega] & := \int_{\Omega}f(\mathscr{E}u)\dif x + \int_{\Omega}f^{\infty}\left(\frac{\dif \E^{s}u}{\dif |\E^{s}u|}\right)\dif |\E^{s}u|+ \int_{\partial\Omega}f^{\infty}(\trace_{\partial\Omega}(u_{0}-u)\odot\nu_{\partial\Omega})\dif\mathcal{H}^{n-1}
\end{align*}
for $u\in\bd(\Omega)$ is lower semicontinuous with respect to weak*-convergence in the space $\bd(\Omega)$.
\end{theorem}
Finally, a lemma on the continuity of symmetric rank-one-convex functions for the area-strict metric that we shall frequently employ in conjunction with smooth approximation; in effect, it appears as a generalisation of the classical convex \textsc{Reshetnyak} (semi-)continuity theory \cite{Resh}:
\begin{lemma}[Symmetric rank-one-convexity and area-strict continuity]\label{lem:symareastrict}
Let $f\in\hold(\rsym)$ be symmetric rank-one convex with \eqref{eq:lingrowth} and let $\Omega\subset\R^{n}$ be an open and bounded set. Then $\bd(\Omega)\ni u \mapsto f(\E u)(\Omega)$ is continuous with respect to area-strict convergence. 
\end{lemma}
The lemma follows from \cite[Prop.~5.1]{BDG} upon \textsc{Kristensen \& Rindler}'s refinement for signed integrands, \cite[Thm.~4]{KristensenRindler} and specifying to the symmetric gradient. Rather than reproducing the proof of \cite[Prop.~3.1]{GK2} with the relevant but easy modifications, we confine to stating the following equivalence between strong symmetric quasiconvexity at some $z_{0}\in\rsym$ and coerciveness; recall that $\Omega\subset\R^{n}$ is assumed to be open and bounded with Lipschitz boundary.
\begin{lemma}\label{lem:minseq}
Let $f\in\hold(\rsym)$ satisfy \eqref{eq:lingrowth} and let $u_{0}\in\ld(\Omega)$ be a given Dirichlet datum. Then all minimising sequences of the variational problem 
\begin{align}\label{eq:minseq}
\text{to infimise}\;\int_{\Omega}f(\sg(v))\dif x\;\;\;\;\text{over $\ld_{u_{0}}(\Omega)$}
\end{align}
are bounded in $\ld(\Omega)$ if and only if there exists $z_{0}\in\rsym$ and $\ell>0$ such that the function $h\colon z \mapsto f(z)-\ell V(z)$ is symmetric quasiconvex at $z_{0}$ (meaning that $h(z_{0})\leq \int_{Q}h(z_{0}+\sg(\varphi))\dif x$ for all $\varphi\in\hold_{c}^{\infty}(Q;\R^{n})$).
\end{lemma}
\subsection{$V$-function estimates and Korn-type inequalities}\label{sec:Orlicz}
For future applications in Section~\ref{sec:SSC}, we record some non-standard forms of Korn-type inequalities and gather here the relevant background results from \textsc{Breit \& Diening} \cite{BD12}. Note that, alternatively, the specifically required forms could also be traced back to \textsc{Acerbi \& Mingione} \cite{AcerbiMingione2} but then would follow only by inspection of the proof of \cite[Thm.~3.1]{AcerbiMingione2}. 

A differentiable function $\psi\colon\R_{\geq 0}\to [0,\infty)$ is called an $N$-function provided $\psi(0)=0$, its derivative $\psi'$ is right-continuous, non-decreasing and satisfies 
\begin{align}\label{eq:Phi1}
\psi'(0)=0,\;\;\;\psi'(t)>0\;\;\text{for}\;t>0\;\;\;\text{and}\;\lim_{t\to\infty}\psi'(t)=\infty. 
\end{align}
We now say that an $N$-function $\psi$ is of class $\Delta_{2}$ provided there exists $K>0$ such that $\psi(2t)\leq K\psi(t)$ for all $t\geq 0$, and the infimum over all possible such constants is denoted $\Delta_{2}(\psi)$. Similarly, we say that an $N$-function $\psi$ is class $\nabla_{2}$ provided the Fenchel conjugate $\psi^{*}(t):=\sup_{s\geq 0}(st-\psi(s))$ is of class $\Delta_{2}$; we put $\nabla_{2}(\psi):=\Delta_{2}(\psi^{*})$. We then have 
\begin{proposition}[{\cite[Thm.~1.1]{BD12}}]\label{prop:OrliczKorn}
Let $\psi\colon \R_{\geq 0}\to \R_{\geq 0}$ be an $N$-function. Then the following are equivalent: 
\begin{enumerate}
\item\label{item:OK1} $\psi$ is both of class $\Delta_{2}$ and $\nabla_{2}$, abbreviated by $\psi\in\Delta_{2}\cap\nabla_{2}$. 
\item\label{item:OK2} There exists a constant $A>0$ such that for all $u\in\hold_{c}^{\infty}(\R^{n};\R^{n})$ there holds 
\begin{align*}
\int_{\R^{n}}\psi(|Du|)\dif x \leq \int_{\R^{n}}\psi(A|\sg(u)|)\dif x. 
\end{align*}
\item\label{item:OK3} There exists a constant $A'>0$ such that for all $u\in\hold^{1}(\R^{n};\R^{n})$ and all open balls $\ball(x_{0},R)$ there holds 
\begin{align*}
\int_{\ball(x_{0},R)}\psi(Du-(Du)_{\ball(x_{0},R)})\dif x \leq \int_{\ball(x_{0},R)}\psi(A'(\sg(u)-(\sg(u))_{\ball(x_{0},R)}))\dif x.
\end{align*}
\end{enumerate}
Should \ref{item:OK1} be satisfied, then the constants from \ref{item:OK2} and \ref{item:OK3} only depend on $\Delta_{2}(\psi)$ and $\nabla_{2}(\psi)$. 
\end{proposition}
We next consider \emph{shifted} $N$-functions (cf.~ \cite{DieningEttwein,DLSV}). Letting $\varphi\colon\R_{\geq 0}\to \R_{\geq 0}$ be an $N$-function, we put for $a\geq 0$
\begin{align}\label{eq:shiftedNfunction}
\varphi_{a}(t):=\int_{0}^{t}\frac{\varphi'(a+s)}{a+s}s\dif s,\qquad t\geq 0. 
\end{align}
The following lemma compactly gathers the for us most relevant results on shifted $N$-functions:
\begin{lemma}[{\cite[Lem.~23]{DieningEttwein}}]\label{lem:conjugate}
Let $\varphi\in \hold^{1}([0,\infty))\cap\hold^{2}((0,\infty);\R_{\geq 0})$ be an $N$-function such that $c_{1}t\varphi''(t) \leq \varphi'(t) \leq c_{2}t\varphi''(t)$ for some $c_{1},c_{2}>0$ and all $t>0$. Given $a\geq 0$, define $\varphi_{a}$ by \eqref{eq:shiftedNfunction}. Then both $\Delta_{2}(\varphi_{a})$ and $\nabla_{2}(\varphi_{a})$ can be bounded independently of $a$ and so $\varphi_{a}$ satisfies the $\Delta_{2}\cap\nabla_{2}$-condition uniformly in $a\geq 0$. 
\end{lemma}
We come to the requisite estimates for $V$-functions, which we define for $1\leq p < \infty$ by 
\begin{align}\label{eq:Vpfunction}
V_{p}(z):=\big(1+|z|^{2}\big)^{\frac{p}{2}}-1,\qquad z\in\R^{m}, 
\end{align}
so that, with the terminology of \eqref{eq:SSQC}ff., $V=V_{1}$; note that $V_{p}\in \Delta_{2}\cap\nabla_{2}$ if and only if $1<p<\infty$. 
\begin{lemma}[{\cite[Sec.~2.4]{GK2}, {\cite[Lem.~2.4]{DFLM}}}]\label{lem:auxVSQC}
Let $m\in\mathbb{N}$. Then there exist constants $c>0$ merely depending on $m$ such that there holds
\begin{align}\label{eq:Vest1}
\begin{split}
& (\sqrt{2}-1)\min\{|z|,|z|^{2}\}\leq V(z) \leq \min\{|z|,|z|^{2}\},\\
& V(\lambda z) \leq \lambda^{2}V(z),\\
& V(z+w)\leq 2(V(z)+V(w)) 
\end{split}
\end{align}
for all $z,w\in\R^{m}$ and $\lambda\geq 1$. Moreover, for each $1<p<\infty$ there exist two constants $0<\theta_{p}\leq\Theta_{p}<\infty$ such that for all $z,z'\in\R^{m}$ there holds 
\begin{align}\label{eq:Vpestimate}
\begin{split}
\theta_{p}\big(1+|z|^{2}+|z'|^{2}\big)^{\frac{p-2}{2}}|z'|^{2} & \leq V_{p}(z+z')-V_{p}(z)-\langle V_{p}'(z),z'\rangle  \\ & \;\;\;\;\;\;\;\;\;\;\;\;\;\;\;\;\;\;\;\;\;\;\;\;\;\;\;\;\leq \Theta_{p}\big(1+|z|^{2}+|z'|^{2}\big)^{\frac{p-2}{2}}|z'|^{2}.
\end{split}
\end{align}
\end{lemma}
\subsection{Miscellaneous auxiliary results}
In this final subsection we gather some mixed technical results. We begin with the  \textsc{Ekeland} variational principle \cite{Ekeland}, helping us to obtain good approximating sequences of certain $\bd$-maps later on, in a form given in \cite[Thm.~5.6, Rem.~5.5]{Giusti}:
\begin{lemma}[Ekeland variational principle]\label{lem:EkeLemma}
Let $(X,d)$ be a complete metric space and let $\mathcal{F}\colon X\to\R\cup\{+\infty\}$ be a lower semicontinuous function for the metric topology, bounded from below and taking a finite value at some point. Assume that for some $u\in X$ and some $\varepsilon>0$ we have $\mathscr{F}[u]\leq \inf_{X}\mathcal{F}+\varepsilon$. Then there exists $v\in X$ such that 
\begin{enumerate}
\item $d(u,v)\leq \sqrt{\varepsilon}$, 
\item $\mathcal{F}[v]\leq\mathcal{F}[u]$, 
\item $\mathcal{F}[v]\leq\mathcal{F}[w]+\sqrt{\varepsilon}d(v,w)$ for all $w\in X$. 
\end{enumerate}
\end{lemma}
For the following, let us recall that a symmetric bilinear form $\mathbb{A}\in\mathbb{S}(\R^{N\times n})$ is called \emph{strongly elliptic} or \emph{Legendre-Hadamard} provided there exists $\lambda>0$ such that for all $a\in\R^{N},b\in\R^{n}$ there holds $\A[a\otimes b,a\otimes b]\geq\lambda |a\otimes b|^{2}$. For such bilinear forms, we have the following
\begin{lemma}[{\cite[Lem.~15.2.1]{MazSha},\cite[Prop.~2.11]{GK2}, \cite[Lem.~2.11]{CFM}}]\label{lem:toplinear}
Let $1<p<\infty$ and $k\in\mathbb{N}$. Then, for any open ball $\ball$ in $\R^{n}$ and any strongly elliptic bilinear form $\mathbb{A}\in\mathbb{S}(\R^{N\times n})$, the mapping 
\begin{align*}
\sobo^{k,p}(\ball;\R^{N})\ni u \mapsto (-\di(\mathbb{A}\!\nabla u),\trace_{\partial\!\ball}u)\in\sobo^{k-2,p}(\ball;\R^{N})\times\sobo^{k-\frac{1}{p},p}(\partial\!\ball;\R^{N})
\end{align*}
is a topological isomorphism. Moreover, if $|\A|\leq\Lambda$ and $-\di(\mathbb{A}\!\nabla u)=0$ in $\mathscr{D}'(\Omega;\R^{N})$, then there holds $u\in\hold^{\infty}(\Omega;\R^{N})$ together with 
\begin{align*}
\sup_{\ball(x_{0},\frac{R}{2})}|\nabla u - A| + R\sup_{\ball(x_{0},\frac{R}{2})}|\nabla^{2}u|\leq C\dashint_{\ball(x_{0},R)}|\nabla u-A|\dif x\qquad\text{for all}\;A\in\R^{N\times n}
\end{align*}
for all $\ball(x_{0},R)\Subset\Omega$, where $C=C(n,N,\lambda,\Lambda)>0$ is a constant.
\end{lemma}
Finally, a standard iteration result:
\begin{lemma}[{\cite[Lem.~4.4]{GK2}}]\label{lem:iterationlemmaQC}
Let $\theta\in (0,1)$ and $R>0$. Suppose that $\Phi,\Psi\colon (0,R]\to\R$ are non-negative functions such that $\Phi$ is bounded and $\Psi$ is decreasing together with $\Psi(\sigma t)\leq\sigma^{-2}\Psi(t)$ for all $t\in (0,R]$ and $\sigma\in (0,1]$. Moreover, suppose that there holds 
\begin{align*}
\Phi(r) \leq \theta\Phi(s)+\Psi(s-r)
\end{align*}
for all $r,s\in [\tfrac{R}{2},R]$ with $r<s$. Then there exists a constant $C=C(\theta)>0$ such that 
\begin{align*}
\Phi\big(\tfrac{R}{2}\big) \leq C\Psi(R). 
\end{align*}
\end{lemma}
\section{The superlinear growth case: Proof of Theorem~\ref{thm:ppartialregularity}}\label{sec:SSC}
In this section we provide the proof of Theorem~\ref{thm:ppartialregularity}. We thereby establish that in the superlinear growth case, the partial regularity of minima for $p$-strongly symmetric quasiconvex integrands, $1<p<\infty$, can be reduced to the full gradient situation. This reduction cannot be employed in the linear growth case $p=1$, thereby particularly motivating an independent proof of Theorem~\ref{thm:main1}.

Henceforth, let $1\leq p < \infty$ and suppose that $G\in\hold(\R^{n\times n})$ is of $p$-growth in the sense that there exists $c>0$ such that 
\begin{align}\label{eq:pgrowth}
|G(z)|\leq c (1+|z|^{p})
\end{align}
for all $z\in\R^{n\times n}$. Recalling the function $V_{p}$ from \eqref{eq:Vpfunction}, we say that a function $G\in\hold(\R^{n\times n})$ is $p$-\emph{strongly quasiconvex} if and only if there exists $\lambda>0$ such that 
\begin{align}\label{eq:pSQC}
\R^{n\times n}\ni z\mapsto G(z)-\lambda V_{p}(z)\in \R\qquad\text{is quasiconvex}. 
\end{align}
As a consequence of the last part of Lemma~\ref{lem:auxVSQC}, if $1<p<\infty$, then $p$-strong quasiconvexity of $G\in\hold^{2}(\R^{n\times n})$ is equivalent to the existence of a constant $\nu>0$ such that 
\begin{align}\label{eq:convenientrewriteSQC}
\nu\int_{Q}(1+|z|^{2}+|D\varphi|^{2})^{\frac{p-2}{2}}|D\varphi|^{2}\dif x \leq \int_{Q}G(z+D\varphi)-G(z)\dif x
\end{align}
holds for all $z\in\R^{n\times n}$ and all $\varphi\in\hold_{c}^{\infty}(Q;\R^{n})$. Using $(1+t^{2})^{\frac{p-2}{2}}t^{2}\simeq t^{2}+t^{p}$ for all $t\geq 0$ provided $2\leq p <\infty$, for this range of $p$ it is easily seen that $p$-strong quasiconvexity of $G$ implies the existence of a constant $\mu>0$ such that 
\begin{align}\label{eq:weakerStrongQC}
\mu\int_{Q}|D\varphi|^{2}+|D\varphi|^{p}\dif x \leq \int_{Q}G(z+D\varphi)-G(z)\dif x
\end{align}
holds for all $z\in\R^{n\times n}$ and all $\varphi\in\hold_{c}^{\infty}(Q;\R^{n})$. In particular, however, we remark that the $1$-strong quasiconvexity in the sense of \eqref{eq:pSQC} is \emph{not} equivalent to \eqref{eq:convenientrewriteSQC}, cf. Remark~\ref{rem:p=1SQC}. In both the superlinear and linear growth regimes, analogous statements hold for ($p$-)strongly symmetric quasiconvex integrands. Toward the proof of Theorem~\ref{thm:ppartialregularity}, we record the following result attributable to \textsc{Acerbi \& Fusco} \cite{AcerbiFusco1} for $p\geq 2$ and \textsc{Carozza, Fusco} and \textsc{Mingione} \cite{CFM} for $1<p<2$:
\begin{proposition}[{\cite[Thm.~II.1]{AcerbiFusco1}, \cite[Thm.~3.2]{CFM}}]\label{prop:CFM}
Let $1<p<\infty$ and let $\Omega\subset\R^{n}$ be open. Suppose that $G\in\hold^{2}(\R^{n\times n})$ satisfies \eqref{eq:pgrowth} for all $z\in\R^{n\times n}$ and, 
\begin{enumerate}
\item\label{item:reductionreg1} if $p\geq 2$, \eqref{eq:weakerStrongQC} holds for some $\mu>0$, all $z\in\R^{n\times n}$ and $\varphi\in\hold_{c}^{\infty}(Q;\R^{n})$,
\item\label{item:reductionreg2} if $1<p<2$, \eqref{eq:convenientrewriteSQC} holds for some $\nu>0$, all $z\in\R^{n\times n}$ and $\varphi\in\hold_{c}^{\infty}(Q;\R^{n})$. 
\end{enumerate}
Then for any local minimiser $u\in\sobo_{\locc}^{1,p}(\Omega;\R^{n})$ of the integral functional 
\begin{align*}
v\mapsto \int G(Dv)\dif x
\end{align*}
there exists an open set $\Omega_{u}\subset\Omega$ with $\mathscr{L}^{n}(\Omega\setminus\Omega_{u})=0$ such that $u$ is of class $\hold^{1,\alpha}$ for any $0<\alpha<1$ in a neighbourhood of any of the elements of $\Omega_{u}$.
\end{proposition}
Working from Proposition~\ref{prop:CFM}, we proceed to establish Theorem~\ref{thm:ppartialregularity}:
\begin{proof}[Proof of Theorem~\ref{thm:ppartialregularity}, $p\geq 2$]
Let $g\in\hold^{2}(\rsym)$ satisfy the assumptions of Theorem~\ref{thm:ppartialregularity}. We then define a new integrand $G_{g}\colon\R^{n\times n}\to\R$ by 
\begin{align}\label{eq:Ggdef}
G_{g}(z):=g(z^{\sym}),\qquad z\in\R^{n\times n}.
\end{align}
Our aim is to establish that $G_{g}$ satisfies the assumptions of Proposition~\ref{prop:CFM}. Clearly, $G_{g}=g\circ\Pi_{\sym}$, where $\Pi_{\sym}\colon\R^{n\times n}\to\rsym$ is the orthogonal projection onto the symmetric matrices, and hence $G_{g}\in\hold^{2}(\R^{n\times n})$. Moreover, since $|z^{\sym}|\leq |z|$ for all $z\in\R^{n\times n}$, $|G_{g}(z)|\leq L_{p}(1+|z^{\sym}|^{p}) \leq L_{p}(1+|z|^{p})$ by \eqref{eq:pgrowthintro}, and so $G_{g}$ satisfies \eqref{eq:pgrowth} for all $z\in\R^{n\times n}$. It thus remains to show that $G_{g}$ satisfies \eqref{eq:weakerStrongQC}.

Recall that $g\colon\rsym\to\R$ satisfies \eqref{eq:pSSQC}. We note that, as $p\geq 2$, for each fixed $x\in Q$, the function $H_{x}\colon s\mapsto (1+s^{2}+|\sg(u)(x)|^{2})^{\frac{p-2}{2}}|\sg(u)(x)|^{2}$ is non-decreasing and $(1+t^{2})^{\frac{p-2}{2}}t^{2}\simeq t^{2}+t^{p}$ for all $t\geq 0$. We obtain similarly to \eqref{eq:convenientrewriteSQC} that there exists $\widetilde{\mu}>0$ such that 
\begin{align*}
\widetilde{\mu} \int_{Q}|\sg(\varphi)|^{2}+|\sg(\varphi)|^{p}\dif x \leq \int_{Q}g(z^{\sym}+\sg(\varphi))-g(z^{\sym})\dif x
\end{align*}
for all $z\in\R^{n\times n}$ and all $\varphi\in\hold_{c}^{\infty}(Q;\R^{n})$. By the usual Korn inequalities (i.e., considering the $N$-function $t\mapsto t^{q}$ for $q>1$ in Proposition~\ref{prop:OrliczKorn}), there exists $\mu>0$ such that 
\begin{align*}
\int_{Q}G_{g}(z+D\varphi)-G_{g}(z)\dif x & = \int_{Q}g(z^{\sym}+\sg(\varphi))-g(z^{\sym})\dif x \\ & \geq \widetilde{\mu} \int_{Q}|\sg(\varphi)|^{2}+|\sg(\varphi)|^{p}\dif x \geq\mu\int_{Q}|D\varphi|^{2}+|D\varphi|^{p}\dif x
\end{align*}
for all $\varphi\in\hold_{c}^{\infty}(Q;\R^{n})$ and $z\in\R^{n\times n}$. Hence $G_{g}$ satisfies \eqref{eq:weakerStrongQC} and so $G_{g}$ satisfies the requirements of Proposition~\ref{prop:CFM}. To conclude the proof, it suffices to note that if $u\in\sobo_{\locc}^{1,p}(\Omega;\R^{n})$ is a local minimiser for $v\mapsto \int g(\sg(v))\dif x$, then it is a local minimiser for $v\mapsto \int G_{g}(Dv)\dif x$. By Proposition~\ref{prop:CFM}~\ref{item:reductionreg1}, the proof is complete. 
\end{proof}
\begin{proof}[Proof of Theorem~\ref{thm:ppartialregularity}, $1<p<2$.] We may argue as in the proof in the superquadratic growth case apart from establishing the $p$-strong quasiconvexity of $G_{g}$. Note that, for $1<p<2$, $H_{x}$ as defined above is not non-decreasing anymore; hence a more refined argument is required. To establish that $G_{g}$ is $p$-strongly quasiconvex in the sense of Proposition~\ref{prop:CFM}~\ref{item:reductionreg2}, we claim that there exists a constant $c=c(p,n)>0$ such that 
\begin{align}\label{eq:Kornmainp}
\int_{Q}(1+|z|^{2}+|D\varphi|^{2})^{\frac{p-2}{2}}|D\varphi|^{2}\dif x \leq c\int_{Q}(1+|z^{\sym}|^{2}+|\sg(\varphi)|^{2})^{\frac{p-2}{2}}|\sg(\varphi)|^{2}\dif x  
\end{align}
holds for all $z\in\R^{n\times n}$ and all $\varphi\in\hold_{c}^{\infty}(Q;\R^{n})$. 

In view of \eqref{eq:Kornmainp}, let $\varphi\in\hold_{c}^{\infty}(Q;\R^{n})$ and $z\in\R^{n\times n}$ be arbitrary. Since $1<p<2$, the function $s\mapsto (1+|s|^{2}+|D\varphi(x)|^{2})^{\frac{p-2}{2}}$ is decreasing in $s$ for every $x\in Q$. Thus, as $|z^{\sym}|\leq |z|$ for all $z\in\R^{n\times n}$, 
\begin{align}\label{eq:Kornstart}
\int_{Q}(1+|z|^{2}+|D\varphi|^{2})^{\frac{p-2}{2}}|D\varphi|^{2}\dif x & \leq \int_{Q}(1+|z^{\sym}|^{2}+|D\varphi|^{2})^{\frac{p-2}{2}}|D\varphi|^{2}\dif x =: (*).
\end{align}
Now, define a function $\psi\colon\R_{\geq 0}\to\R_{\geq 0}$ by  
\begin{align}\label{eq:psimaindef}
\psi(t):=(1+t)^{p-2}t^{2},\qquad t\geq 0. 
\end{align}
Then we have, with the correspondingly shifted function $\psi_{a}$ being defined for $a\geq 0$ by \eqref{eq:shiftedNfunction},
\begin{align}\label{eq:psiaestimate}
\psi_{a}(t)\simeq \psi''(a+t)t^{2} \simeq (1+a+t)^{p-2}t^{2}\simeq (1+|a|^{2}+|t|^{2})^{\frac{p-2}{2}}t^{2}, 
\end{align}
and the constants implicit in '$\simeq$' are \emph{independent of $a$}; the lengthy yet elementary verification of this fact is deferred to the appendix, Section~\ref{sec:app1}. 
As a consequence of Lemma~\ref{lem:conjugate} and $p>1$, $\psi_{a}$ belongs to $\Delta_{2}\cap\nabla_{2}$ and, most importantly, $\Delta_{2}(\psi_{a})$ and $\nabla_{2}(\psi_{a})$ are independent of $a\geq 0$. Hence, by Proposition~\ref{prop:OrliczKorn}, there exists a constant $A=A(\Delta_{2}(\psi_{a}),\nabla_{2}(\psi_{a}))>0$ -- which, since $\Delta_{2}(\psi_{a})$ and $\nabla_{2}(\psi_{a})$ \emph{do not depend on $a$}, is actually independent of $a$: $A=A(\Delta_{2}(\psi),\nabla_{2}(\psi))>0$ -- such that for all $\varphi\in\hold_{c}^{\infty}(Q;\R^{n})$ there holds
\begin{align}\label{eq:TabbyKorn}
\int_{Q}\psi_{a}(|D\varphi|)\dif x \leq \int_{Q}\psi_{a}(A|\sg(\varphi)|)\dif x. 
\end{align}
Clearly, since $\psi$ and each $\psi_{a}$ are monotonically increasing, we may assume that $A>1$. Applying the previous inequality to the particular choice $a=|z^{\sym}|$, we therefore obtain
\begin{align}\label{eq:Korncontinue}
\begin{split}
(*) & \,\,\stackrel{\eqref{eq:psiaestimate}}{\leq} c\int_{Q}\psi_{|z^{\sym}|}(|D\varphi|)\dif x \\
& \; \stackrel{\eqref{eq:TabbyKorn}}{\leq} c \int_{Q}\psi_{|z^{\sym}|}(A|\sg(\varphi)|)\dif x \\ 
& \;\stackrel{\eqref{eq:psiaestimate}}{\leq} c \int_{Q}(1+|z^{\sym}|^{2}+A^{2}|\sg(\varphi)|^{2})^{\frac{p-2}{2}}A^{2}|\sg(\varphi)|^{2}\dif x \\ 
& \;\;\, \leq cA^{2}\int_{Q} (1+|z^{\sym}|^{2}+|\sg(\varphi)|^{2})^{\frac{p-2}{2}}|\sg(\varphi)|^{2}\dif x,
\end{split}
\end{align}
the last inequality being valid by $A>1$ and $p-2<0$. Then, combining \eqref{eq:Kornstart} and \eqref{eq:Korncontinue}, we arrive at \eqref{eq:Kornmainp}. We can then proceed in showing that $G_{g}$ defined by \eqref{eq:Ggdef} is $p$-strongly quasiconvex. To this end, let $\varphi\in\hold_{c}^{\infty}(Q;\R^{n})$ and $z\in\R^{n\times n}$ be arbitrary. We then find  
\begin{align*}
\int_{Q}G_{g}(z+D\varphi)-G_{g}(z)\dif x & \;= \int_{Q}g(z^{\sym}+\sg(\varphi))-g(z^{\sym})\dif x \\ 
& \!\!\!\!\!\!\!\!\!\stackrel{\eqref{eq:Vpestimate},\,\eqref{eq:pSSQC}}{\geq} \ell_{p}\theta_{p} \int_{Q}(1+|z^{\sym}|^{2}+|\sg(\varphi)|^{2})^{\frac{p-2}{2}}|\sg(\varphi)|^{2}\dif x \\ 
& \stackrel{\eqref{eq:Kornmainp}}{\geq} \nu \int_{Q}(1+|z|^{2}+|D\varphi|^{2})^{\frac{p-2}{2}}|D\varphi|^{2}\dif x, 
\end{align*}
where $\nu=\frac{\ell_{p}\theta_{p}}{c}$ with $\ell_{p}>0$ from \eqref{eq:pSSQC} and the constant $c>0$ from \eqref{eq:Kornmainp}. Thus Theorem~\ref{thm:ppartialregularity} follows for the growth range $1<p<2$ and in combination with the above, the proof of Theorem~\ref{thm:ppartialregularity} is complete. 
\end{proof}

\begin{remark}\label{rem:p=1SQC}
When linear growth integrands are concerned, setting $p=1$ in \eqref{eq:convenientrewriteSQC} does not give rise to an equivalent notion of ($1$-)strong quasiconvexity in the sense of \eqref{eq:pSQC} with $p=1$ (also see the restriction of exponents in Lemma~\ref{lem:auxVSQC}). This can be even seen for strongly convex linear growth integrands such as the area integrand $m\colon z\mapsto\sqrt{1+|z|^{2}}(=V(z)+1)$, compare \eqref{eq:VestimatesSQCAuxB} from below. The underlying reason for this is that convex, linear growth $\hold^{2}$-integrands typically exhibit $(p,q)$-type growth behaviour on the level of the second derivatives in the following sense: There exist $1<a<\infty$ and constants $\Lambda_{1},\Lambda_{2}>0$ such that 
\begin{align}\label{eq:aelliptic}
\Lambda_{1}\frac{|z|^{2}}{(1+|\xi|^{2})^{\frac{a}{2}}}\leq \langle m''(\xi)z,z\rangle \leq \Lambda_{2}\frac{|z|^{2}}{(1+|\xi|^{2})^{\frac{1}{2}}}\qquad\text{for all}\;z,\xi\in\R_{\sym}^{n\times n},
\end{align}
see \cite{GK1} and \cite[Ex.~4.17]{Bildhauer} for a discussion. In addition, note that if we consider an $a$-elliptic integrand satisfying \eqref{eq:SSQC} (the latter in turn expressing a coerciveness property on $\ld$ or $\bd$ but \emph{not} $\sobo^{1,1}$ or $\bv$), one obtains that generalised minima belong to $\sobo_{\locc}^{1,1}$ for if $1<a<1+\frac{2}{n}$, cf.~\cite[Thm.~1.1]{G1}. Even though the general reducibility to the full gradient case for $p=1$ is ruled out by \textsc{Ornstein}'s Non-Inequality, one might hope to employ a variant of such a procedure for $a$-elliptic integrands, where one has improved integrability estimates. However, the available $\sobo_{\locc}^{1,1}$-regularity result of \cite{G1} hinges on \emph{specific} minimising sequences being locally uniformly bounded in some $\sobo^{1,q}$, $q>1$, but this neither implies boundedness of all minimising sequences in $\sobo_{\locc}^{1,1}$ nor in $\sobo^{1,1}(\Omega;\R^{n})$. In light of \cite[Prop.~3.1]{GK2}, this would be necessary for a possible reduction procedure; still, if possible, it would only work for \emph{convex} integrands with a certain ellipticity ratio, not being applicable to the strongly (symmetric) quasiconvex case. 
\end{remark}

\section{A Fubini--type theorem for $\bd$--maps}\label{sec:fubini}
As one of the main tools in the proof of Theorem~\ref{thm:main1}, we now give a Fubini-type result for functions of bounded deformation. In effect, this establishes that on $\mathscr{L}^{1}$-a.e. sphere with fixed center, $\bd$-maps possess additional fractional differentiability and integrability; on arbitrary spheres, we can only expect $\lebe^{1}$-integrability of interior traces.  Aiming to linearise later on, suitable competitor maps attaining these more regular boundary values then will equally belong to better Sobolev spaces and so the results of Lemma~\ref{lem:toplinear} become accessible.
%\begin{proposition}\label{prop:fractionalspheresmain}
%Let $0<s<1$ and $1\leq p<\infty$. Then there exists a constant $C=C(s,p,n)>0$ such that for all $x_{0}\in\R^{n}$, $R>0$ and all $u\in\sobo^{s,p}(\R^{n};\R^{N})$ there holds 
%\begin{align}
%\int_{0}^{R}[u]_{\sobo^{s,p}(\partial\!\ball(x_{0},r);\R^{N})}^{p}\dif r \leq C [u]_{\sobo^{s,p}(\ball(x_{0},R);\R^{N})}^{p}, 
%\end{align}
%that is, 
%\begin{align}
%\begin{split}
%\int_{0}^{R}\iint_{\partial\!\ball(x_{0},r)\times\partial\!\ball(x_{0},r)}\frac{|u(x)-u(y)|^{p}}{|x-y|^{n+sp-1}} & \dif\sigma_{y}\dif\sigma_{x}\dif r \\ & \leq C \iint_{\ball(x_{0},R)\times\ball(x_{0},R)}\frac{|u(x)-u(y)|^{p}}{|x-y|^{n+sp}}\dif x\dif y. 
%\end{split} 
%\end{align}
%\end{proposition}
\begin{theorem}\label{thm:fubini1}
Let $n\geq 2$, $0<\theta<1$, $x_{0}\in\R^{n}$ and $u\in\bd_{\locc}(\R^{n})$. Then for $\mathscr{L}^{1}$--almost all radii $r>0$, $\mathscr{H}^{n-1}$-a.e. $x\in\partial\!\ball(x_{0},r)$ are Lebesgue points for $u$. For such $r>0$, the restrictions $u|_{\partial\!\ball(x_{0},r)}$ are hereafter well-defined and moreover belong to $\sobo^{\theta,n/(n-1+\theta)}(\partial\!\ball(x_{0},r);\R^{n})$. 

More precisely, there exists a constant $C=C(n,\theta)>0$ (which, in particular, is independent of $x_{0}$ and $u$) with the following property: For all $0<s<r<\infty$ there exist $t\in (s,r)$ and $\alpha\in\mathscr{R}(\R^{n})$ such that for $\mathscr{H}^{n-1}$-a.e. $x\in\partial\!\ball(x_{0},t)$, $u(x)$ coincides with its precise representative and there holds 
\begin{align}\label{eq:fubinispheremainestPR}
\begin{split}
\left(\dashint_{\partial\!\ball(x_{0},t)}\int_{\partial\!\ball(x_{0},t)}\frac{|u_{\alpha}(x)-u_{\alpha}(y)|^{\frac{n}{n-1+\theta}}}{|x-y|^{n-1+\frac{n\theta}{n-1+\theta}}}\dif\sigma_{x}\dif \sigma_{y} \right)^{\frac{n-1+\theta}{n}} &  \leq C\frac{r^{n}}{t^{\frac{(n-1)(n-1+\theta)}{n}}(r-s)^{\frac{n-1+\theta}{n}}}\times \\ & \times
\dashint_{\overline{\ball(x_{0},r)}}|\E u|.
\end{split}
\end{align}

\end{theorem}
\begin{proof}
It is no loss of generality to assume $x_{0}=0$, and hence we write $\ball_{r}:=\ball(0,r)$ in the sequel. For clarity, we divide the proof into three parts.

\emph{Step 1. A general Fubini-type theorem for $\sobo^{\vartheta,p}$-maps.} In a first step, we let $0<\vartheta<1$, $1\leq p<\infty$ and let $u\in(\sobo^{\vartheta,p}\cap\hold)(\R^{n};\R^{n})$. The aim of this step is to show the inequality
\begin{align}\label{eq:Fubinifractional}
\int_{0}^{R}\iint_{\partial\!\ball_{r}\times\partial\!\ball_{r}}\frac{|u(\widetilde{x})-u(\widetilde{y})|^{p}}{|\widetilde{x}-\widetilde{y}|^{n+\vartheta p-1}}\dif\sigma_{\widetilde{x}}\dif\sigma_{\widetilde{y}}\dif r \leq C \iint_{\ball_{R}\times\ball_{R}}\frac{|u(\widetilde{x})-u(\widetilde{y})|^{p}}{|\widetilde{x}-\widetilde{y}|^{n+\vartheta p}}\dif \widetilde{x}\dif \widetilde{y}
\end{align}
for all $R>0$, where $C=C(n,\vartheta,p)>0$ is a constant. Denoting the integral on the left by $(*)$, we change variables to the unit ball and put $\widetilde{x}=rx$, $\widetilde{y}=ry$. We thereby obtain, with $\mathbb{S}^{n-1}:=\partial\!\ball(0,1)$,
\begin{align}\label{eq:FubiniTscherp1}
\begin{split}
(*)\leq \iint_{\mathbb{S}^{n-1}\times\mathbb{S}^{n-1}}\int_{0}^{R}(r^{n-1})^{2}\frac{|u(rx)-u(ry)|^{p}}{|rx-ry|^{n+\vartheta p-1}} \dif r\dif\sigma_{y}\dif\sigma_{x}.
\end{split}
\end{align}
In comparison with the right-hand side of \eqref{eq:Fubinifractional}, the ultimate integral only contains \emph{one} integral with respect to the radii at the cost of a lower power in the integrand's denominator. We thus must argue for the appearance of the second such integral while raising the power of the relevant integrand's denominator by $1$. Let $x\in\mathbb{S}^{n-1}$ and $0<t<R$ and be given. We put
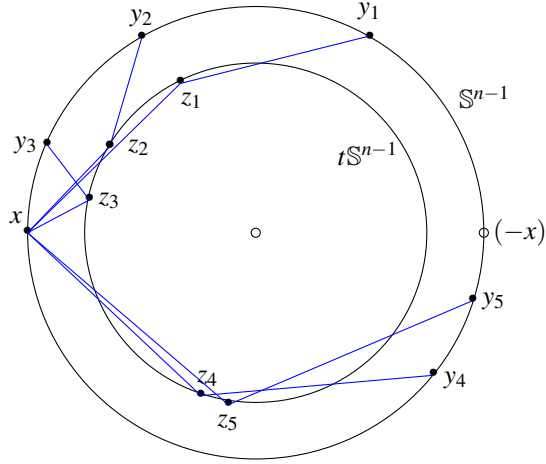
\begin{figure}\label{fig:Fubini2}
\centering
\begin{tikzpicture}[scale=3]
\draw (0,0) circle (0.75cm);
\draw (0,0) circle (1cm);
\node at (-1,0) {\textbullet};
\node[above] at (-1.05,0) {$x$};
\node at (0.5,0.86) {\textbullet};
\node at (-0.33,0.66) {\textbullet};
\node[below] at (-0.28,0.66) {$z_{1}$};
\node[above] at (0.5,0.9) {$y_{1}$};
\draw[blue] (-1,0) -- (-0.33,0.66);
\draw[blue] (0.5,0.87) -- (-0.33,0.66);
\node at (-0.5,0.86) {\textbullet};
\node[above] at (-0.5,0.87) {$y_{2}$};
\node at (-0.64,0.38) {\textbullet};
\node[right] at (-0.60,0.38) {$z_{2}$};
\draw[blue] (-1,0) -- (-0.64,0.38);
\draw[blue] (-0.64,0.38) -- (-0.5,0.86);
\node at (-0.64,0.38) {\textbullet};
\node at (-0.915,0.39) {\textbullet};
\node[left] at (-0.915,0.39) {$y_{3}$};
\node at (-0.729,0.145) {\textbullet};
\draw[blue] (-1,0) -- (-0.729,0.145);
\draw[blue] (-0.729,0.145) -- (-0.915,0.39);
\node[right] at (-0.729,0.145) {$z_{3}$};
\draw (1,0) circle (0.02cm);
\node[right] at (1,0) {$(-x)$};
\node[right] at (0.85,0.6) {$\mathbb{S}^{n-1}$};
\node[right] at (0.32,0.35) {$t\mathbb{S}^{n-1}$};
\node at (0.78,-0.63) {\textbullet};
\node at (-0.24,-0.72) {\textbullet};
\draw[blue] (-1,0) -- (-0.24,-0.72);
\draw[blue] (0.78,-0.63) -- (-0.24,-0.72);
\node[above] at (-0.21,-0.72) {$z_{4}$};
\node at (0.952,-0.3) {\textbullet};
\node at (-0.12,-0.76) {\textbullet};
\node[below] at (-0.12,-0.76) {$z_{5}$};
\draw[blue] (-1,0) -- (-0.12,-0.76);
\draw[blue] (0.952,-0.3)-- (-0.12,-0.76);
\node[right] at (0.952,-0.3) {$y_{5}$};
\node[right] at (0.785,-0.65) {$y_{4}$};
\draw (0,0) circle (0.02cm);
\end{tikzpicture}
\caption{The geometric situation in the proof of Theorem~\ref{thm:fubini1} in two dimensions for selected points $y=y_{i}$. Excluding the $\mathcal{H}^{n-1}$--nullset $(-x)$, we project the midpoints of the line segment of $x$ and $y_{i}$ onto $t\mathbb{S}^{n-1}$. This gives rise to the projections $z_{i}=\pi_{t}(x,y_{i})$, and we consequently integrate with respect to $t$ to have the requisite second radius integral emerging.}
\end{figure} 
\begin{align*}
\pi_{t}(x,y):= t\frac{x+y}{|x+y|},\qquad y\in\mathbb{S}^{n-1}\setminus\{-x\}, 
\end{align*}
which is the projection of the mid point of the line segment $[x,y]$ onto $\partial\!\ball_{t}$, cf. Figure~\ref{fig:Fubini2}. Hence,  the mapping $\Pi_{t,x}\colon \mathbb{S}^{n-1}\setminus\{-x\}\to t\mathbb{S}^{n-1}$ given by $\Pi_{t,x}(y):=\pi_{t}(x,y)$ is well--defined. We now estimate for arbitrary $x\in\mathbb{S}^{n-1}$ and $y\in\mathbb{S}^{n-1}\setminus\{-x\}$
\begin{align*}
|u(rx)-u(ry)|^{p} & \leq C ( |u(rx)-u(\pi_{t}(x,y))|^{p}+|u(ry)-u(\pi_{t}(x,y))|^{p}).
\end{align*}
Hence for all $0<a(x,y)\leq b(x,y)\leq R$, an integration with respect to $t\in [a(x,y),b(x,y)]$ yields
\begin{align}\label{eq:FubiniTscherp3}
\begin{split}
|u(rx)-u(ry)|^{p} & \leq C \dashint_{a(x,y)}^{b(x,y)}|u(rx)-u(\pi_{t}(x,y))|^{p}\dif t  + C\dashint_{a(x,y)}^{b(x,y)}|u(ry)-u(\pi_{t}(x,y))|^{p}\dif t.
\end{split}
\end{align}
At this point, fix $0<r\leq R$. We then choose $a(x,y):=r(1-\frac{|x-y|}{4})$ and $b(x,y):=r$. This particularly implies by $|x-y|\leq 2$ for all $x,y\in\mathbb{S}^{n-1}$
\begin{align}\label{eq:FubiniTscherp4}
|b(x,y)-a(x,y)|=r\frac{|x-y|}{4}\;\;\;\text{and}\;\;\;\frac{r}{2}\leq a(x,y)\leq b(x,y)=r. 
\end{align}
Now, an elementary geometric consideration (cf. Figure~\ref{fig:Fubini2}) yields that for all $x\in\mathbb{S}^{n-1}$ and $y\in\mathbb{S}^{n-1}\setminus\{-x\}$ there holds
\begin{align}\label{eq:projectioninequalityappendix}
\left\vert x- \frac{x+y}{|x+y|}\right\vert \leq |x-y|.
\end{align}
We thus have for all $0<t\leq r \leq R$
\begin{align}\label{eq:TabeaProjectionInequality}
|rx - \pi_{t}(x,y)| & = r\left\vert x-\frac{t}{r}\frac{x+y}{|x+y|}\right\vert \leq r\left\vert x-\frac{x+y}{|x+y|}\right\vert + r(1-\tfrac{t}{r})\leq r|x-y|+(r-t).
\end{align}
Combining \eqref{eq:FubiniTscherp1},  \eqref{eq:FubiniTscherp3} and \eqref{eq:FubiniTscherp4}, we then obtain 
\begin{align*}
(*) & \leq C\iint_{\mathbb{S}^{n-1}\times\mathbb{S}^{n-1}}\int_{0}^{R}(r^{n-1})^{2}\dashint_{r(1-\frac{|x-y|}{4})}^{r}\frac{|u(rx)-u(\pi_{t}(x,y))|^{p}}{|rx-ry|^{n+\vartheta p-1}}\dif t\dif r\dif\sigma_{x}\dif\sigma_{y} \\ & + C\iint_{\mathbb{S}^{n-1}\times\mathbb{S}^{n-1}}\int_{0}^{R}(r^{n-1})^{2}\dashint_{r(1-\frac{|x-y|}{4})}^{r}\frac{|u(ry)-u(\pi_{t}(x,y))|^{p}}{|rx-ry|^{n+\vartheta p-1}}\dif t\dif r\dif\sigma_{y}\dif\sigma_{x} =:\mathbf{I}+\mathbf{II}, 
\end{align*}
where we have used that for each $x\in\mathbb{S}^{n-1}$, $\{-x\}$ is a nullset for $\mathscr{H}^{n-1}$. The two integrals are symmetric in $x$ and $y$ (also note that $\pi_{t}(x,y)=\pi_{t}(y,x)$), and so it suffices to employ the desired estimate for one of these two integrals. We first estimate by virtue of the first part of \eqref{eq:FubiniTscherp4}
\begin{align*}
\mathbf{I} & \leq C \iint_{\mathbb{S}^{n-1}\times\mathbb{S}^{n-1}}\int_{0}^{R}(r^{n-1})^{2}\int_{r(1-\frac{|x-y|}{4})}^{r}\frac{|u(rx)-u(\pi_{t}(x,y))|^{p}}{|rx-ry|^{n+\vartheta p}}\dif t\dif r\dif\sigma_{x}\dif\sigma_{y} =: \mathbf{J}, 
\end{align*}
so that the desired second radius integral has emerged. To estimate $\mathbf{J}$, note that if $r(1-\frac{|x-y|}{4})\leq t \leq r$, then 
\begin{align}\label{eq:TabeaProjectionInequality1}
-t\leq r\Big(\frac{|x-y|}{4}-1\Big)\Rightarrow r-t\leq r\frac{|x-y|}{4}  \stackrel{\eqref{eq:TabeaProjectionInequality}}{\Longrightarrow} |rx-\pi_{t}(x,y)|\leq \frac{5}{4}r|x-y|.
\end{align}
Moreover, we note that for such $t$, we have 
\begin{align}\label{eq:TabeaProjectionInequality2}
r(1-\frac{|x-y|}{4})\leq t \leq r \Rightarrow (1-\frac{|x-y|}{4})\leq \frac{t}{r} \leq 1 \Rightarrow 1 \leq \frac{r}{t} \leq \frac{1}{1-\frac{|x-y|}{4}}\stackrel{|x-y|\leq 2}{\leq}2. 
\end{align}
We then estimate 
\begin{align*}
\mathbf{J} & \stackrel{\eqref{eq:TabeaProjectionInequality1}}{\leq} C \iint_{\mathbb{S}^{n-1}\times\mathbb{S}^{n-1}}\int_{0}^{R}(r^{n-1})^{2}\int_{r(1-\frac{|x-y|}{4})}^{r}\frac{|u(rx)-u(\pi_{t}(x,y))|^{p}}{|rx-\pi_{t}(x,y)|^{n+\vartheta p}}\dif t\dif r\dif\sigma_{x}\dif\sigma_{y}\\
& =  C \iint_{\mathbb{S}^{n-1}\times\mathbb{S}^{n-1}}\int_{0}^{R}r^{n-1}\int_{r(1-\frac{|x-y|}{4})}^{r}\frac{|u(rx)-u(\pi_{t}(x,y))|^{p}}{|rx-\pi_{t}(x,y)|^{n+\vartheta p}}\left(\frac{r}{t}\right)^{n-1}t^{n-1}\dif t\dif r\dif\sigma_{x}\dif\sigma_{y}\\
& \stackrel{\eqref{eq:TabeaProjectionInequality2}}{\leq}  C \iint_{\mathbb{S}^{n-1}\times\mathbb{S}^{n-1}}\int_{0}^{R}r^{n-1}\int_{r(1-\frac{|x-y|}{4})}^{r}\frac{|u(rx)-u(\pi_{t}(x,y))|^{p}}{|rx-\pi_{t}(x,y)|^{n+\vartheta p}}t^{n-1}\dif t\dif r\dif\sigma_{x}\dif\sigma_{y} =\mathbf{J}'. 
\end{align*}
At this point, we change variables and put $z:=(x+y)/|x+y|$. By the geometry of the map $\Pi_{1,x}$ and the fact that for any $y\in\mathbb{S}^{n-1}$ there holds $\mathbb{S}^{n-1}\setminus\{-x\}\ni y\mapsto (x+y)/|x+y|\in\mathbb{S}^{n-1}$, the relevant Jacobian is seen to be bounded. A routine estimation thus yields 
\begin{align*}
\mathbf{J}' & \leq  C \iint_{\mathbb{S}^{n-1}\times\mathbb{S}^{n-1}}\int_{0}^{R}r^{n-1}\int_{r(1-\frac{|x-y|}{4})}^{r}\frac{|u(rx)-u(tz)|^{p}}{|rx-tz|^{n+\vartheta p}}t^{n-1}\dif t\dif r\dif\sigma_{x}\dif\sigma_{z}\\
& \leq C \int_{0}^{R}\int_{0}^{R}\int_{\mathbb{S}^{n-1}}\int_{\mathbb{S}^{n-1}}\frac{|u(rx)-u(tz)|^{p}}{|rx-tz|^{n+\vartheta p}}\dif\sigma_{x}\dif\sigma_{z}t^{n-1}r^{n-1}\dif t\dif r\\
& \leq C \iint_{\ball(0,R)\times\ball(0,R)}\frac{|u(\widetilde{x})-u(\widetilde{y})|^{p}}{|\widetilde{x}-\widetilde{y}|^{n+\vartheta p}}\dif \widetilde{x}\dif\widetilde{y}, 
\end{align*}
the ultimate inequality being a direct consequence of a passage to polar coordinates; here, $C>0$ still only depends on $n,\vartheta$ and $p$.  This establishes \eqref{eq:Fubinifractional} and concludes step 1. 

\emph{Step 2. Existence of sufficiently many Lebesgue points.} Since we finally aim to apply step 1 for the particular choice $\vartheta=\theta$ and $p=\frac{n}{n-1+\theta}$, we record that $\vartheta p<1$ so that the traces of $\sobo^{\vartheta,p}$-maps are a priori not well-defined along $\partial\!\ball_{r}$; thus we assumed $u\in(\sobo^{s,p}\cap\hold)(\R^{n};\R^{n})$ in step 1 so that this issue did not arise. In order to make use of step 1 for $\bd$-maps $u$ by Proposition~\ref{prop:fractionalpoincareBD}, we start off by ensuring the explicit pointwise evaluability of $u$ $\mathscr{H}^{n-1}$-a.e. on $\mathscr{L}^{1}$-a.e. sphere centered at the origin. Without loss of generality, let $u\in\bd(\R^{n})$ and $0<R_{1}<R_{2}<\infty$ be arbitrary. Since $|\E u|$ is a finite Radon measure, the set $I:=\{t\in (R_{1},R_{2})\colon\; |\E u|(\partial\!\ball_{t})>0\}$ is at most countable. Hence $\mathscr{L}^{1}((R_{1},R_{2})\setminus I)=R_{2}-R_{1}$. Let $t\in (R_{1},R_{2})\setminus I$. Since $\partial\!\ball_{t}$ is a $\hold^{1}$--hypersurface, \eqref{eq:Kohnformula} yields
\begin{align}\label{eq:RNSQC}
\begin{split}
\E u\mres\partial\!\ball_{t} = (u^{+}-u^{-})\odot\nu_{\partial\!\ball_{t}}\mathcal{H}^{n-1}\mres\partial\!\ball_{t}
\end{split}
\end{align}
with the one--sided Lebesgue limits $u^{\pm}$ and the outer unit normal $\nu_{\partial\!\ball_{t}}$ to $\partial\!\ball_{t}$. Therefore,
\begin{align}
\int_{\partial\!\ball_{t}}|(u^{+}-u^{-})\odot\nu_{\partial\!\ball_{t}}|\dif\mathcal{H}^{n-1}=|\E u|(\partial\!\ball_{t})\stackrel{t\in (R_{1},R_{2})\setminus I}{=}0. 
\end{align}
This implies $|(u^{+}-u^{-})\odot\nu_{\partial\!\ball_{t}}|=0$ $\mathcal{H}^{n-1}$--a.e. on $\partial\!\ball_{t}$, and since $|a|\,|b|\leq \sqrt{2}|a\odot b|$ by \eqref{eq:symmetrictensorproduct} for all $a,b\in\R^{n}$, we conclude that $\widetilde{u}(x):=u^{+}(x)=u^{-}(x)$ holds for $\mathcal{H}^{n-1}$--a.e. $x\in\partial\!\ball_{t}$. Then, by \eqref{eq:tracesdash}, we have for $\mathscr{H}^{n-1}$-a.e. such $x\in\partial\!\ball_{t}$
\begin{align*}
\lim_{r\searrow 0} \dashint\limits_{\ball(x,r)\cap\{y\colon\langle y-x,\nu_{\partial\!\ball_{t}}\rangle \geq 0 \}}|u-\widetilde{u}(x)|\dif\mathscr{L}^{n}=\lim_{r\searrow 0} \dashint\limits_{\ball(x,r)\cap \{y\colon\langle y-x,\nu_{\partial\!\ball_{t}}\rangle < 0 \}}|u-\widetilde{u}(x)|\dif\mathscr{L}^{n}=0. 
\end{align*}
Since $\mathscr{L}^{n}(\ball(x,r)\cap \{y\colon\langle y-x,\nu_{\partial\!\ball_{t}}\rangle \gtrless 0 \})= \frac{1}{2}\mathscr{L}^{n}(\ball(x,r))$, this consequently yields 
\begin{align}\label{eq:claimSu}
\begin{split}
\lim_{r\searrow 0}\dashint_{\ball(x,r)}|u-\widetilde{u}(x)|\dif\mathscr{L}^{n}  = 0.
\end{split}
\end{align}
Hence, $\mathscr{H}^{n-1}$-a.e. $x\in\partial\!\ball_{t}$ is a Lebesgue point of $u$. In conclusion, $\mathcal{H}^{n-1}$--a.e. $x\in\partial\!\ball_{t}$ is a Lebesgue point for $u$ for $\mathscr{L}^{1}$--a.e. radius $t\in (R_{1},R_{2})$. 

In an intermediate step, we claim the following: Let $-\infty<a<b<\infty$ and let $J\subset (a,b)$ be a measurable subset of full Lebesgue measure, i.e., $\mathscr{L}^{1}((a,b)\setminus J)=0$. Then for every $g\in\lebe^{1}((a,b);\R_{\geq 0})$ there exists $\xi_{0}\in J$ which is a Lebesgue point for $g$ and satisfies 
\begin{align}\label{eq:contrameanvalue}
g^{*}(\xi_{0})=\lim_{r\searrow 0}\dashint_{(\xi_{0}-r,\xi_{0}+r)}g\dif x \leq \frac{2}{b-a}\int_{(a,b)}g\dif x,  
\end{align}
where $g^{*}$ is the precise representative of $g$. To see this, we note that $\mathscr{L}^{1}$--a.e. element of $J$ is a Lebesgue point for $g$, and hence the first equality in \eqref{eq:contrameanvalue} holds for $\mathscr{L}^{1}$-a.e. $\xi_{0}\in J$. Assume towards a contradiction that the overall claim is wrong. Then we find $g\in\lebe^{1}((a,b);\R_{\geq 0})$ such that for all $\xi_{0}\in J$ which are Lebesgue points for $g$ there holds 
\begin{align}\label{eq:contracontra}
g^{*}(\xi_{0})> \frac{2}{b-a}\int_{(a,b)}g(x)\dif x. 
\end{align}
Since this holds for $\mathscr{L}^{1}$--a.e. $\xi_{0}\in (a,b)$, we infer by integrating with respect to $\xi_{0}\in J$
\begin{align*}
2\int_{(a,b)}g(x)\dif x = \frac{2}{b-a}\int_{(a,b)}\int_{(a,b)}g(x)\dif x\dif\xi_{0}  \leq \int_{(a,b)}g(\xi_{0})\dif\xi_{0}. 
\end{align*}
By non--negativity of $g$, this implies $g\equiv 0$ $\mathscr{L}^{1}$--a.e. in $(a,b)$. This contradicts \eqref{eq:contracontra} and the proof of the intermediate claim is complete. 

\emph{Step 3. Conclusion.} Let now $0<\theta<1$ be arbitrary and put $p:=n/(n-1+\theta)$. Since the statement is local, we may assume that $u\in\bd_{\locc}(\R^{n})$ is compactly supported. For $\varepsilon>0$ consider the smooth approximations $u^{\varepsilon}(x):=\rho_{\varepsilon}*u(x)$, where $\rho\in\hold_{c}^{\infty}(\ball(0,1);[0,1])$ is a radial mollifier with $\|\rho\|_{\lebe^{1}(\ball(0,1))}=1$, and $\rho_{\varepsilon}(x):=\varepsilon^{-n}\rho(\frac{x}{\varepsilon})$ is its $\varepsilon$-rescaled variant. Based on Proposition~\ref{prop:fractionalpoincareBD}~\ref{item:poincare2}, we choose a rigid deformation $\alpha_{\varepsilon}\in\mathscr{R}(\ball_{r})$ such that, with $u_{\alpha_{\varepsilon}}^{\varepsilon}:=u^{\varepsilon}-\alpha_{\varepsilon}$, 
\begin{align}\label{eq:scalingSQC2}
\left(\dashint_{\ball_{r}}\int_{\ball_{r}}\frac{|u_{\alpha_{\varepsilon}}^{\varepsilon}(x)-u_{\alpha_{\varepsilon}}^{\varepsilon}(y)|^{p}}{|x-y|^{n+\theta p}}\dif x \dif y\right)^{\frac{1}{p}}\leq Cr^{1-\theta}\dashint_{\ball_{r}}|\E u^{\varepsilon}|. 
\end{align}
By construction of $\alpha_{\varepsilon}$ (cf.~\eqref{eq:PoincareonBD} and Remark~\ref{rem:stability}), $\|\alpha_{\varepsilon}\|_{\lebe^{1}(\ball_{r};\R^{n})}\leq c \|u^{\varepsilon}\|_{\lebe^{1}(\ball_{r};\R^{n})}\to \|u\|_{\lebe^{1}(\ball_{r};\R^{n})}$. Since $\mathscr{R}(\ball_{r})$ is finite dimensional, we hence find a sequence $(\varepsilon_{j})\subset\R_{>0}$ and $\alpha\in\mathscr{R}(\ball_{r})$ such that $\alpha_{\varepsilon_{j}}\to \alpha$ in $\mathscr{R}(\ball_{r})$. Therefore, by Fatou's lemma and \eqref{eq:Fubinifractional} with $R=r$ in the first step, 
\begin{align}\label{eq:splitfubini1}
\begin{split}
\int_{s}^{r}\iint_{\partial\!\ball_{t}\times\partial\!\ball_{t}}\frac{|u_{\alpha}^{*}(x)-u_{\alpha}^{*}(y)|^{p}}{|x-y|^{n+\theta p-1}}\dif\sigma_{x}\dif\sigma_{y}\dif t & \leq C\liminf_{j\to\infty}\int_{\ball_{r}}\int_{\ball_{r}}\frac{|u_{\alpha_{\varepsilon_{j}}}^{\varepsilon_{j}}(x)-u_{\alpha_{\varepsilon_{j}}}^{\varepsilon_{j}}(y)|^{p}}{|x-y|^{n+\theta p}}\dif x \dif y \\ 
& \!\!\!\stackrel{\eqref{eq:scalingSQC2}}{\leq} C\liminf_{j\to\infty} r^{n}\Big(r^{1-\theta}\dashint_{\ball_{r}}|\E u^{\varepsilon_{j}}|\Big)^{p}\\
& \,= Cr^{n}\Big(r^{1-\theta}\dashint_{\overline{\ball}_{r}}|\E u|\Big)^{p}
\end{split}
\end{align}
with the precise representative $u_{\alpha}^{*}$ of $u_{\alpha}$. Put $J:=\{t\in (s,r)\colon\;|\E u|(\partial\!\ball_{t})=0\}$ and define $\lambda\colon (s,r)\to\R_{\geq 0}$ by 
\begin{align*}
\lambda(t):=\iint_{\partial\!\ball_{t}\times\partial\!\ball_{t}}\frac{|u_{\alpha}^{*}(x)-u_{\alpha}^{*}(y)|^{p}}{|x-y|^{n+\theta p-1}}\dif\sigma_{x}\dif\sigma_{y},\qquad t\in J
\end{align*}
and $\lambda(t)=0$ otherwise. With $J$ from above, step 2 implies the existence of some $t\in J$ such that 
\begin{align*}
\lambda(t) \leq \frac{2}{r-s}\int_{(s,r)}\lambda(t)\dif t \stackrel{\eqref{eq:splitfubini1}}{\leq} C\frac{r^{n}}{r-s}\Big(r^{1-\theta}\dashint_{\overline{\ball}_{r}}|\E u|\Big)^{p}
\end{align*}
which, upon rewriting the left-hand side of the previous inequality in terms of $u_{\alpha}^{*}$, yields
\begin{align*}
\Big(\dashint_{\partial\!\ball(0,t)}\int_{\partial\!\ball(0,t)} \frac{|u_{\alpha}^{*}(x)-u_{\alpha}^{*}(y)|^{p}}{|x-y|^{n+\theta p-1}}\dif\sigma_{x}\dif\sigma_{y}\Big)^{\frac{1}{p}}
\leq C\frac{r^{\frac{n}{p}}r^{1-\theta}}{t^{\frac{n-1}{p}}(r-s)^{\frac{1}{p}}}
\dashint_{\overline{\ball}_{r}}|\E u|.
\end{align*}
It is clear that $C>0$ does not depend on $u$ nor $s,r$, and so we arrive at \eqref{eq:fubinispheremainestPR}; recall that $\mathscr{H}^{n-1}$-a.e. $x\in\partial\!\ball_{t}$ is a Lebesgue point for $u_{\alpha}$. The proof is complete. 
\end{proof}
\begin{remark}\label{rem:VersusBV}
In the $\bv$-case, a Fubini-type property can be established by noting that for $u\in\bv(\R^{n};\R^{N})$, the tangential derivative $\partial_{\tau}u$ on $\mathscr{L}^{1}$-almost every sphere $\partial\!\ball(0,t)$ is a finite Radon measure, too. This is  discussed and utilised in \cite{AFP} and \cite{GK2}. By \textsc{Ornstein}'s Non-Inequality, we see no argument to ensure that for generic maps $u\in\bd(\Omega)$, $\partial_{\tau}u$ should be a Radon measure on even sufficiently many spheres. Also note that, by the very nature of the objects considered, any sort of 'symmetric tangential derivative' does not make sense. As to step 1 in the above proof, Fubini-type theorems for maps $u\in\text{B}{_{p,q}^{s}}$ and $u\in\text{F}{_{p,q}^{s}}$ have been given by \textsc{Triebel} in the case where spheres are replaced by affine subspaces of $\R^{n}$, cf. \cite[Chpt.~2.5.13]{Triebel1}. To reduce to this setting by local coordinate transformations, transforming the left hand side of \eqref{eq:Fubinifractional} gives rise to additional localisation terms on the right hand side. It is not clear to us how to control these to obtain the requisite form of the estimate, an issue which does not arise in the above proof. 
\end{remark}
\section{The linear growth case: Proof of Theorem~\ref{thm:main1}}\label{sec:main1}
This section is devoted to the proof of Theorem~\ref{thm:main1}. Toward this objective, we aim to compare the given generalised minimiser with a suitable $\A$-harmonic approximations via linearisation. Since linear elliptic problems subject to $\lebe^{1}$-boundary data are, in general, ill posed, this can only be achieved on \emph{good} balls where the boundary traces of $u$ share higher fractional differentiability. In this way, the corresponding $\A$-harmonic approximation will be well-defined; note that this unclear for \emph{general} balls on whose boundaries a given $\bd$-minimiser $u$ is only known to possess traces in $\lebe^{1}$. Consequently, this is where the Fubini-type property of $\bd$-maps as given in the last section enters. To arrive at the desired excess decay, we shall estimate a $V$-function-type distance of $u$ to its $\A$-harmonic approximation in terms of a \emph{superlinear} power of the excess, cf.~Proposition~\ref{prop:improvement}. Postponing the precise discussion to Remark~\ref{rem:discussAharmonicity}, a linear instead of superlinear power of the excess -- which would come out by easier means -- is not sufficient to conclude the excess decay. In conjunction with the Caccioppoli inequality of the second kind to be proved in Section~\ref{sec:Cacc}, we will then show in Section~\ref{sec:excess} that the estimates gathered so far for \emph{good} balls are in fact sufficient to conclude a preliminary excess for all relevant balls, i.e., those on which the excess does not exceed a certain constant.

In order to implement the linearisation strategy in the main part of the partial regularity proof, we introduce for $f\colon\R_{\sym}^{n\times n}\to\R$ satisfying \ref{item:reg1}--\ref{item:reg3} from Theorem~\ref{thm:main1} and $w\in\rsym$ the integrands 
\begin{align}\label{eq:shift}
f_{w}(\xi) & := f(\xi+w)-f(w)-\langle f'(w),\xi\rangle,\qquad\xi\in\rsym, 
\end{align}
and remind the reader of the function $V\colon\R_{\sym}^{n\times n}\to\R$ given by $V(\xi):=\sqrt{1+|\xi|^{2}}-1$.
\begin{lemma}\label{lem:auxSQC}
For all $w,z\in\R_{\sym}^{n\times n}$ we have (with an obvious interpretation for $w=0$ or $z=0$)
\begin{align}\label{eq:VestimatesSQCAuxB}
\langle V''(w)z,z\rangle = \frac{1+|w|^{2}-|w|^{2}\Big(\frac{w}{|w|}\cdot\frac{z}{|z|}\Big)^{2}}{(1+|w|^{2})^{\frac{3}{2}}}\;\;\;\text{and}\;\;\;V_{w}(z)\geq \frac{1}{16}\frac{V(z)}{(1+|w|^{2})^{\frac{3}{2}}}.
\end{align}
Moreover, for each $m>0$ there exists a constant $c=c(m)\in [1,\infty)$ with the following properties: If $f\colon\R_{\sym}^{n\times n}\to\R$ satisfies hypotheses \ref{item:reg1}--\ref{item:reg3} from Theorem~\ref{thm:main1}, then for all $z\in\R_{\sym}^{n\times n}$ and all $w\in\R_{\sym}^{n\times n}$ with $|w|\leq m$ there holds
\begin{itemize}
\item[\emph{(i)}] $|f_{w}(z)|\leq cLV(z)$, 
\item[\emph{(ii)}] $|f'_{w}(z)|\leq cL\min\{|z|,1\}$, 
\item[\emph{(iii)}] $|f''_{w}(0)z-f'_{w}(z)|\leq cLV(z)$.
\end{itemize}
and for all $w\in\R_{\sym}^{n\times n}$ and open balls $\ball\subset\R^{n}$ we have
\begin{align}\label{eq:auxGKconsequenceSQC}
\frac{\ell}{c}\int_{\ball}V(\sg(\varphi))\dif x \leq \int_{\ball}f_{w}(\sg(\varphi))\dif x\qquad\text{for all}\;\varphi\in\ld_{0}(\ball). 
\end{align} 
\end{lemma}
\begin{proof}
All assertions apart from \eqref{eq:auxGKconsequenceSQC} are taken from \cite[Lems.~4.1,~4.2]{GK2}. To see \eqref{eq:auxGKconsequenceSQC}, let $\ball\subset\R^{n}$ be an open ball and let $\varphi\in\ld_{0}(\ball), w\in\R_{\sym}^{n\times n}$ with $|w|\leq m$ be arbitrary. With condition \ref{item:reg3} from Theorem~\ref{thm:main1} in the third step we deduce
\begin{align*}
\int_{\ball}&\frac{V(\sg(\varphi))\dif x}{(1+|w|^{2})^{\frac{3}{2}}} \stackrel{\eqref{eq:VestimatesSQCAuxB}}{\leq} 16 \int_{\ball}V_{w}(\sg(\varphi))\dif x  = 16\Big(\int_{\ball}V(w+\sg(\varphi))-V(w)\dif x-\underbrace{\int_{\ball}\langle V'(w),\sg(\varphi)\rangle\dif x}_{=0}\Big) \\
& \leq \frac{16}{\ell} \int_{\ball}f(w+\sg(\varphi))-f(w)\dif x -\frac{16}{\ell}\underbrace{\int_{\ball}\langle f'(w),\sg(\varphi)\rangle\dif x}_{=0} = \frac{16}{\ell}\int_{\ball}f_{w}(\sg(\varphi))\dif x. 
\end{align*}
Here the underbraced integrals vanish by the Gau\ss --Green theorem and the fact that $\varphi|_{\partial\!\ball}=0$. Noting that $|w|\leq m$, \eqref{eq:auxGKconsequenceSQC} follows and the proof is complete.
\end{proof}

\subsection{Caccioppoli inequality of the second kind}\label{sec:Cacc}
In this section we give the requisite form of the Caccioppoli inequality of the second kind, and it is here where the $\bd$-minimality crucially enters. However, different from other proof schemes, let us emphasize that this inequality will \emph{not} be used to deduce higher integrability of generalised minima; in fact, \textsc{Gehring}'s lemma does not  quite seem to fit into the linear growth framework, cf. Section~\ref{sec:extensions} below for a discussion. From now on, we tacitly suppose that $f\colon\R_{\sym}^{n\times n}\to\R$ satisfies \ref{item:reg1}-\ref{item:reg3} from Theorem~\ref{thm:main1} without further mentioning unless it is explicitely stated otherwise.
\begin{proposition}[of Caccioppoli-type]\label{prop:CaccQC}
Let $m>0$. Then there exists a constant $c=c(m,n,\frac{L}{\ell})\in [1,\infty)$ such that if $a\colon\R^{n}\to\R^{n}$ is an affine-linear mapping with $|\E a|\leq m$ and $\ball=\ball(x_{0},R)\Subset\Omega$ a ball, then there holds 
\begin{align}\label{eq:CaccQC}
\int_{\ball(x_{0},\frac{R}{2})}V(\E\,(u-a))\leq c \int_{\ball(x_{0},R)}V\Big(\frac{u-a}{R}\Big)\dif x
\end{align}
for every local $\bd$-minimiser $u\in\bd_{\locc}(\Omega)$. 
\end{proposition}
\begin{proof}
The proof evolves around a scheme for establishing Caccioppoli--type inequalities in the quasiconvex setting originally due to \textsc{Evans}~\cite[Lem.~3.1]{Ev}. Recalling the definition of the shifted integrands, cf.~\eqref{eq:shift}, we put $\widetilde{f}:=f_{\sg(a)}$ and $\widetilde{u}:=u-a$. We then record that $\widetilde{u}$ is a local minimiser the functional 
\begin{align*}
\mathcal{F}[v]:=\int_{\Omega}\widetilde{f}(\E v)
\end{align*}
over $\bd_{\locc}(\Omega)$. Let $\frac{R}{2}<r<s<R$ be arbitrary and choose a cut--off function $\rho\in\hold_{c}^{1}(\ball(x_{0},s);[0,1])$ with $\mathbbm{1}_{\ball(x_{0},r)}\leq\rho\leq\mathbbm{1}_{\ball(x_{0},s)}$ and $|\nabla\rho|\leq \frac{2}{s-r}$. We then define $\varphi :=\rho\widetilde{u}$ and $\psi:=(1-\rho)\widetilde{u}$, so that $\widetilde{u}=u-a=\varphi+\psi$. Before we continue, let us remark that with $\ell>0$ from hypothesis \ref{item:reg3} of Theorem~\ref{thm:main1} and $c=c(m)>0$, 
\begin{align}\label{eq:Caccaux1}
\frac{\ell}{c}\int_{\ball(x_{0},s)}V(\E\varphi) \leq \int_{\ball(x_{0},s)}\widetilde{f}(\E\varphi).
\end{align}
To see this inequality, note $\varphi|_{\partial\!\ball(x_{0},s)}=0$ and hence we find an approximating sequence $(\varphi_{k})\subset\hold_{c}^{\infty}(\ball(x_{0},s);\R^{n})$ which converges in the (symmetric) area--strict sense on $\ball(x_{0},s)$ to $\varphi$ as $k\to\infty$. From Lemma~\ref{lem:auxSQC}, cf.~\eqref{eq:auxGKconsequenceSQC}, we then deduce \eqref{eq:Caccaux1} with $\varphi$ replaced by $\varphi_{k}$. In the resulting inequality, by definition of (symmetric) area-strict convergence, the left-hand side converges to $\frac{\ell}{c}\int_{\ball(x_{0},s)} V(\E\varphi)$. For the right-hand side we invoke the continuity result for symmetric rank-one convex functionals with respect to symmetric area-strict convergence, cf.~Lemma~\ref{lem:symareastrict}. By symmetric area-strict convergence and the fact that symmetric quasiconvexity implies symmetric rank-1-convexity, we hereby obtain \eqref{eq:Caccaux1}.

Consequently, using (generalised) minimality of $\widetilde{u}$ with respect to its own boundary values and $\widetilde{u}|_{\partial\!\ball(x_{0},s)}=\psi|_{\partial\!\ball(x_{0},s)}$ in the second step, we obtain
\begin{align*}
\frac{\ell}{c}\int_{\ball(x_{0},r)}V(\E\widetilde{u}) \leq \frac{\ell}{c}\int_{\ball(x_{0},s)}V(\E\varphi) & \leq  \int_{\ball(x_{0},s)}\widetilde{f}(\E\widetilde{u})+\int_{\ball(x_{0},s)}(\widetilde{f}(\E\varphi)-\widetilde{f}(\E\widetilde{u}))\qquad(\text{by \eqref{eq:Caccaux1}})\\
& \leq  \int_{\ball(x_{0},s)}\widetilde{f}(\E\psi)+\int_{\ball(x_{0},s)}(\widetilde{f}(\E\varphi)-\widetilde{f}(\E\widetilde{u}))\\
& \leq  \int_{\ball(x_{0},s)\setminus\ball(x_{0},r)}\widetilde{f}(\E\psi) +\int_{\ball(x_{0},s)\setminus\ball(x_{0},r)}(\widetilde{f}(\E\varphi)-\widetilde{f}(\E\widetilde{u})), \\
& =: \mathbf{I}+\mathbf{II}, 
\end{align*}
where the last inequality holds as $\varphi,\widetilde{u}$ coincide on $\ball(x_{0},r)$. Then, by Lemmas~\ref{lem:auxSQC}(i) and \ref{lem:auxVSQC}, 
\begin{align*}
\mathbf{I}  \leq cL\int_{\ball(x_{0},s)\setminus\ball(x_{0},r)}V(\E\psi)  
& = cL\int_{\ball(x_{0},s)\setminus\ball(x_{0},r)}V(\Big((1-\rho)\frac{\dif\E\widetilde{u}}{\dif |\E\widetilde{u}|}\Big)|\E\widetilde{u}|-(\nabla\rho\odot\widetilde{u})\mathscr{L}^{n}) \\
& \leq 2cL \int_{\ball(x_{0},s)\setminus\ball(x_{0},r)}V(\E\widetilde{u})+8cL\int_{\ball(x_{0},s)}V\Big(\frac{\widetilde{u}}{s-r}\Big)\dif x
\end{align*}
On the other hand, we similarly find
\begin{align*}
\mathbf{II} & \leq  \int_{\ball(x_{0},s)\setminus\ball(x_{0},r)}\widetilde{f}(\Big(\rho\frac{\dif\E\widetilde{u}}{\dif|\E\widetilde{u}|}\Big)|\E\widetilde{u}| + \nabla\rho\odot\widetilde{u}\mathscr{L}^{n})-\widetilde{f}(\E\widetilde{u})\\
& \leq 3cL \int_{\ball(x_{0},s)\setminus\ball(x_{0},r)}V(\E\widetilde{u})+8cL\int_{\ball(x_{0},s)}V\Big(\frac{\widetilde{u}}{s-r}\Big)\dif x.
\end{align*}
Therefore, gathering estimates, we find 
\begin{align*}
\frac{\ell}{c}\int_{\ball(x_{0},r)}V(\E\widetilde{u})\leq  16cL\int_{\ball(x_{0},s)\setminus\ball(x_{0},r)}V(\E\widetilde{u}) + 16cL\int_{\ball(x,s)}V\Big(\frac{\widetilde{u}}{s-r}\Big)\dif x. 
\end{align*}
We now apply \textsc{Widman}'s \emph{hole--filling trick} and hence add $16cL\int_{\ball(x_{0},r)}V(\E\widetilde{u})$ to both sides of the previous inequality and divide the resulting inequality by $(\frac{\ell}{c}+16cL)$. In consequence, letting $\theta:=16cL/(\frac{\ell}{c}+16cL)$, we have $0<\theta<1$ and get
\begin{align*}
\int_{\ball(x_{0},r)}V(\E\widetilde{u})\leq \theta \int_{\ball(x_{0},s)}V(\E\widetilde{u}) + \theta\int_{\ball(x_{0},R)}V\Big(\frac{\widetilde{u}}{s-r}\Big)\dif x. 
\end{align*}
From here the conclusion is immediate by Lemma~\ref{lem:iterationlemmaQC}. The proof is complete. 
\end{proof}
\subsection{Estimating the distance to the $\A$-harmonic approximation} 
In this section we present the key result that allows to deduce the requisite excess decay needed in the proof of Theorem~\ref{thm:main1}. 
Here our strategy is as follows: Letting $m>0$ be a given number and $a\colon\R^{n}\to\R^{n}$ an affine-linear map with 
$|\E a|\leq m$, we first establish an improved estimate for the $V$-function type distance of $\widetilde{u}:=u-a$ to a suitable $\A$-harmonic approximation on \emph{good} balls $\ball(x_{0},R_{0})\Subset\Omega$. Here \emph{goodness} refers to balls on whose boundaries $\partial\!\ball(x_{0},R_{0})$ the map $\widetilde{u}$ is of class $\sobo^{\frac{1}{n+1},\frac{n+1}{n}}(\partial\!\ball(x_{0},R_{0});\R^{n})$. This is accomplished in Proposition~\ref{prop:improvement}. By the Fubini-type property of $\bd$-maps, it is then clear that whenever $x_{0}\in\Omega$ is fixed, then $\mathscr{L}^{1}$-a.e. radius $R_{0}\in (x_{0},\frac{1}{2}\dist(x_{0},\partial\Omega))$ will qualify as a good radius. It shall then be the aim of the subsequent section to justify to have the relevant estimates on good balls to conclude a preliminary excess decay. We begin with the following proposition, making Lemma~\ref{lem:toplinear} available for the sequel.
\begin{proposition}\label{thm:mazsha}
Let $\A\in\mathbb{S}(\rsym)$ be a strongly symmetric rank-one convex bilinear form, i.e., $\A$ satisfies for two constants $\nu_{1},\nu_{2}>0$ and all $a,b\in\R^{n}$, $z_{1},z_{2}\in\rsym$
\begin{align}\label{eq:TscherpelCenter10}
\nu_{1}|a\odot b|^{2}\leq \A [a\odot b,a\odot b]\;\;\;\text{and}\;\;\;|\A[z_{1},z_{2}]|\leq \nu_{2}|z_{1}|\,|z_{2}|.
\end{align}
Let $Lv:=-\di(\A\sg(v))$, where $\A$ is identified with its representing matrix in $\R^{(n\times n)\times (n\times n)}$. Then for each $k\in\mathbb{N}$, $1<q<\infty$ and any open ball $\ball\subset\R^{n}$, the mapping 
\begin{align}\label{eq:mapppropsMazSha}
\Phi\colon\sobo^{k,q}(\ball;\R^{n})\ni u \mapsto (L(u),\trace_{\partial\!\ball}u)\in \sobo^{k-2,q}(\ball;\R^{n})\times\sobo^{k-\frac{1}{q},q}(\partial\!\ball;\R^{n})
\end{align} 
is a topological isomorphism. Moreover, if $u\in\ld(\Omega)$ satisfies $Lu=0$ in $\mathscr{D}'(\Omega;\R^{n})$, then there holds $u\in\hold^{\infty}(\Omega;\R^{n})$ and 
\begin{align}\label{eq:TscherpelCenter00}
\sup_{\ball(x_{0},\frac{R}{2})}|\nabla u-A| + R\sup_{\ball(x_{0},\frac{R}{2})}|\nabla^{2}u|\leq C\dashint_{\ball(x_{0},R)}|\nabla u - A|\dif x
\end{align}
for all $A\in\rsym$ and balls $\ball(x_{0},R)\Subset\Omega$, where $C=C(n,\nu_{1},\nu_{2})>0$ is a constant. 
\end{proposition}
\begin{proof}
We reduce to Lemma~\ref{lem:toplinear} and define $\mathscr{A}\in\mathbb{S}(\R^{n\times n})$ by $\mathscr{A}[z_{1},z_{1}]:=\mathbb{A}[z_{1}^{\sym},z_{2}^{\sym}]$, $z_{1},z_{2}\in\R^{n\times n}$. Then \eqref{eq:TscherpelCenter10} in conjunction with \eqref{eq:symmetrictensorproduct} yields 
\begin{align*}
|a\otimes b|^{2} \leq |a|^{2}|b|^{2}\leq 2|a\odot b|^{2} \leq \frac{2}{\nu_{1}}\mathbb{A}[a\odot b,a\odot b] = \frac{2}{\nu_{1}}\mathscr{A}[a\otimes b,a\otimes b]. 
\end{align*}
Hence $\mathscr{A}\in\mathbb{S}(\R^{n\times n})$ satisfies the hypotheses of Lemma~\ref{lem:toplinear} with $\lambda=\frac{\nu_{1}}{2}$. With the above terminology, we then have $\int_{\Omega}\mathbb{A}[\sg(u),\sg(\varphi)]\dif x = \int_{\Omega}\mathscr{A}[\nabla u,\nabla\varphi]\dif x$ for all $\varphi\in\hold_{c}^{\infty}(\Omega;\R^{n})$. Thus, $\Phi$ given by \eqref{eq:mapppropsMazSha} is a topological isomorphism by Lemma~\ref{lem:toplinear}. The additional estimate \eqref{eq:TscherpelCenter00} then follows similarly, now invoking the second part of Lemma~\ref{lem:toplinear}. The proof is complete. 
\end{proof}
We now come to the $\A$-harmonic approximation. Recalling that the number $m>0$ and the affine-linear map $a\colon\R^{n}\to\R^{n}$ with $|\E a|\leq m$ are assumed fixed throughout, we put 
\begin{center}
$\widetilde{u}:=u-a.$
\end{center}
Given a ball $\ball=\ball(x_{0},R)\Subset\Omega$ and $u\in\bd(\Omega)$ with $\widetilde{u}|_{\partial\!\ball}\in\sobo^{\frac{1}{n+1},\frac{n+1}{n}}(\partial\!\ball;\R^{n})$, we consider the strongly symmetric rank-one system 
\begin{align}\label{eq:lin1}
\begin{cases} 
-\di(\A\sg(h))=0&\;\text{in}\;\ball,\\
h=\widetilde{u}&\;\text{on}\;\partial\!\ball, 
\end{cases}
\end{align}
where $\A:=\widetilde{f}''(0)$ with $\widetilde{f}:=f_{\sg(a)}$, cf.~\eqref{eq:shift}; note that, if $f$ satisfies hypothesis \ref{item:reg3} from Theorem~\ref{thm:main1}, it is routine to check that $\A$ is a strongly symmetric rank-one-convex bilinear form. Put $k=1$ and $q:=1+\frac{1}{n}$. Then $k-\frac{1}{q}=\frac{1}{n+1}$, and in this situation Theorem~\ref{thm:mazsha} yields that there exists a unique $h\in\sobo^{1,1+1/n}(\ball(x_{0},R);\R^{n})$ solving \eqref{eq:lin1}. We now have the following 
\begin{proposition}\label{prop:improvement}
Suppose that $f\in\hold(\rsym)$ satisfies \ref{item:reg1}--\ref{item:reg3} from Theorem~\ref{thm:main1} and let $1<q<\frac{n+1}{n}$, $m>0$ be given. Then there exists a constant $C=C(m,n,q,L,\ell)>0$ with the following property: 
Suppose that $u\in\bd_{\locc}(\Omega)$ is a local $\bd$-minimiser for $F$ and $\ball=\ball(x_{0},R)\Subset\Omega$ is an open ball such that $u|_{\partial\!\ball}\in \sobo^{\frac{1}{n+1},\frac{n+1}{n}}(\partial\!\ball;\R^{n})$. Moreover, let $a\colon\R^{n}\to\R^{n}$ be an affine--linear mapping with $|\E a|\leq m$ and denote $h$ the unique solution of the linear system \eqref{eq:lin1} with $\widetilde{u}:=u-a$. Then there holds 
\begin{align}\label{eq:absolutemain}
\dashint_{\ball(x_{0},R)}V\Big(\frac{\widetilde{u}-h}{R}\Big)\dif x \leq C \Big(\dashint_{\ball(x_{0},R)}V(\E\widetilde{u}) \Big)^{q}. 
\end{align}
\end{proposition}
\begin{proof}
We fix a ball $\ball(x_{0},R)\Subset\Omega$ such that the hypotheses of the proposition are in action. The proof then evolves in three steps: 

\emph{Step 1. Ekeland approximation.} To avoid manipulations on measures, we first employ an approximation procedure that allows us to work with $\ld$- instead of $\bd$-maps. To this end, let $\delta>0$ be arbitrary but fixed. Then we apply the area-strict approximation of Lemma~\ref{lem:smooth} to find $\widetilde{w}_{\delta}\in\ld_{\widetilde{u}}(\ball(x_{0},R)):=\widetilde{u}+\ld_{0}(\ball(x_{0},R))$ such that 
\begin{align}\label{eq:EkePrep}
\begin{split}
&\dashint_{\ball(x_{0},R)}\left\vert\frac{\widetilde{u}-\widetilde{w}_{\delta}}{R}\right\vert\dif x + \left\vert\dashint_{\ball(x_{0},R)}V(\E\widetilde{u})-\dashint_{\ball(x_{0},R)}V(\sg(\widetilde{w}_{\delta}))\dif x\right\vert \leq \delta^{2},\\
& \dashint_{\ball(x_{0},R)}\widetilde{f}(\sg(\widetilde{w}_{\delta}))\dif x \leq  \dashint_{\ball(x_{0},R)}\widetilde{f}(\E\widetilde{u})+\delta^{2}, 
\end{split}
\end{align}
where the dash is understood with respect to the Lebesgue measure $\mathscr{L}^{n}$. Note that we can assume without loss of generality that $\widetilde{w}_{\delta}\in\ld(\ball(x_{0},R))$ since $\widetilde{u}$ only enters in the definition of $\ld_{\widetilde{u}}(\ball(x_{0},R))$ through prescribing the traces. However, as $\ld(\ball(x_{0},R))$ and $\bd(\ball(x_{0},R))$ have the same trace space on $\partial\!\ball(x_{0},R)$, we can find a $\ld$-map that has the same boundary traces on $\partial\!\ball(x_{0},R)$ and then proceed as before. Crucially, $(\ld_{\widetilde{u}}(\ball(x_{0},R)),d_{\sym})$ is a complete metric space, where $d_{\sym}(v_{1},v_{2}):=\|\sg(v_{1}-v_{2})\|_{\lebe^{1}(\ball(x_{0},R);\rsym)}$ is the symmetric gradient-$\lebe^{1}$-metric. It is then routine to check that all the requirements for the Ekeland variational principle, Lemma~\ref{lem:EkeLemma}, are satisfied; in particular, by \eqref{eq:nogap} from the appendix, the local $\bd$-minimality of $\widetilde{u}$ for $v\mapsto \int \widetilde{f}(\E v)$ gives 
\begin{align*}
 \dashint_{\ball(x_{0},R)}\widetilde{f}(\sg(\widetilde{w}_{\delta}))\dif x \leq \inf_{w\in\ld_{\widetilde{u}}(\ball(x_{0},R))}\dashint_{\ball(x_{0},R)}\widetilde{f}(\sg(w)) + \delta^{2},
\end{align*}
We deduce that there exists a mapping $\widetilde{v}\in\ld_{\widetilde{u}}(\ball(x_{0},R))$ which satisfies 
\begin{align}\label{eq:EkeMAIN}
\begin{split}
& \int_{\ball(x_{0},R)}\widetilde{f}(\sg(\widetilde{v}))\dif x \leq \int_{\ball(x_{0},R)}\widetilde{f}(\sg(\widetilde{w}_{\delta}))\dif x,\\
& \dashint_{\ball(x_{0},R)}\left\vert\frac{\widetilde{v}-\widetilde{w}_{\delta}}{R}\right\vert\dif x + \dashint_{\ball(x_{0},R)}|\sg(\widetilde{v})-\sg(\widetilde{w}_{\delta})|\dif x\leq (1+c_{\poinc}) \delta,\\
& \int_{\ball(x_{0},R)}\widetilde{f}(\sg(\widetilde{v}))\dif x \leq \int_{\ball(x_{0},R)}\widetilde{f}(\sg(\widetilde{\varphi}))\dif x + \delta \int_{\ball(x_{0},R)}|\sg \big(\widetilde{v}-\widetilde{\varphi} \big)|\dif x
\end{split}
\end{align}
for all $\widetilde{\varphi}\in\ld_{\widetilde{u}}(\ball(x_{0},R))$, where $c_{\poinc}>0$ is an arbitrary but fixed constant for the Poincar\'{e} inequality in $\ld_{0}(\ball(x_{0},R))$; note that the above inequality scales correctly and hence $c_{\poinc}>0$ is in fact independent of $R$. Working from here, we obtain 
\begin{align}\label{eq:EkeMAIN1}
\left\vert\,\int_{\ball(x_{0},R)}\langle\widetilde{f}'(\sg(\widetilde{v})),\sg(\varphi)\rangle\dif x\right\vert \leq \delta\int_{\ball(x_{0},R)}|\sg(\varphi)|\dif x
\end{align}
for all $\varphi\in\ld_{0}(\ball(x_{0},R))$ and 
\begin{align}\label{eq:EkeMAIN2}
\left\vert\,\int_{\ball(x_{0},R)}\langle \widetilde{f}''(0)\sg(\widetilde{v}),\sg(\varphi)\rangle\dif x \right\vert \leq \int_{\ball(x_{0},R)}(cL V(\sg(\widetilde{v}))+\delta)|\sg(\varphi)|\dif x
\end{align}
for all $\varphi\in\sobo_{0}^{1,\infty}(\ball(x_{0},R);\R^{n})$. Indeed, for every $\theta>0$, $\widetilde{\varphi}_{\theta}^{\pm}:=\widetilde{v}\pm\theta\varphi$ qualifies as a competitor in $\eqref{eq:EkeMAIN}_{3}$. Hence, 
\begin{align*}
\left\vert\,\int_{\ball(x_{0},R)} \frac{\widetilde{f}(\sg(\widetilde{v}))-\widetilde{f}(\sg(\widetilde{v}\pm\theta\varphi))}{\theta}\dif x\right\vert \leq \delta\int_{\ball(x_{0},R)}|\sg(\varphi)|\dif x. 
\end{align*}
In this situation, sending $|\theta|\searrow 0$ yields $\eqref{eq:EkeMAIN1}$. We then consequently find
\begin{align*}
\int_{\ball(x_{0},R)}\langle \widetilde{f}''(0)\sg(\widetilde{v}),\sg(\varphi)\rangle\dif x & \leq \int_{\ball(x_{0},R)}\langle \widetilde{f}''(0)\sg(\widetilde{v})-\widetilde{f}'(\sg(\widetilde{v})),\sg(\varphi)\rangle\dif x \\ & + \int_{\ball(x_{0},R)}\langle \widetilde{f}'(\sg(\widetilde{v})),\sg(\varphi)\rangle\dif x \leq \int_{\ball(x_{0},R)}(cLV(\sg(\widetilde{v}))+\delta)|\sg(\varphi)|\dif x
\end{align*}
by Lemma~\ref{lem:auxSQC}(iii) and $\eqref{eq:EkeMAIN1}$; note that now $c$ depends on $m$. The same obviously is valid for $-\varphi$ instead of $\varphi$. This establishes \eqref{eq:EkeMAIN2}. In effect, \eqref{eq:EkeMAIN1} provides perturbed Euler-Lagrange equations as a substitute for the \textsc{Anzellotti}-type Euler-Lagrange equations for measures.

\emph{Step 2. Truncations and improved regularity for the comparison maps.} Starting from \eqref{eq:EkeMAIN2}, we let $\varphi\in\sobo_{0}^{1,\infty}(\ball(x_{0},R);\R^{n})$ be arbitrary and put $\psi:= \widetilde{v}-h$. We scale back to the unit ball and therefore put, for $x\in\ball(0,1)$, 
\begin{align*}
 \Psi(x):=\frac{1}{R}\psi(x_{0}+Rx),\;\;\;\Phi(x):=\frac{1}{R}\varphi(x_{0}+Rx),\;\;\;U(x):=\frac{1}{R}\widetilde{v}(x_{0}+Rx). 
\end{align*}
Since $h$ satisfies \eqref{eq:lin1}, we conclude from \eqref{eq:EkeMAIN2} with $\A:=\widetilde{f}''(0)$
\begin{align}\label{eq:TscherpelCenter10A}
\left\vert\,\int_{\ball(0,1)}\langle \A\sg(\Psi),\sg(\Phi)\rangle\dif x \right\vert \leq cL\int_{\ball(0,1)}V(\sg(U))|\sg(\Phi)|\dif x + \delta\int_{\ball(0,1)}|\sg(\Phi)|\dif x.
\end{align}
We then define a truncation operator $T\colon\R^{n}\to\R^{n}$ by 
\begin{align*}
T(y):=\begin{cases} y&\;\text{if}\;|y|\leq 1,\\
\frac{y}{|y|}&\;\text{if}\;|y|>1, 
\end{cases}
\qquad y\in\R^{n},
\end{align*}
and note that $T(\Psi)\in\lebe^{\infty}(\ball(0,1);\R^{n})$. Let us now consider the linear system 
\begin{align}\label{eq:Tscherpelsys}
\begin{cases}
-\di(\A\sg(\mathbf{T}))=T(\Psi)&\;\text{in}\;\ball(0,1),\\
\mathbf{T}=0&\;\text{on}\;\partial\!\ball(0,1)
\end{cases}
\end{align}
with its corresponding weak formulation 
\begin{align}\label{eq:Tscherpelsys1}
\int_{\ball(0,1)}\langle \A\sg(\mathbf{T}),\sg(\varrho)\rangle\dif x = \int_{\ball(0,1)}\langle T(\Psi),\varrho\rangle\dif x\qquad\text{for all}\;\varrho\in\hold_{c}^{\infty}(\ball(0,1);\R^{n}). 
\end{align}
Since $f$ is assumed strongly symmetric quasiconvex, it is strongly symmetric rank-one convex. Fix $p>n+1$. Then, by Proposition~\ref{thm:mazsha}, there exists a unique solution $\mathbf{T}\in\sobo^{2,p}(\ball(0,1);\R^{n})$ of \eqref{eq:Tscherpelsys} with $u|_{\partial\!\ball(0,1)}=0$. Thus there exists a constant $C=C(n,p,L,\ell)>0$ such that 
\begin{align}\label{eq:boundednessT}
\int_{\ball(0,1)}|\mathbf{T}|^{p}\dif x + \int_{\ball(0,1)}|\D\!\mathbf{T}|^{p}\dif x + \int_{\ball(0,1)}|\D^{2}\mathbf{T}|^{p}\dif x \leq C \int_{\ball(0,1)}|T(\Psi)|^{p}\dif x. 
\end{align}
In this situation, we invoke Morrey's embedding $\sobo^{1,p}(\ball;\R^{n})\hookrightarrow \hold^{0,1-n/p}(\overline{\ball};\R^{n})$ to find that $\mathbf{T}$ is Lipschitz together with the corresponding bound 
\begin{align}\label{eq:Tabbybound}
\|\D\!\mathbf{T}\|_{\lebe^{\infty}(\ball;\R^{n\times n})} \leq C(\|\D\!\mathbf{T}\|_{\lebe^{p}(\ball;\R^{n\times n})} + \|\D^{2}\mathbf{T}\|_{\lebe^{p}(\ball;\R^{n\times n}\times\R^{n})})\stackrel{\eqref{eq:boundednessT}}{\leq} C \|T(\Psi)\|_{\lebe^{p}(\ball;\R^{n})}. 
\end{align}
As $\mathbf{T}|_{\partial\!\ball(0,1)}=0$, from here we deduce $\mathbf{T}\in\sobo_{0}^{1,\infty}(\ball(0,1);\R^{n})$. Approximating a generic map $\rho\in\ld_{0}(\ball(0,1))$ by elements from $\hold_{c}^{\infty}(\ball(0,1);\R^{n})$ in the $\ld$-norm topology, we obtain 
\begin{align}\label{eq:Tscherpelsys2}
\int_{\ball(0,1)}\langle \A\sg(\mathbf{T}),\sg(\varrho)\rangle\dif x = \int_{\ball(0,1)}\langle T(\Psi),\varrho\rangle\dif x\qquad\text{for all}\;\varrho\in\ld_{0}(\ball(0,1)).  
\end{align}
Now, because of $2\leq n<p<\infty$, we have $|T(y)|^{p}= |y|^{p} \leq |y|^{2}$ for if $|y|\leq 1$ and thus there holds
\begin{align}\label{eq:Tscherpelsys3}
\|T(\Psi)\|_{\lebe^{p}}^{p}& = \int_{\ball(0,1)}|T(\Psi)|^{p}\dif x \leq c\int_{\ball(0,1)}V(\Psi)\dif x
\end{align}
by Lemma~\ref{lem:auxVSQC}. Combining \eqref{eq:Tscherpelsys3} with \eqref{eq:Tabbybound} consequently yields
\begin{align}\label{eq:TabbyBound1}
\|\sg(\mathbf{T})\|_{\lebe^{\infty}} \leq \|\D\mathbf{T}\|_{\lebe^{\infty}} \leq c\Big(\int_{\ball(0,1)}V(\Psi)\dif x\Big)^{\frac{1}{p}}, 
\end{align}
and here $c>0$ only depends on $\ell,L,m,n$ and $p$. 

\emph{Step 3. Conclusion for the approximating maps $\widetilde{v}$.} 
We now combine the estimates gathered so far to obtain inequality \eqref{eq:absolutemain} in a perturbed form. Recalling \eqref{eq:Vest1}, we succesively obtain
\begin{align*}
\int_{\ball(0,1)}V(\Psi)\dif x & \stackrel{\eqref{eq:Vest1}_{1}}{\leq} \int_{\ball(0,1)} 
\min\{|\Psi|,|\Psi|^{2}\}\dif x \\
& = \int_{\ball(0,1)} \langle T(\Psi),\Psi\rangle \dif x \;\;\;\;\;\;\;\;\;\;\;\;\;\;\;\;\;\;\;\;\;\;\;\;\;\;\;\;\;\;\;\;\;\;\;\;\;\;\;\;\;\;\;\;\;\;\;\;\;\;\;\;\;\;\;\;\;\;\;\;\;(\text{by definition of $T$})\\
& = \int_{\ball(0,1)} \langle \A\sg(\mathbf{T}),\sg(\Psi)\rangle \dif x\;\;\;\;\;\;\;\;\;\;\;\;\;\;\;\;\;\;\;\;\;\;\;\;\;\;\;\;\;\;\;\;\;(\text{by testing \eqref{eq:Tscherpelsys2} with $\rho=\Psi$})\\
& =  \int_{\ball(0,1)} \langle \A\sg(\Psi),\sg(\mathbf{T})\rangle \dif x \;\;\;\;\;\;\;\;\;\;\;\;\;\;\;\;\;\;\;\;\;\;\;\;\;\;\;\;\;\;\;\;\;\;\;\;\;\;\;\;\;\;\;\;\;\;\;\;\;\;\;\;\;\;\;\;(\text{as $\A\in\mathbb{S}(\rsym)$})\\ &\leq \int_{\ball(0,1)}(cLV(\sg(U))+\delta)|\sg(\mathbf{T})|\dif x \;\;\;\;\;\;\;\;\;\;\;\;\;\;\;\;\;(\text{by testing \eqref{eq:TscherpelCenter10A} with $\Phi=\mathbf{T}$})\\
& \leq  \int_{\ball(0,1)}(cLV(\sg(U))+\delta)\dif x \|\sg(\mathbf{T})\|_{\lebe^{\infty}} \\
& \leq  \Big(\int_{\ball(0,1)}(cLV(\sg(U))+\delta)\dif x\Big)\Big(\int_{\ball(0,1)}|V(\Psi)|\dif x\Big)^{\frac{1}{p}} \;\;\;\;\;\;\;\;\;\;\;\;\;\;\;\,\;\;(\text{by \eqref{eq:TabbyBound1}}).
\end{align*}
We therefore obtain 
\begin{align}
\Big(\int_{\ball(0,1)}V(\Psi)\dif x\Big)^{1-\frac{1}{p}}\leq \Big(\int_{\ball(0,1)}(cLV(\sg(U))+\delta)\dif x\Big).
\end{align}
At this stage recall that our choice of $p$ was only restricted to $p>n+1$. For $1<q<\frac{n+1}{n}$ as in the proposition, we thus find $n+1<p<\infty$ such that $p'=\frac{p}{p-1}=q$ and thus  
\begin{align}
\int_{\ball(0,1)}V(\Psi)\dif x \leq C \Big(\int_{\ball(0,1)}V(\sg(U))\dif x \Big)^{q}+ C \delta^{q}\mathscr{L}^{n}(\ball(0,1))^{q}. 
\end{align}
We consequently scale back to the original ball to find 
\begin{align}\label{eq:TabbyStarBound}
\dashint_{\ball(x_{0},R)}V\Big(\frac{\widetilde{v}-h}{R}\Big)\dif x \leq C \Big(\dashint_{\ball(0,R)}V(\sg(\widetilde{v}))\dif x \Big)^{q}+ C \delta^{q}\mathscr{L}^{n}(\ball(0,1))^{q}, 
\end{align}
and we note that the constant $C>0$ only depends on $m,n,q,L$ and $\ell$. 

\emph{Step 4. Limit passage $\delta\searrow 0$ and conclusion.} We now intend to send $\delta\searrow 0$; note that $\widetilde{v}$ actually depends on $\delta$: $\widetilde{v}=\widetilde{v}_{\delta}$. By Lipschitz continuity of $V$ we see that
\begin{align*}
\left\vert\dashint_{\ball(x_{0},R)}V\Big(\frac{\widetilde{v}-h}{R}\Big) - \dashint_{\ball(x_{0},R)}V\Big(\frac{\widetilde{u}-h}{R}\Big)\right\vert & \leq C(V) \dashint_{\ball(x_{0},R)}\left\vert \frac{\widetilde{u}-\widetilde{w}_{\delta}}{R}\right\vert + C(V)\dashint_{\ball(x_{0},R)}\left\vert \frac{\widetilde{v}-\widetilde{w}_{\delta}}{R}\right\vert \\
& \leq C(V)(\delta^{2} + (1+c_{\poinc})\delta)\to 0 
\end{align*}
by \eqref{eq:EkePrep} and $\eqref{eq:EkeMAIN}_{2}$ as $\delta\searrow 0$. Second, we obtain similarly
\begin{align*}
\left\vert \dashint_{\ball(x_{0},R)}V(\E\widetilde{u})-\dashint_{\ball(x_{0},R)}V(\sg(\widetilde{v}))\dif x\right\vert & \leq \left\vert \dashint_{\ball(x_{0},R)}V(\E\widetilde{u})-\dashint_{\ball(x_{0},R)}V(\sg(\widetilde{w}_{\delta}))\dif x\right\vert \\ & + C(V)\dashint_{\ball(x_{0},R)}|\sg(w_{\delta}-\widetilde{v})|\dif x \leq (\delta^{2}+C(V)\delta)\to 0
\end{align*}
as $\delta\searrow 0$. In conclusion, by \eqref{eq:TabbyStarBound} we have established  
\begin{align*}
\dashint_{\ball(x_{0},R)}V\Big(\frac{\widetilde{u}-h}{R}\Big)\dif x \leq C \Big(\dashint_{\ball(x_{0},R)} V(\E\widetilde{u}))\Big)^{q}, 
\end{align*}
which is the desired inequality \eqref{eq:absolutemain} and the proof is complete. 
\end{proof}
\begin{remark}[On the exponent $q$ in the previous proposition]\label{rem:discussAharmonicity}
It is important to remark that the exponent $q$ as it appears in the previous proposition can be chosen \emph{strictly larger} than one. From a technical perspective, the importance of $q>1$ is given by \eqref{eq:TabbyCenterEstimateMain2} from below, where the smallness assumption on the excess gives smallness of the critical quantity
\begin{align*}
\left(\frac{\mathbf{E}(x_{0},R_{0})}{R_{0}^{n}} \right)^{q-1}.
\end{align*}
If we could not use $q>1$ and only had $q=1$ at our disposal, this critical term would equal one and thus  destroy the excess decay later on in Proposition~\ref{prop:epsreg}. 
\end{remark}
In the preceding Proposition~\ref{prop:improvement} we have estimated a $V$-function type distance of $\widetilde{u}=u-a$ to its $\A$-harmonic approximation $h$, where $\A=\widetilde{f}''(0)=f''_{\sg(a)}(0)$. We conclude this subsection by showing how suitable Lebesgue norms of $\D\!h$ can be controlled by means of $\widetilde{u}$: 
%The need for passing to this system instead of sticking to \eqref{eq:lin1} lies in the fact that we eventually wish to estimate suitable Lebesgue norms of $\widetilde{h}$ and thus $h$ by virtue of a suitable excess quantity. Anticipating \eqref{eq:harmoniccontinuity} from below, we shall be able to estimate the right hand side of \eqref{eq:harmoniccontinuity} by a suitable fractional Poincar\'{e}-type inequality by the symmetric gradient of $\widetilde{u}$ once the correcting rigid deformation $b$ is chosen appropriately. This, however, is not achievable if we work without a possible corrector $b$. We now record some implications of the continuity of the solution operator $\Phi^{-1}$ of Theorem~\ref{thm:mazsha} with focus on its scaling properties. We recall that $\dif\sigma_{x}:=\dif\mathcal{H}^{n-1}(x)$ and accordingly for $y$. Also, we put $\widetilde{u}_{b}:=\widetilde{u}-b$. 
\begin{lemma}\label{lem:SQCcontinuityimplc}
In the situation of Proposition~\ref{prop:improvement} there exists a constant $C=C(n,\ell,L)>0$ such that for each $b\in\mathscr{R}(\R^{n})$ the map $\widetilde{h}:=h-b$ satisfies, with $\widetilde{u}_{b}:=\widetilde{u}-b$, 
\begin{align}\label{eq:harmoniccontinuity}
\left(\dashint_{\ball(x_{0},R)}|\D\!\widetilde{h}|^{\frac{n+1}{n}}\dif x\right)^{\frac{n}{n+1}}\leq CR^{-\frac{n}{n+1}}\Big(\dashint_{\partial\!\ball(x_{0},R)}\!\int_{\partial\!\ball(x_{0},R)}\frac{|\widetilde{u}_{b}(x)-\widetilde{u}_{b}(y)|^{\frac{n+1}{n}}}{|x-y|^{(n-1+\frac{1}{n})}}\dif\sigma_{x}\dif\sigma_{y}\Big)^{\frac{n}{n+1}}.
\end{align}
\end{lemma}
\begin{proof}
It is no loss of generality to assume $x_{0}=0$ and $R=1$. Then we have $\widetilde{u}|_{\partial\!\ball(0,1)}\in\sobo^{\frac{1}{n},\frac{n+1}{n}}(\partial\!\ball(0,1);\R^{n})$. With this choice of $x_{0}$ and $R$, and adopting the terminology of Proposition~\ref{thm:mazsha}, denote $S:=\Phi^{-1}(0,\cdot)$. Given $b\in\mathscr{R}(\R^{n})$, we define
\begin{align*}
\mathbf{u}_{b}:=(\widetilde{u}_{b})_{\partial\!\ball(0,1)}:=\dashint_{\partial\!\ball(0,1)}\widetilde{u}_{b}\dif\mathcal{H}^{n-1}(\in\R^{n}),
\end{align*}
where the dash is now understood with respect to $\mathcal{H}^{n-1}\mres\partial\!\ball(0,1)$. Since $h$ solves \eqref{eq:lin1}, $\mathbf{h}:=\widetilde{h}-\mathbf{u}_{b}:=h-b-\mathbf{u}_{b}$ is the unique solution of \begin{align}\label{eq:lin2}
\begin{cases}
-\di(\A\sg(\mathbf{h}))=0&\;\text{in}\;\ball(0,1),\\
\mathbf{h}=\widetilde{u}_{b}-\mathbf{u}_{b}&\;\text{on}\;\partial\!\ball(0,1). 
\end{cases}
\end{align}
Hence we have $\mathbf{h}=S(\widetilde{u}_{b}-\mathbf{u}_{b})$ so that, by Proposition~\ref{thm:mazsha} with some $C=C(n,L,\ell)>0$,
\begin{align}\label{eq:TscherpelsysDir}
\|\D\!\mathbf{h}\|_{\lebe^{\frac{n+1}{n}}(\ball(0,1);\R^{n\times n})} & \leq \|\mathbf{h}\|_{\sobo^{1,\frac{n+1}{n}}(\ball(0,1);\R^{n})} \leq C \|\widetilde{u}_{b}-\mathbf{u}_{b}\|_{\sobo^{\frac{1}{n+1},\frac{n+1}{n}}(\partial\!\ball(0,1);\R^{n})}.
\end{align}
On the other hand, if $x,y\in\partial\!\ball(0,1)$, then $|x-y|\leq 2$ and therefore $
\mathscr{H}^{n-1}(\partial\!\ball(0,1))^{-1}\leq c(n)|x-y|^{-n+1-1/n}$. Thus,  
\begin{align*}
\Big(\int_{\partial\!\ball(0,1)}|\widetilde{u}_{b}(x)-\mathbf{u}_{b}|^{\frac{n+1}{n}}\dif\sigma_{x}\Big)^{\frac{n}{n+1}} & = \Big(\int_{\partial\!\ball(0,1)}\dashint_{\partial\!\ball(0,1)}|\widetilde{u}_{b}(x)-\widetilde{u}_{b}(y)|^{\frac{n+1}{n}}\dif\sigma_{y}\dif\sigma_{x}\Big)^{\frac{n}{n+1}}\\
& \leq C(n)\Big(\int_{\partial\!\ball(x_{0},1)}\int_{\partial\!\ball(0,1)}\frac{|\widetilde{u}_{b}(x)-\widetilde{u}_{b}(y)|^{\frac{n+1}{n}}}{|x-y|^{n-1+\frac{1}{n}}}\dif\sigma_{y}\dif\sigma_{x}\Big)^{\frac{n}{n+1}}.
\end{align*}
As a consequence, we obtain in conjunction with \eqref{eq:TscherpelsysDir} and a constant $C=C(n,L,\ell)>0$
\begin{align}\label{eq:fracPoinc}
\|\D\!\widetilde{h}\|_{\lebe^{\frac{n+1}{n}}(\ball(0,1);\R^{n\times n})}=\|\D\!\mathbf{h}\|_{\lebe^{\frac{n+1}{n}}(\ball(0,1);\R^{n\times n})}\leq C[\widetilde{u}_{b}]_{\sobo^{\frac{1}{n+1},\frac{n+1}{n}}(\partial\!\ball(0,1);\R^{n})}.
\end{align}
The general case now follows by scaling, and the proof is complete. 
\end{proof}
\subsection{Excess decay}\label{sec:excess}
The objective of the present subsection is to establish the excess decay that will eventually lead to the desired partial regularity assertion of Theorem~\ref{thm:main1} by an iteration scheme. To this end, let $u\in\bd_{\locc}(\Omega)$ be a local generalised minimiser of $F$, where the integrand satisfies \ref{item:reg1}--\ref{item:reg3} from Theorem~\ref{thm:main1}, and let $M_{0}>0$ be a given number. Our strategy then runs in four steps: In a first step, we choose a ball for which both the mean value \emph{and} a certain excess quantity of $\E u$ is small. Then, in a second step, we slightly diminish the radius of the given ball to obtain a ball on whose boundary we may apply the Fubini-type theorem for $\bd$-maps. This makes the $\A$-harmonic approximation of the previous subsection available. Defining suitable comparison maps in step 3, we then combine Propositions~\ref{prop:CaccQC} and \ref{prop:improvement} in step 4 to conclude a preliminary excess decay. In doing so, we define for $z\in\Omega$ and $0<r<\dist(z,\partial\Omega)$ two excess quantities by
\begin{align*}
\mathbf{E}(u;z,r):=\int_{\ball(z,r)}V(\E u - (\E u)_{\ball(z,r)})\;\;\text{and}\;\;\widetilde{\mathbf{E}}(u;z,r):=\frac{\mathbf{E}(u;z,r)}{\mathscr{L}^{n}(\ball(z,r))},
\end{align*}
and we will often write $\mathbf{E}(z,r):=\mathbf{E}(u;z,r)$, assuming that $u$ is fixed. Here, as usual, $(\E u)_{\ball(z,r)}=\E u(\ball(z,r))/\mathscr{L}^{n}(\ball(z,r))$. 

\emph{Step 1.} \emph{Smallness Assumptions.} Let $M_{0}>0$ be given and fix a ball $\ball_{R_{0}}=\ball(x_{0},R_{0})\Subset\Omega$ such that 
\begin{align}\label{eq:smallness1}
|(\E u)_{\ball_{R_{0}}}|\leq M_{0}. 
\end{align} 
and 
\begin{align}\label{eq:smallness2}
\dashint_{\ball_{R_{0}}}|\E u-(\E u)_{\ball_{R_{0}}}|\leq 1. 
\end{align}
We write $\ball_{r}:=\ball(x_{0},r)$ in all of what follows. 

\emph{Step 2.} \emph{Selection of a good radius.} In a second step, we fix an affine--linear map $a\colon\R^{n}\to\R^{n}$ with $\sg(a)=(\E u)_{\ball_{R_{0}}}$. We then put $\widetilde{u}:=u-a$ and  $\widetilde{f}:=f_{\sg(a)}$, cf.~\eqref{eq:shift}. Starting from $R_{0}>0$ as given above, we now apply Theorem~\ref{thm:fubini1}. Consequently, we find $R\in (\frac{9}{10}R_{0},\frac{19}{20}R_{0})$ such that $\widetilde{u}|_{\partial\!\ball_{R}}\in\sobo^{\frac{1}{n+1},1+\frac{1}{n}}(\partial\!\ball_{R};\R^{n})$ and a rigid deformation $b\in\mathscr{R}(\R^{n})$ together with the corresponding estimate (with $\theta=\frac{1}{n+1}$ and accordingly $p=\frac{n+1}{n}$ in Theorem~\ref{thm:fubini1})
\begin{align}\label{eq:TabTHeartbound}
\begin{split}
 \left(\dashint_{\partial\!\ball_{R}}\int_{\partial\!\ball_{R}} \right. & \left. \frac{|\widetilde{u}_{b}(x)-\widetilde{u}_{b}(y)|^{1+\frac{1}{n}}}{|x-y|^{(n-1)+\frac{1}{n}}} \dif\sigma_{x}\dif\sigma_{y}\right)^{\frac{n}{n+1}} \\ & \leq C\frac{(\frac{19}{20}R_{0})^{n}}{(\tfrac{9}{10}R_{0})^{\frac{n(n-1)}{n+1}}(\frac{1}{20}R_{0})^{\frac{n}{n+1}}}\frac{1}{(\frac{19}{20}R_{0})^{n}}\int_{\overline{\ball_{19R_{0}/20}}}|\E\widetilde{u}| \\ & \leq C \frac{R_{0}^{\frac{n}{n+1}}}{R_{0}^{n}}\int_{\ball_{R_{0}}}|\E u -(\E u)_{\ball_{R_{0}}}|.
\end{split}
\end{align}
where we recall $\widetilde{u}_{b}:=\widetilde{u}-b(=u-a-b)$, and $C=C(n)>0$ is a constant.

\emph{Step 3.} \emph{Definition of comparison maps.} We put $\A:=f''((\E u)_{\ball_{R}})$ and  pick the $\A$--harmonic mapping $\widetilde{h}\colon\ball_{R}\to\R^{n}$ solving 
\begin{align}\label{eq:lin6}
\begin{cases}
-\di(f''( (\E u)_{\ball_{R}})\sg(\widetilde{h}))= 0 &\;\text{in}\;\ball_{R},\\
\widetilde{h} = \widetilde{u}_{b}&\;\text{on}\;\partial\!\ball_{R}. 
\end{cases}
\end{align}
We are thus in the setting of \eqref{eq:lin1} and Lemma~\ref{lem:SQCcontinuityimplc} from above; by Proposition~\ref{thm:mazsha}, $\widetilde{h}\in\hold^{\infty}(\ball_{R};\R^{n})$.  Then we define 
\begin{align}\label{eq:defa0}
A(x):=\widetilde{h}(x_{0})+ \D\!\widetilde{h}(x_{0})(x-x_{0})\;\;\;\text{and}\;\;\;a_{0}(x):=a(x)+A(x),\qquad x\in\ball_{R}.
\end{align}
We then obtain 
\begin{align*}
|\sg(a_{0})| & =|\sg(a)+\sg(\widetilde{h})(x_{0})| = |(\E u)_{\ball_{R_{0}}} + \sg(\widetilde{h})(x_{0})| \\
& \leq M_{0} + |\sg(\widetilde{h})(x_{0})|\qquad\qquad\qquad\qquad\qquad\qquad\qquad\qquad\qquad\;\;\;\;\,(\text{by \eqref{eq:smallness1}})\\
& \leq   M_{0} + \sup_{\ball_{R/2}}|\D\!\widetilde{h}| \;\;\;\;\;\;\;\;\qquad\quad\qquad\qquad\qquad\qquad\quad\quad\;\;\;\qquad(\text{as}\,|\E v|\leq |\D\!v|)\\ & \leq   M_{0} + c\dashint_{\ball_{R}}|\D\!\widetilde{h}|\dif x\qquad\qquad\qquad\qquad\qquad\qquad\qquad\qquad\;\qquad\;(\text{by \eqref{eq:TscherpelCenter00}})\\
& \leq  M_{0} + c\Big(\dashint_{\ball_{R}}|\D\!\widetilde{h}|^{\frac{n+1}{n}}\dif x\Big)^{\frac{n}{n+1}}
\end{align*}
and thus 
\begin{align}\label{eq:M0est}\begin{split}
|\sg(a_{0})| & \leq M_{0} \\
& + cR^{-\frac{n}{n+1}}\Big(\dashint_{\partial\!\ball_{R}}\int_{\partial\!\ball_{R}}\frac{|\widetilde{u}_{b}(x)-\widetilde{u}_{b}(y)|^{\frac{n+1}{n}}}{|x-y|^{n-1+\frac{1}{n}}}\dif\sigma_{x}\dif\sigma_{y}\Big)^{\frac{n}{n+1}}\;\;\;\;\;\;\;\;\;\;\;\;\text{(by Lemma~\ref{lem:SQCcontinuityimplc})}\\
& \leq M_{0} + \frac{c}{R_{0}^{n}}\int_{\ball(x_{0},R_{0})}|\E u-(\E  u)_{\ball_{R_{0}}}|\qquad\;\;\;\;\;\;\;\;\;\;\;\;\;\;\;\;\;\;\;\;\;\;\;\;(\text{by~\eqref{eq:TabTHeartbound} and $R_{0}\sim R$})
\\ &\leq M_{0}+c,
\end{split}
\end{align}
where the last estimates holds because of \eqref{eq:smallness2}. Here, $c=c(n,L,\ell)>0$ is a constant that we fix now. In particular, the constants appearing here do not depend on $R$ or $R_{0}$. Summarising, if we put $m:=M_{0}+c$ as on the right side of the previous chain of inequalities, then we obtain 
\begin{align}\label{eq:TabTHeartBound1}
|\sg(a_{0})|\leq m. 
\end{align}

\emph{Step 4. Comparison estimates.} Let $0<\sigma<\frac{1}{5}$ be arbitrary. We note, as a consequence of Lemma~\ref{lem:auxVSQC} and Jensen's inequality, 
\begin{align}\label{eq:TabTHeart2}
\begin{split}
\int_{\ball_{\sigma R_{0}}}V(\E u-(\E u)_{\ball_{\sigma R_{0}}}) &  = \int_{\ball_{\sigma R_{0}}}V(\E u-\E a_{0}+\E a_{0}-(\E u)_{\ball_{\sigma R_{0}}})\\
& \leq C\int_{\ball_{\sigma R_{0}}}V(\E u-\E a_{0}) + C \int_{\ball_{\sigma R_{0}}}V((\E\,(u-a_{0}))_{\ball_{\sigma R_{0}}})\\
& \leq C\int_{\ball_{\sigma R_{0}}}V(\E\,(u-a_{0})) \stackrel{b\in\mathscr{R}(\R^{n})}{=} C\int_{\ball_{\sigma R_{0}}}V(\E\,(u-b-a_{0})).
\end{split}
\end{align}
Our next objective is to apply the Caccioppoli--type inequality, Proposition~\ref{prop:CaccQC}. Having chosen $m>0$ as it appears in Proposition~\ref{prop:CaccQC} by \eqref{eq:TabTHeartBound1}, we find $c=c(m,n,L,\ell)>0$ such that \eqref{eq:CaccQC} holds with the requisite modifications; note that $b+a_{0}$ is affine-linear, too, with $|\sg(b+a_{0})|=|\sg(a_{0})|\leq m$. We then estimate, using \eqref{eq:TabTHeart2} and the Caccioppoli--type inequality in the first step, 
\begin{align*}
\int_{\ball_{\sigma R_{0}}}V (\E u - (\E u)_{\ball_{\sigma R_{0}}}) & \leq C\int_{\ball_{2\sigma R_{0}}}V\Big(\frac{\widetilde{u}(x)-b(x)-A(x)}{\sigma R_{0}} \Big)\dif x \\
& \leq C\int_{\ball_{2\sigma R_{0}}}V\Big(\frac{\widetilde{u}(x)-b(x)-\widetilde{h}(x)-A(x)+\widetilde{h}(x)}{\sigma R_{0}} \Big)\dif x \\
& \leq C\int_{\ball_{2\sigma R_{0}}}V\Big(\frac{(\widetilde{u}-b)-\widetilde{h}}{\sigma R_{0}} \Big)\dif x + C\int_{\ball_{2\sigma R_{0}}}V\Big(\frac{\widetilde{h}-A}{\sigma R_{0}} \Big)\dif x \\
& \leq  \frac{C}{\sigma^{2}}\int_{\ball_{R}}V\Big(\frac{\widetilde{u}_{b}-\widetilde{h}}{R} \Big)\dif x + C\int_{\ball_{2\sigma R_{0}}}V\Big(\frac{\widetilde{h}-A}{\sigma R_{0}} \Big)\dif x\\
& =: \mathbf{I}+\mathbf{II}, 
\end{align*}
where $C=C(m,n,\tfrac{L}{\ell})>0$ is a constant. Here we have used $\ball_{2\sigma R_{0}}\subset\ball_{R}$, uniform comparability of $R$ and $R_{0}$ and the fact that $V(\lambda z)\leq c\lambda^{2}V(z)$ for a constant $c>0$, all $z\in\rsym$ and $|\lambda|\geq 1$ (cf.~Lemma~\ref{lem:auxVSQC}). We continue with the estimation of $\mathbf{I}$, and for this purpose let $1<q<\frac{n+1}{n}$ be arbitrary but fixed. We use Proposition~\ref{prop:improvement} and uniform comparability of $R,R_{0}$ to obtain
\begin{align}\label{eq:Iestimateexcessdecay}
\mathbf{I} & =  \frac{C}{\sigma^{2}}\int_{\ball_{R}}V\Big(\frac{\widetilde{u}_{b}-\widetilde{h}}{ R} \Big)\dif x =\frac{CR^{n}}{\sigma^{2}}\dashint_{\ball_{R}}V\Big(\frac{\widetilde{u}-h}{ R} \Big)\dif x \leq C\frac{R_{0}^{n}}{\sigma^{2}} \left(\dashint_{\ball_{R}}V(\E \widetilde{u}) \right)^{q}, 
\end{align}
the last step being valid by uniform comparability of $R$ and $R_{0}$. As usual, the map $h$ is defined as the solution of the strongly symmetric elliptic system \eqref{eq:lin6} with boundary datum $\widetilde{u}=u-a$. As to $\mathbf{II}$, let $x\in\ball_{2\sigma R_{0}}$. We employ a pointwise estimate to find by use of Taylor's formula 
\begin{align*}
\left\vert\frac{\widetilde{h}(x)-A(x)}{\sigma R_{0}}\right\vert & = \left\vert\frac{\widetilde{h}(x)-\widetilde{h}(x_{0})-\langle \D \widetilde{h}(x_{0}),x-x_{0}\rangle}{\sigma R_{0}}\right\vert &\\ & \leq C \big(\sup_{\ball_{R/2}}|\D^{2}\widetilde{h}|\big)\frac{|x-x_{0}|^{2}}{\sigma R_{0}}&\\
& \leq C \big(\sup_{\ball_{R/2}}|\D^{2}\widetilde{h}|\big)\frac{(2\sigma R_{0})^{2}}{\sigma R_{0}}&\;(\text{since}\,x\in\ball(x_{0},2\sigma R_{0}))\\
& \leq C\sigma R \big(\sup_{\ball_{R/2}}|\D^{2}\widetilde{h}|\big)&\;(\text{since}\;R_{0}\leq \tfrac{10}{9}R)\\
& \leq C\sigma\dashint_{\ball_{R}}|\D\widetilde{h}|\dif x &\;(\text{by Proposition~\ref{thm:mazsha}}) \\
& \leq C\sigma\Big(\dashint_{\ball_{R}}|\D\widetilde{h}|^{\frac{n+1}{n}}\dif x\Big)^{\frac{n}{n+1}} =: \mathbf{III}&\;\text{(by Jensen)}.
\end{align*}
Similarly as in the estimation given in \eqref{eq:M0est}, we again employ Lemma~\ref{lem:SQCcontinuityimplc} to further obtain 
\begin{align*}
\mathbf{III} & \leq C\sigma\dashint_{\ball_{R_{0}}}|\E \widetilde{u}|\stackrel{\text{Def}}{=}C\sigma\dashint_{\ball_{R_{0}}}|\E\,(u-a)|&\;\text{(by Lemma~\ref{lem:SQCcontinuityimplc} and \eqref{eq:TabTHeartbound})}\\
& = C\sigma \dashint_{\ball_{R_{0}}}|\E u-(\E u)_{\ball_{R_{0}}}|&\;(\text{since}\,\E a=(\E u)_{\ball_{R_{0}}})\\
& = C\sigma \Big(\dashint_{\ball_{R_{0}}}|\E u-(\E u)_{\ball_{R_{0}}}|\Big)^{2\cdot\frac{1}{2}}& \\
& \leq C\sigma \Big(V\Big(\dashint_{\ball_{R_{0}}}|\E u-(\E u)_{\ball_{R_{0}}}|\Big)\Big)^{\frac{1}{2}}&\;\text{(by \eqref{eq:smallness2} and $\eqref{eq:Vest1}_{1}$)}\\
& \leq C\sigma \Big(\dashint_{\ball_{R_{0}}}V(|\E u-(\E u)_{\ball_{R_{0}}}|)\Big)^{\frac{1}{2}}&\;\text{(by Jensen and $\tfrac{9}{10}R_{0}<R< R_{0}$)}.
\end{align*}
Collecting estimates, we obtain with a constant $C=C(m,n,L,\ell)>0$ and for all $x\in\ball_{2\sigma R_{0}}$
\begin{align}\label{eq:almostfinalTscherp}
V\Big(\frac{\widetilde{h}(x)-A(x)}{\sigma R_{0}}\Big) \leq CV\Big(\sigma \Big(\dashint_{\ball_{R_{0}}}V(|\E u-(\E u)_{\ball_{R_{0}}}|)\Big)^{\frac{1}{2}}\Big)=:CV(\Upsilon),  
\end{align}
where $\Upsilon$ is defined in the obvious manner. Now, since $V(\cdot)\leq |\cdot|$, $0<\sigma<\frac{1}{5}$ and by \eqref{eq:smallness2}, $\Upsilon\leq 1$.  Consequently, integrating \eqref{eq:almostfinalTscherp} over $\ball_{2\sigma R_{0}}$, we obtain with $C=C(m,n,L,\ell)>0$
\begin{align}\label{eq:IIestimateexcessdecay}
\begin{split}
\mathbf{II} & = C\int_{\ball_{2\sigma R_{0}}}V\left(\frac{\widetilde{h}(x)-A(x)}{\sigma R_{0}} \right)\dif x \leq C(\sigma R_{0})^{n} V(\Upsilon) \\ & \stackrel{\eqref{eq:Vest1}_{1}}{\leq} C(\sigma R_{0})^{n}\min\{\Upsilon,\Upsilon^{2}\}\leq C\sigma^{n+2} R_{0}^{n}\dashint_{\ball_{R_{0}}}V(|\E u-(\E u)_{\ball_{R_{0}}}|).
\end{split}
\end{align}
Combining estimates \eqref{eq:Iestimateexcessdecay} and \eqref{eq:IIestimateexcessdecay}, we then find with a constant $C=C(n,m,L,\ell,q)>0$ that
\begin{align}\label{eq:TabbyCenterEstimateMain}
\begin{split}
\mathbf{E}(x_{0},\sigma R_{0}) & \leq C\frac{R_{0}^{n}}{\sigma^{2}} \left(\dashint_{\ball_{R}}V(\E \widetilde{u}) \right)^{q} + C\sigma^{n+2} R_{0}^{n}\dashint_{\ball_{R}}V(|\E\widetilde{u}|) \\
& \leq \frac{C}{\sigma^{2}} \left(\frac{\mathbf{E}(x_{0},R_{0})}{R_{0}^{n}} \right)^{q-1}\mathbf{E}(x_{0},R_{0})+ C\sigma^{n+2} \mathbf{E}(x_{0},R_{0})\\
& = \left( \frac{C}{\sigma^{2}} (\widetilde{\mathbf{E}}(x_{0},R_{0}))^{q-1}+ C\sigma^{n+2} \right)\mathbf{E}(x_{0},R_{0}).
\end{split}
\end{align}
We will now use the previous inequality to deduce a preliminary excess decay. 
\begin{proposition}\label{prop:epsreg}
Let $f\colon\rsym\to\R$ satisfy \ref{item:reg1}--\ref{item:reg3} from Theorem~\ref{thm:main1}. Given $0<\alpha<1$, $M_{0}>0$ and $1<q<\frac{n+1}{n}$, there exist two parameters $\sigma=\sigma(n,L,\ell,\alpha,M_{0},q)\in (0,\tfrac{1}{5})$ as well as $\widetilde{\varepsilon}=\widetilde{\varepsilon}(n,L,\ell,\alpha,M_{0},q)\in (0,1)$ such that every local $\bd$-minimiser $u\in\bd_{\locc}(\Omega)$ of the functional $F$ satisfies the following: If $\ball(x_{0},R_{0})\Subset\Omega$ is an open ball with 
$0<R_{0}\leq 1$ together with 
\begin{align}\label{eq:TabbyCenterEstimateCond00}
\widetilde{\mathbf{E}}(u;x_{0},R_{0})\leq \widetilde{\varepsilon}^{2}\;\;\;\text{and}\;\;\;|(\E u)_{\ball(x_{0},R_{0})}|\leq M_{0},  
\end{align}
then there holds 
\begin{align}\label{eq:TabbyCenterEstimateMain1}
\widetilde{\mathbf{E}}(u;x_{0},\sigma R_{0})\leq \sigma^{1+\alpha}\widetilde{\mathbf{E}}(u;x_{0},R_{0}).
\end{align}
\end{proposition}
\begin{proof}
Let $\alpha\in (0,1)$ and $M_{0}>0$ be given. We estimate with $H:=\E u - (\E u)_{\ball(x_{0},R_{0})}$, Lemma~\ref{lem:auxVSQC} and the shorthands $A_{R_{0}}^{\lessgtr}:=\ball(x_{0},R_{0})\cap \{|H|\lesseqgtr 1\}$
\begin{align*}
\dashint_{\ball(x_{0},R_{0})}|H| & \leq \frac{1}{\omega_{n}R_{0}^{n}}\int_{A_{R_{0}^{\leq}}}|H| + \frac{1}{\omega_{n}R_{0}^{n}}\int_{A_{R_{0}}^{>}}|H| = \frac{\mathscr{L}^{n}(A_{R_{0}}^{\leq})}{\omega_{n}R_{0}^{n}}
\dashint_{A_{R_{0}}^{\leq }}|H|  + \frac{C}{R_{0}^{n}}\int_{A_{R_{0}}^{>}}V(|H|) \\
& \leq C\frac{\mathscr{L}^{n}(A_{R_{0}}^{\leq})^{\frac{1}{2}}}{R_{0}^{n}}\Big(\int_{\ball(x_{0},R_{0})}V(|H|)\Big)^{\frac{1}{2}}  + \frac{C}{R_{0}^{n}}\int_{\ball(x_{0},R_{0})}V(|H|) \\
& \leq C\left(\dashint_{\ball(x_{0},R_{0})}V(|H|)\right)^{\frac{1}{2}} + C\dashint_{\ball(x_{0},R_{0})}V(|H|) \leq C(\sqrt{\widetilde{\mathbf{E}}(u;x_{0},R_{0})}+\widetilde{\mathbf{E}}(u;x_{0},R_{0})), 
\end{align*}
where $C=C(n)>0$. We may thus choose a preliminary $\widetilde{\varepsilon}_{0}\in (0,1)$ such that  $\widetilde{\mathbf{E}}(u;x_{0},R_{0})\leq \widetilde{\varepsilon}_{0}^{2}$ and $|(\E u)_{\ball(x_{0},R_{0})}|\leq M_{0}$ imply \eqref{eq:smallness1} and \eqref{eq:smallness2}.    At this stage, for $1<q<\frac{n+1}{n}$, \eqref{eq:TabbyCenterEstimateMain} is available and therefore yields for $0<\sigma<\frac{1}{5}$
\begin{align}\label{eq:TabbyCenterEstimateMain2}
\begin{split}
\widetilde{\mathbf{E}}(u;x_{0},\sigma R_{0}) &  \leq \left( \frac{C}{\sigma^{n+2}} \left(\widetilde{\mathbf{E}}(u;x_{0},R_{0})\right)^{q-1}+ C\sigma^{2} \right)\widetilde{\mathbf{E}}(u;x_{0},R_{0}),
\end{split}
\end{align}
where now\footnote{Note that the constant $C>0$ in \eqref{eq:TabbyCenterEstimateMain} depends on $n,m,L,\ell$ and $q$, but by \eqref{eq:M0est}, $m$ depends on $n$ and $M_{0}$ only.} $C=C(n,M_{0},L,\ell,q)>0$. We subsequently choose $\sigma=\sigma(n,M_{0},L,\ell,q,\alpha)>0$ so small such that with the constant $C>0$ from \eqref{eq:TabbyCenterEstimateMain2} there holds
$2C\sigma^{2}\leq \sigma^{1+\alpha}$. We then put $\widetilde{\varepsilon}:=\min\{\sigma^{\frac{n+4}{2(q-1)}},\widetilde{\varepsilon}_{0}\}$.
In turn, if $\widetilde{\mathbf{E}}(u;x_{0},R_{0})\leq \widetilde{\varepsilon}^{2}$ and $|(\E u)_{\ball(x_{0},R_{0})}|\leq M_{0}$, then  \eqref{eq:TabbyCenterEstimateMain2} gives 
\begin{align*}
\widetilde{\mathbf{E}}(u;x_{0},\sigma R_{0}) & \leq ( 2C\sigma^{2})\widetilde{\mathbf{E}}(u;x_{0},R_{0}) \leq \sigma^{1+\alpha}\widetilde{\mathbf{E}}(u;x_{0},R_{0}), 
\end{align*}
and this is precisely \eqref{eq:TabbyCenterEstimateMain1}. The proof is complete. 
\end{proof}
\subsection{Iteration and Proof of Theorem~\ref{thm:main1}}
To conclude the proof of Theorem~\ref{thm:main1}, we need to iterate Proposition~\ref{prop:epsreg}. 
\begin{corollary}[Iteration]\label{thm:iterationSQC}
Let $f\colon\rsym\to\R$ satisfy \ref{item:reg1}--\ref{item:reg3} from Theorem~\ref{thm:main1}. Given $0<\alpha<1$ and $M_{0}>0$, there exist $\varepsilon=\varepsilon(n,L,\ell,\alpha,M_{0})\in (0,1)$ and $R_{0}=R_{0}(n,L,\ell,M_{0},\alpha)\in (0,1)$ such that every generalised local minimiser $u\in\bd_{\locc}(\Omega)$ of the functional $F$ satisfies the following: If $x_{0}\in\Omega$ and $0<R<R_{0}$ are such that $\ball(x_{0},R_{0})\Subset\Omega$ and 
\begin{align}\label{eq:TabbyCenterEstimateCond}
\widetilde{\mathbf{E}}(u;x_{0},R)\leq \varepsilon^{2}\;\;\;\text{and}\;\;\;|(\E u)_{\ball(x_{0},R)}|\leq \frac{M_{0}}{2}, 
\end{align}
then there holds 
\begin{align}\label{eq:TabbyCenterEstimateMain1}
\widetilde{\mathbf{E}}(u;x_{0},r)\leq C\Big(\frac{r}{R}\Big)^{2\alpha}\widetilde{\mathbf{E}}(u;x_{0},R)\qquad\text{for all}\;0<r\leq R.
\end{align}
Here, $C=C(n,L,\ell,\alpha,M_{0})>0$ is a constant. 
\end{corollary}
The corollary is proved in a standard manner, the proof following, e.g., \cite[Prop.~4.8]{GK2} or \cite[Lem.~5.8]{Beck}; note that the dependence on $q$ in Proposition~\ref{prop:epsreg} is removed by specialising, e.g., to $q=\frac{2n+1}{2n}\in (1,\frac{n+1}{n})$. Working from here, we can proceed to the 
\begin{proof}[Proof of Theorem~\ref{thm:main1}]
Let $0<\alpha<1$ and $M>0$ be given. We put $M_{0}:=8\max\{M,1\}$. Then, by the previous corollary, there exist $\varepsilon=\varepsilon(n,L,\ell,\alpha,M_{0})\in (0,1)$ and $R_{0}=R_{0}(n,L,\ell,M_{0},\alpha)\in (0,1)$ such that \eqref{eq:TabbyCenterEstimateCond} implies \eqref{eq:TabbyCenterEstimateMain1}. Within the framework of Theorem~\ref{thm:main1}, we put $\varepsilon_{M}:=\varepsilon^{2}/2^{2n+4}$ and let $0<R<R_{0}$ be such that 
\begin{align}\label{eq:chooseepsfinal}
\widetilde{\mathbf{E}}(u;x_{0},R)\leq\varepsilon_{M}= \frac{\varepsilon^{2}}{2^{2n+4}}\;\;\;\text{and}\;\;\;|(\E u)_{\ball(x_{0},R)}|\leq M (\leq \tfrac{1}{2}M_{0}). 
\end{align}
 Our aim is to show that with $R':=\frac{1}{2}R$ there holds $
\widetilde{\mathbf{E}}(u;x,R')\leq \varepsilon^{2}$ and $|(\E u)_{\ball(x,R')}|\leq \tfrac{1}{2}M_{0}$ 
for all $x\in\ball(x_{0},R')$. We have 
\begin{align*}
\dashint_{\ball(x,R')} & V(|\mathscr{E}u-(\E u)_{\ball(x,R')}|)\dif\mathscr{L}^{n}  + \frac{|\E^{s}u|(\ball(x,R'))}{\omega_{n}(R')^{n}} \\
& \stackrel{\eqref{eq:Vest1}_{3},\,V(\cdot)\leq|\cdot|}{\leq} 2\dashint_{\ball(x,R')} V(|\mathscr{E}u-(\mathscr{E}u)_{\ball(x,R')}|)\dif\mathscr{L}^{n}  + 3\frac{|\E^{s}u|(\ball(x,R'))}{\omega_{n}(R')^{n}}\\
& \;\;\;\;\;\;\;\;\;\,\leq \frac{2\cdot 2^{2n}}{\omega_{n}^{2}R^{2n}}\int_{\ball(x_{0},R)}\int_{\ball(x_{0},R)}V(|\mathscr{E}u(y)-\mathscr{E}u(z)|)\dif y\dif z + 2^{n+2}\frac{|\E^{s}u|(\ball(x_{0},R))}{\omega_{n}R^{n}}\\
& \;\;\;\;\;\;\stackrel{\eqref{eq:Vest1}_{3}}{\leq} \frac{8\cdot 2^{2n}}{\omega_{n}^{2}R^{2n}}\int_{\ball(x_{0},R)}\int_{\ball(x_{0},R)}V(|\mathscr{E}u(y)-(\E u)_{\ball(x_{0},R)}|)\dif y\dif z  + 2^{n+2}\frac{|\E^{s}u|(\ball(x_{0},R))}{\omega_{n}R^{n}}\\
& \;\;\;\;\;\;\;\stackrel{\eqref{eq:chooseepsfinal}}{\leq} 2^{2n+3}\frac{\varepsilon^{2}}{2^{2n+4}} < \varepsilon^{2}. 
\end{align*}
On the other hand, we have by \eqref{eq:chooseepsfinal} and $\eqref{eq:Vest1}_{3}$ in the third step
\begin{align*}
|(\E u)_{\ball(x,R')}| & \leq \left\vert\dashint_{\ball(x,R')}\mathscr{E}u \dif\mathscr{L}^{n}\right\vert + \frac{|\E^{s}u|(\ball(x,R'))}{\omega_{n}(R')^{n}} \\
& \leq \left\vert\dashint_{\ball(x,R')}\mathscr{E}u -(\E u)_{\ball(x_{0},R)}\dif\mathscr{L}^{n}\right\vert + \frac{|\E^{s}u|(\ball(x,R'))}{\omega_{n}(R')^{n}} + |(\E u)_{\ball(x_{0},R)}|\\
& \leq 2^{n}\Big(\frac{1}{\sqrt{2}-1}\dashint_{\ball(x_{0},R)}V(|\mathscr{E}u -(\E u)_{\ball(x_{0},R)}|)\dif\mathscr{L}^{n}\Big)^{\frac{1}{2}} + 2^{n}\frac{|\E^{s}u|(\ball(x_{0},R))}{\omega_{n}R^{n}} + M\\
& \leq \frac{\varepsilon}{4\sqrt{\sqrt{2}-1}}+\frac{\varepsilon^{2}}{2^{n+4}}+M \leq M+1
\end{align*}
having used that $\varepsilon\in (0,1)$ in the ultimate step. Since $M+1\leq 2\max\{M,1\}\leq \frac{M_{0}}{2}$, we thus obtain by \eqref{eq:chooseepsfinal} and \eqref{eq:TabbyCenterEstimateMain1} that for all $x\in\ball(x_{0},R')$ and all $0<r<R'$ there holds 
\begin{align*}
\widetilde{\mathbf{E}}(u;x,r)\leq C(n,L,\ell,\alpha,M_{0})\Big(\frac{r}{R'}\Big)^{2\alpha}\widetilde{\mathbf{E}}(u;x,R').
\end{align*} 
Working from here, we first deduce by sending $r\searrow 0$ that $\E^{s}u \equiv 0$ in $\ball(x_{0},R')$. Therefore, setting $G(t):=\min\{t,t^{2}\}$ for $t\geq 0$, we find by  $\eqref{eq:Vest1}_{1}$ and Jensen's inequality 
\begin{align*}
(\sqrt{2}-1)G(\Upsilon'):=(\sqrt{2}-1)G\Big(\dashint_{\ball(x,r)}|\mathscr{E}u-(\mathscr{E}u)_{\ball(x,r)}|\dif\mathscr{L}^{n}\Big) & \leq C(n,L,\ell,\alpha,M_{0})\Big(\frac{r}{R'}\Big)^{2\alpha}\varepsilon^{2},
\end{align*}
for all $0<r<R'$, with $\Upsilon'$ defined in the obvious manner. Now, if $0\leq\Upsilon'\leq 1$, the previous estimate yields $|\Upsilon'|^{2}\leq C(\tfrac{r}{R'})^{2\alpha}$ whereas if $|\Upsilon'|>1$, we use $(\tfrac{r}{R'})^{2\alpha}\leq (\tfrac{r}{R'})^{\alpha}$ to infer that 
\begin{align*}
\frac{1}{r^{\alpha}}\dashint_{\ball(x,r)}|\mathscr{E}u-(\mathscr{E}u)_{\ball(x,r)}|\dif\mathscr{L}^{n} \leq \frac{C(n,L,\ell,\alpha,M_{0})}{(R')^{\alpha}}
\end{align*}
for all $0<r<R'$. Now, by the Campanato-Meyers characterisation of H\"{o}lder continuity, this implies that $\mathscr{E}u$ is of class $\hold^{0,\alpha}$ in $\ball(x_{0},R')$. Now we use $\hold^{0,\alpha}\simeq \mathcal{L}^{2,n+2\alpha}$ with the \textsc{Campanato} spaces $\mathscr{L}^{p,\lambda}$ (cf.~\cite[Thm.~2.9]{Giusti}) and Proposition~\ref{prop:OrliczKorn}\ref{item:OK3} with $\psi(t)=t^{2}$ to find that $\D\!u$ is of class $\hold^{0,\alpha}$ in $\ball(x_{0},R')$, too. 

Since $V(\cdot)\leq |\cdot|$ and with the above choices of $\varepsilon_{M},R_{0}$, \eqref{eq:thm:smallness} implies \eqref{eq:TabbyCenterEstimateCond}. Hence, to conclude the proof of the theorem, we note that by the Lebesgue differentiation theorem for Radon measures, for $\mathscr{L}^{n}$-a.e. $x_{0}\in\Omega$ there exists $z\in\rsym$ such that 
\begin{align}\label{eq:regsetcondSQC}
\lim_{r\searrow 0}\dashint_{\ball(x_{0},r)}|\mathscr{E}u-z|\dif\mathscr{L}^{n} + \frac{|\E^{s}u|(\ball(x_{0},r))}{\omega_{n}r^{n}}=0.
\end{align}
Thus, for such $x_{0}$, there exists $\delta>0$ with $\sup_{0<r<\delta}|(\E u)_{\ball(x_{0},r)}|=:M<\infty$. Applying the foregoing to the number $M$ and invoking \eqref{eq:regsetcondSQC}, the conclusion of the theorem follows immediately and the proof is complete.  
\end{proof}
We concude with the following 
\begin{remark}\label{rem:unifying}
In the $\bv$-case as considered by \textsc{Kristensen} and the author \cite{GK2}, different Fubini-type properties needed to be invoked to deal with $n=2$ and $n\geq 3$. Starting from the fact that for $\bv$-maps the tangential derivatives of $u$ on $\partial\!\ball(x_{0},t)$ for $\mathscr{L}^{1}$-a.e. $t>0$ are finite Radon measures themselves, the approach in \cite{GK2} is to embed $\bv(\partial\!\ball(x_{0},t);\R^{N})$ into higher fractional Sobolev spaces. If $n=2$, spheres are one-dimensional manifolds, and here Remark~\ref{rem:BVFubinibad} excludes the relevant embeddings. This forces to argue via Besov-Nikolski\u{\i} spaces in the full gradient, strongly quasiconvex case for $n=2$. However, the approach as outlined above for $\bd$ equally works in the easier $\bv$-situation, too, and thus yields a unifying method for all $n\geq 2$. 
\end{remark}
\subsection{Remarks and Extensions}\label{sec:extensions}
In this concluding section, we discuss some aspects, extensions and limitations of the results presented so far. 

We begin by noting that, under the assumptions of Theorem~\ref{thm:main1}, we can actually establish $\hold^{2,\alpha}$-partial regularity of generalised minima. Namely, letting $x_{0}\in\Omega_{u}$, we have $u\in\hold^{2,\alpha}(\ball(x_{0},r);\R^{n})$ for some $r>0$ and all $0<\alpha<1$. This is a consequence of Schauder estimates based on the $\hold^{1,\alpha}$-regularity of $u$ in a neighbourhood of $x_{0}$; choosing $|h|$ small enough, we consider the finite differences $\tau_{s,h}\sg(u)(x):=\sg(u)(x+he_{s})-\sg(u)(x)$, where $x$ belongs to a suitable neighbourhood of $x_{0}$ and $e_{s}$ is the $s$-th unit vector. We then set
\begin{align*}
\mathscr{Q}(x)[\xi,\eta] := \int_{0}^{1}\langle f''(\sg(u)(x) + t \tau_{s,h}\sg(u)(x))\xi,\eta\rangle \dif t,\qquad \xi,\eta\in\rsym. 
\end{align*}
By condition \ref{item:reg1} from Theorem~\ref{thm:main1} and the $\hold^{1,\alpha}$-partial regularity (for any $0<\alpha<1$) established above, $\mathscr{Q}\in\hold^{0,\alpha}(U;\mathbb{S}(\rsym))$ for some open neighbourhood $U$ of $x_{0}$ and any $0<\alpha<1$. Possibly diminishing $U$, we infer similarly as to \eqref{eq:lin1} that $\mathscr{Q}$ is uniformly  symmetric Legendre-Hadamard in $U$; i.e, there exists a constant $\lambda>0$ such that $\mathscr{Q}(x)[a\odot b,a\odot b]\geq \lambda |a\odot b|^{2}$ for all $a,b\in\R^{n}$ and all $x\in U$ together with $\sup_{x\in U}|\mathscr{Q}(x)|<\infty$. Working from here, it is not too difficult to establish an inequality of Garding type for some $r>0$ suitably small: There exist $\gamma_{1},\gamma_{2}>0$ such that there holds 
\begin{align*}
\int_{\ball(x_{0},r)}\mathscr{Q}(x)[\sg(\varphi),\sg(\varphi)]\dif x \geq \int_{\ball(x_{0},r)}\gamma_{1}|\sg(\varphi)|^{2}-\gamma_{2}|\varphi|^{2}\dif x\qquad\text{for all}\;\varphi\in\sobo_{0}^{1,\infty}(\ball(x_{0},r);\R^{n}). 
\end{align*}
On the other hand, since $u$ is a minimiser, we deduce the Euler-Lagrange equation 
\begin{align*}
\int_{\ball(x_{0},r)}\langle f'(\sg(u)),\sg(\varphi)\rangle\dif x = 0\qquad\text{for all}\;\varphi\in\ld_{0}(\ball(x_{0},r)). 
\end{align*}
At this stage, picking an arbitrary localisation function $\rho\in \hold_{c}^{\infty}(\ball(x_{0},r);[0,1])$, we may test the preceding equation with $\varphi=\tau_{s,-h}(\rho^{2}\tau_{s,h}u)$ for $|h|$ suitably small. Here $\tau_{s,-h}v(x):=v(x-he_{s})-v(x)$ denotes the backward finite difference. As a consequence, we obtain that 
\begin{align*}
\int_{\ball(x_{0},r)}\mathscr{Q}(x)[\tau_{s,h}\sg(u),\rho^{2}\tau_{s,h}\sg(u)+2\rho\nabla\rho\odot\tau_{s,h}u]\dif x= 0
\end{align*}
holds for any $s\in\{1,...,n\}$. Then it is routine to conclude by the above Garding-type inequality and Korn's inequality that $u$ is of class $\sobo^{2,2}$ in a suitable neighbourhood of $x_{0}$. Diminishing $r>0$ if necessary, we may then assume that 
\begin{align}
\int_{\Omega}\langle f''(\sg(u))\partial_{s}\sg(u),\sg(\varphi))\dif x = 0\qquad\text{for all}\;\varphi\in\hold_{c}^{1}(\ball(x_{0},r);\R^{n}). 
\end{align}
At this stage, we invoke a similar argument as in Proposition~\ref{thm:mazsha} to reduce to the \textsc{Schauder} estimates \cite[Ch.~III, Thm.~3.2]{Giaquinta} to conclude that $\partial_{s}\sg(u)$ is of class $\hold^{0,\alpha}$ in a fixed neighbourhood of $x_{0}$. Switching to the \textsc{Campanato}-\textsc{Meyers} characterisation of H\"{o}lder continuity and employing Proposition~\ref{prop:OrliczKorn}\ref{item:OK3}, $\partial_{s}\D\!u\in\hold^{0,\alpha}$ in an open neighbourhood of $x_{0}$ for any $s\in\{1,...,n\}$.

An analogous regularity theory can be set up when $x$-dependent integrands $f\colon\Omega\times\rsym\to\R$ are considered, and we refer the reader to the corresponding statements in \cite[Sec.~6]{GK2}; these follow in an analogous way once the regularity results from Theorem~\ref{thm:main1} are available. However, the case of \emph{fully} non-autonomous integrands $f\colon \Omega\times\R^{n}\times\rsym\ni (x,y,z)\mapsto f(x,y,z)$ comes along with two major difficulties. First, to the best of the author's knowledge, there is no integral representation of the relaxed functional available at present; the arguments of \textsc{Rindler} \cite{Ri1} do not seem to easily generalise to this situation. In contrast to \eqref{eq:minimality}, the definition of generalised minima then must be given directly by the Lebesgue-Serrin-type extension. To then access the Euler-Lagrange equations satisfied by the respective $\bd$-minimisers, it is necessary to employ a careful approximation procedure. This in principle being possible, we would still need a higher integrability result on the gradients of $\bd$-minima as it is usually required (cf.~\cite[Thm.~9.5 ff.]{Giusti}). In the quasiconvex, superlinear growth context, the latter is obtained as a consequence of the Caccioppoli inequality of the second kind in conjunction with the Sobolev inequality. In this situation, the Gehring lemma then boosts the so derived reverse H\"{o}lder inequality with increasing supports to the higher integrability of the gradients. In the linear growth situation, working from the Caccioppoli-type inequality strictly requires a \emph{sublinear Sobolev inequality}, the unconditional availability of which being ruled out by a counterexample due to \textsc{Buckley \& Koskela} \cite{BuckleyKoskela}. This is an important issue, as otherwise we would immediately obtain that $\bd$-minima belonged to $\sobo_{\locc}^{1,q}$ for some $q>1$, a fact which would simplify several stages of the above proof. A similar issue had been identified by \textsc{Anzellotti \& Giaquinta} \cite[Sec.~6]{AG} within the framework of convex full gradient functionals. However, note that if $f\colon\Omega\times\R^{n}\times\rsym\to\R$ satisfies a splitting condition $f(x,y,z)=f_{1}(x,z) + f_{2}(x,y)$ for some strongly symmetric quasiconvex integrand $f_{1}\colon\Omega\times\rsym\to\R$ of linear growth and $f_{2}\colon\Omega\times\R^{n}\to\R$ being convex and of at most $\frac{n}{n-1}$-growth in the second variable, then suitable regularity results \emph{can be formulated}. \textsc{Schmidt} \cite{SchmidtPR1} provides an interesting alternative of a partial regularity proof for convex, fully non-autonomous integrands of (super)quadratic growth that does not utilise Gehring's lemma. The drawback here is that does not seem to generalise easily to the quasiconvex situation with (super)linear growth; even if it would, it needed to be compatible with the above proof scheme. 

Lastly, let us address the possibility of giving Hausdorff dimension bounds on $\Sigma_{u}$ in the framework of Theorems~\ref{thm:ppartialregularity} or \ref{thm:main1}. Whereas in the convex context higher differentiability of minima can be invoked to establish $\dim_{\mathcal{H}}(\Sigma_{u})<n$ (also see \cite{Giusti,Mingione00,Mingione0} and \cite{CFI,G00,G1,GK1} in the symmetric gradient situation), such methods rely exclusively on the Euler-Lagrange system and thus do not apply to the (symmetric) quasiconvex situation. In this case, one option is to apply the set porosity approach employed by \textsc{Kristensen \& Mingione} \cite{KristensenMingione} for Lipschitzian minima. In the framework of Theorem~\ref{thm:ppartialregularity} for $p\geq 2$, this directly yields the dimension reduction for $\sobo^{1,\infty}$-regular minima. The method is likely to generalise to $1\leq p <\infty$; however, to the best of the author's knowledge, the method is not known to yield a dimension reduction for $\sobo^{1,s}$-regular minima even with  large $s>p\geq 1$; neither is it clear how to obtain  such a higher local gradient integrability in the (strongly) quasiconvex context beyond the usual \textsc{Gehring}-type improvement. Indeed, for $p>1$, \emph{some} higher gradient integrability follows by \textsc{Gehring}'s lemma in conjunction with the Caccioppoli inequality, but by the above discussion even this is unclear in the linear growth context. We intend to tackle questions of this sort in the future. 
\section{Appendix}\label{sec:appendix}
\subsection{Existence of $\bd$-minima}\label{sec:existence}
Implicitly used in the main part, we now briefly justify the existence of generalised minima for the Dirichlet problem \eqref{eq:varprin} in the sense of \eqref{eq:minimality}, now being subject to the strong symmetric quasiconvexity of $f\in\hold(\rsym)$, and gather some consequences. This program is somewhat analogous to \cite[Thm.~5.3]{BDG} where, however, a different coerciveness condition was employed. We hereafter let $u_{0}\in\ld(\Omega)$ be a given Dirichlet datum and $f\in\hold(\rsym)$ a strongly symmetric quasiconvex integrand satisfying both \eqref{eq:SSQC} and the linear growth assumption~\eqref{eq:lingrowth}. Our objective is to establish (with the notation of \eqref{eq:varprin} ff.)
\begin{align}\label{eq:nogap}
\inf_{\mathscr{D}_{u_{0}}}F = \min_{\bd(\Omega)}\overline{F}_{u_{0}}, 
\end{align}
particularly asserting the existence of $\bd$-minimisers. Toward the latter, we note that because $F$ is strongly symmetric quasiconvex, we have for all $\varphi\in\hold_{c}^{\infty}(\Omega;\R^{n})$
\begin{align*}
f(0)=f(0)-\ell V(0)\leq \dashint_{\Omega}f(\sg(\varphi))-\ell V(\sg(\varphi))\dif x,  
\end{align*}
as follows easily by passing from $Q=(0,1)^{n}$ to general open domains $\Omega$. Thus, by Lipschitz continuity of $f$, 
\begin{align}\label{eq:coerciveinequality1}
\begin{split}
f(0)\mathscr{L}^{n}(\Omega) +\ell \int_{\Omega}V(\sg(\varphi))\dif x & \leq \int_{\Omega}f(\sg(\varphi))-f(\sg(u_{0}+\varphi))\dif x + \int_{\Omega}f(\sg(u_{0}+\varphi))\dif x \\ 
& \leq c(L)\int_{\Omega}|\sg(u_{0})|\dif x + \int_{\Omega}f(\sg(u_{0}+\varphi))\dif x. 
\end{split}
\end{align}
At this stage, we pick an open and bounded Lipschitz set $\widetilde{\Omega}\subset\R^{n}$ such that $\Omega\Subset\widetilde{\Omega}$ and find, following the discussion in Section~\ref{sec:bd}, $\overline{u}_{0}\in\ld_{0}(\widetilde{\Omega})$ such that $\overline{u}_{0}|_{\Omega}=u_{0}$. We then put, for $v\in\bd(\Omega)$ 
\begin{align}\label{eq:defextension}
\widetilde{v}:=\begin{cases} v&\;\text{in}\;\Omega,\\ 
\overline{u}_{0} &\;\text{in}\;\widetilde{\Omega}\setminus\overline{\Omega}.
\end{cases} 
\end{align} 
Since $\partial\Omega$ is Lipschitz and $\overline{u}_{0}\in\ld(\widetilde{\Omega})$, $\widetilde{v}\in\bd(\widetilde{\Omega})$. Hence, we have for all $\varphi\in\hold_{c}^{\infty}(\Omega;\R^{n})$
\begin{align*}
\ell\int_{\widetilde{\Omega}}V(\sg(\overline{u}_{0}+\varphi))\dif x & - \ell \int_{\widetilde{\Omega}\setminus\Omega}V(\sg(\overline{u}_{0}))\dif x = \ell\int_{\Omega} V(\sg(u_{0}+\varphi))-V(\sg(\varphi))\dif x + \ell\int_{\Omega}V(\sg(\varphi))\dif x \\  
& \stackrel{\eqref{eq:coerciveinequality1}}{\leq}  C(\ell,L,V) \int_{\Omega}|\sg(u_{0})|  -f(0)\mathscr{L}^{n}(\Omega) + \int_{\Omega}f(\sg(u_{0}+\varphi))\dif x \\ 
&\stackrel{\eqref{eq:lingrowth}}{\leq} C(\ell,L,V,\mathscr{L}^{n}(\widetilde{\Omega}))\Big(\int_{\widetilde{\Omega}}|\sg(\overline{u}_{0})|\dif x +1 \Big) + \int_{\widetilde{\Omega}}f(\sg(\overline{u}_{0}+\varphi))\dif x.
\end{align*}
At this stage, let $v\in\bd(\Omega)$ be arbitrary and pick, due to Lemma~\ref{lem:smooth}, a sequence $(v_{j})\subset u_{0}+\hold_{c}^{\infty}(\Omega;\R^{n})$ such that $\widetilde{v}_{j}\to \widetilde{v}$ area-strictly on $\bd(\widetilde{\Omega})$. Since for every $j\in\mathbb{N}$, $\widetilde{v}_{j}$ is of the form $\overline{u}_{0}+\varphi_{j}$ with some $\varphi_{j}\in\hold_{c}^{\infty}(\Omega;\R^{n})$, we obtain 
\begin{align*}
\ell\int_{\widetilde{\Omega}}V(\sg(\widetilde{v}_{j}))\dif x  - \ell \int_{\widetilde{\Omega}\setminus\Omega}V(\sg(\overline{u}_{0}))\dif x  \leq  C(\ell,L,V,\mathscr{L}^{n}(\widetilde{\Omega}))\Big(\int_{\widetilde{\Omega}}|\sg(\overline{u}_{0})|\dif x +1 \Big) + \int_{\widetilde{\Omega}}f(\sg(\widetilde{v}_{j}))\dif x.
\end{align*}
Since $f$ is symmetric quasiconvex, it is symmetric rank-one convex in the sense as specified in Section~\ref{sec:functionsofmeasures}. Therefore, Lemma~\ref{lem:symareastrict} (which precisely yields continuity of the associated integral functionals for symmetric rank-one convex integrands) and the very definition of $V$ yield by passing $j\to\infty$
\begin{align*}
\ell\int_{\widetilde{\Omega}}V(\E\widetilde{v}) - \ell \int_{\widetilde{\Omega}\setminus\Omega}V(\sg(\overline{u}_{0}))\dif x  \leq  C(\ell,L,V,\mathscr{L}^{n}(\widetilde{\Omega}))\Big(\int_{\widetilde{\Omega}}|\sg(\overline{u}_{0})|\dif x +1 \Big) + \int_{\widetilde{\Omega}}f(\E\widetilde{v}).
\end{align*}
Enlarging the constant $C>0$ from the previous inequality, by definition of $\widetilde{v}$ we thereby obtain 
\begin{align}\label{eq:bddbelow}
\begin{split}
\ell\int_{\Omega}V(\E v) & + \ell\int_{\Omega}|\trace_{\partial\Omega}(u_{0}-v)\odot\nu_{\partial\Omega}|\dif\mathscr{H}^{n-1}-C(\ell,L,V,\mathscr{L}^{n}(\widetilde{\Omega}))\Big(\int_{\widetilde{\Omega}}|\sg(\overline{u}_{0})|\dif x +1 \Big) \\ & \stackrel{\eqref{eq:lingrowth}}{\leq} \int_{\Omega}f(\E v) + \int_{\partial\Omega}f^{\infty}\big(\trace_{\partial\Omega}(u_{0}-v)\odot \nu_{\partial\Omega}\big)\dif\mathscr{H}^{n-1} = \overline{F}_{u_{0}}[v]. 
\end{split}
\end{align}
Since $|\cdot|\leq V(\cdot)+1$, this proves that $\overline{F}_{u_{0}}$ is bounded below on $\bd(\Omega)$. Let $(u_{j})\subset\bd(\Omega)$ be a minimising sequence for $\overline{F}_{u_{0}}$, i.e., $\overline{F}_{u_{0}}[u_{j}]\to \inf_{\bd(\Omega)}\overline{F}_{u_{0}}$. Then $(u_{j})$ is bounded in $\bd(\Omega)$ and we may extract a non-relabeled subsequence and find some $u\in\bd(\Omega)$ such that $u_{j}\stackrel{*}{\rightharpoonup} u$ in $\bd(\Omega)$ (and hence $\widetilde{u}_{j}\stackrel{*}{\rightharpoonup}\widetilde{u}$ in $\bd(\widetilde{\Omega})$). Cancelling the integrals over $\widetilde{\Omega}\setminus\overline{\Omega}$ and as a consequence of Theorem~\ref{thm:rindler}, $\overline{F}_{u_{0}}[u] \leq \liminf_{j\to\infty}\overline{F}_{u_{0}}[u_{j}]=\inf_{\bd(\Omega)}\overline{F}_{u_{0}}$. Hence, $u$ is a $\bd$-minimiser in the sense of \eqref{eq:minimality}. 

We come to \eqref{eq:nogap}. Since there holds $\mathscr{D}_{u_{0}}\subset\bd(\Omega)$ and $\overline{F}_{u_{0}}|_{\mathscr{D}_{u_{0}}(\Omega)}=F$ on $\mathscr{D}_{u_{0}}$, we obtain '$\geq$' in \eqref{eq:nogap}. For the other direction, pick a $\bd$-minimiser $u\in\bd(\Omega)$ for $F$, its existence having been established above. Choosing an extension $\overline{u}_{0}$ of the Dirichlet datum $u_{0}$ as above and defining $\widetilde{u}$ via \eqref{eq:defextension}, we invoke Lemma~\ref{lem:smooth} to obtain a sequence $(u_{j})\subset u_{0}+\hold_{c}^{\infty}(\Omega;\R^{n})$ such that $\widetilde{u}_{j}\to \widetilde{u}$ area-strictly in $\bd(\widetilde{\Omega})$. Then, again by Lemma~\ref{lem:symareastrict}, $\overline{F}_{u_{0}}[u_{j}]\to \overline{F}_{u_{0}}[u]$ as $j\to\infty$. Thus, since $\widetilde{u}_{j}\in \mathscr{D}_{u_{0}}$ and $\overline{F}_{u_{0}}[u_{j}]=F[u_{j}]$ for all $j\in\mathbb{N}$, 
\begin{align*}
\inf_{\mathscr{D}_{u_{0}}}F \leq \lim_{j\to\infty}\overline{F}_{u_{0}}[u_{j}] = \overline{F}_{u_{0}}[u]=\min_{\bd(\Omega)}\overline{F}_{u_{0}}[u]. 
\end{align*}
Since we already established that $\min_{\bd(\Omega)}\overline{F}_{u_{0}}\leq \inf_{\mathscr{D}_{u_{0}}}F$, the proof of \eqref{eq:nogap} is complete. 
\subsection{Auxiliary Estimates on the $V_{p}$-functions}\label{sec:app1}
In this section we provide the proof of the auxiliary estimation \eqref{eq:psiaestimate} that helped to establish a particular form of a Korn-type inequality; recall that now $1<p<2$. The first uniform comparability assertion of \eqref{eq:psiaestimate} is a basic property of shifted $N$-functions, cf. \cite[Def.~2 and Sec.~2]{DLSV}. We thus begin by showing that $\psi$ given by \eqref{eq:psimaindef} satisfies the conditions of Lemma~\ref{lem:conjugate} together with the second uniform comparability of \eqref{eq:psiaestimate}. This means 
\begin{align}\label{eq:simeq1main}
\begin{split}
&c\psi'(t) \leq \psi''(t)t \leq C\psi'(t),\\
&c(1+a+t)^{p-2}t^{2} \leq \psi''(a+t)t^{2} \leq C (1+a+t)^{p-2}t^{2}
\end{split}
\end{align}
for some $0<c\leq C <\infty$ independent of $a\geq 0$ and $t>0$. We start with $\eqref{eq:simeq1main}_{2}$, and recall that $1<p<2$ throughout this section. To this end, note that for $t>0$
\begin{align}\label{eq:psidashcompute}
\begin{split}
&\frac{\dif}{\dif t} \psi(t) = (p-2)(1+t)^{p-3}t^{2} + 2(1+t)^{p-2}t,\\
&\frac{\dif^{2}}{\dif t^{2}}\psi(t) = (p-2)(p-3)(1+t)^{p-4}t^{2} + 4(p-2)(1+t)^{p-3}t + 2(1+t)^{p-2}.
\end{split}
\end{align}
Since $1<p<2$, the second term on the right-hand side of $\eqref{eq:psidashcompute}_{2}$ is negative. Therefore, 
\begin{align*}
\psi''(a+t)t^{2} & \leq (p-2)(p-3)(1+a+t)^{p-4}(a+t)^{2}t^{2} + 2(1+a+t)^{p-2}t^{2} \\
& \leq ((p-2)(p-3)+2)(1+a+t)^{p-2}t^{2}, 
\end{align*}
establishing the upper bound asserted by \eqref{eq:simeq1main}. The lower bound requires a refined argument. Since $p>1$, $c_{p}:=p^{2}-p$ is strictly positive. We write for $t>0$
\begin{align*}
\frac{\psi''(a+t)t^{2}}{(1+a+t)^{p-2}t^{2}} &  = \frac{ (p-2)(p-3)(1+a+t)^{p-4}(a+t)^{2}t^{2}}{(1+a+t)^{p-2}t^{2}} + \frac{4(p-2)(1+a+t)^{p-3}(t+a)t^{2}}{(1+a+t)^{p-2}t^{2}} \\ & + \frac{2(1+a+t)^{p-2}t^{2}}{(1+a+t)^{p-2}t^{2}} \\ 
& = 2 + (p-2)\Big[(p-3)\Big(\frac{a+t}{1+a+t}\Big)^{2} + 4\Big(\frac{a+t}{1+a+t}\Big)\Big]
\end{align*}
We claim that the ultimate term is larger or equal than $c_{p}$. Put $\vartheta\colon \R\ni z \mapsto 2 + (p-2)((p-3)z^{2}+4z)$. Since $1<p<2$, this function has a global minimum at $z_{0}=\frac{2}{3-p}$ which, by $p>1$, satisfies $z_{0}>1$. Hence, for all $z\in (0,1)$, $\vartheta(z)\geq\vartheta(1)=p^{2}-p=c_{p}>0$, and the lower bound of \eqref{eq:simeq1main} follows because of $(a+t)/(1+a+t)\in (0,1)$.

We turn to the third uniform comparability assertion of \eqref{eq:psiaestimate}, which is equivalent to the existence of constants $0<c\leq C <\infty$ such that 
\begin{align}\label{eq:simeq2main}
c(1+t^{2}+a^{2})^{\frac{p-2}{2}}t^{2} \leq (1+t+a)^{p-2}t^{2} \leq C(1+t^{2}+a^{2})^{\frac{p-2}{2}}t^{2}
\end{align}
holds for all $a,t\geq 0$. First note that 
\begin{align*}
\sqrt{1+t^{2}+a^{2}}\leq \sqrt{(1+t+a)^{2}}= 1+t+a
\end{align*}
so that, because of $1<p<2$, $(1+t+a)^{p-2}\leq (1+t^{2}+a^{2})^{\frac{p-2}{2}}$, 
and so the upper bound in \eqref{eq:simeq2main} follows. For the lower bound note that, because of Young's inequality 
\begin{align*}
1+t+a & = \sqrt{(1+t+a)^{2}} \leq \sqrt{8+8t^{2}+8a^{2}} \leq \sqrt{8}\sqrt{1+t^{2}+a^{2}}, 
\end{align*}
thereby establishing the lower bound in \eqref{eq:simeq2main}; for the latter estimate, we could have alternatively argued by virtue of Lemma~\ref{lem:auxVSQC}, cf.~$\eqref{eq:Vest1}_{1}$. We now turn to $\eqref{eq:simeq1main}_{1}$. Setting $a=0$ in $\eqref{eq:simeq1main}_{2}$, $\eqref{eq:simeq1main}_{1}$ is obviously equivalent to 
\begin{align}\label{eq:simeq3main}
\psi'(t) \simeq (1+t)^{p-2}t.
\end{align}
By $\eqref{eq:psidashcompute}_{1}$ and $1<p<2$, we  have $\psi'(t) \leq 2(1+t)^{p-2}t$ for all $t>0$. On the other hand, for $t>0$, 
\begin{align*}
\frac{\psi'(t)}{(1+t)^{p-2}t}=\frac{(p-2)(1+t)^{p-3}t^{2} + 2(1+t)^{p-2}t}{(1+t)^{p-2}t}=(p-2)\frac{t}{1+t}+2\geq p.
\end{align*}
The proof of \eqref{eq:psiaestimate} is complete. 

\begin{center}
\textbf{Declaration}
\end{center}
I hereby declare that there are no conflicts of interest and that the manuscript under consideration has not been submitted to any other journal. The work has been partially funded by the Euorpean Research Council EPSRC and the Hausdorff Center in Mathematics, Bonn. 
\end{document}